\documentclass[aop, noinfoline]{imsart}         
\usepackage{amsmath,amsfonts,amsthm,amssymb}    
\usepackage[pageanchor]{hyperref}               
\usepackage{bbm}                                
\usepackage{mathptmx} 
\usepackage[mathcal]{euscript}                  
\usepackage{enumitem}                           
\usepackage{mathtools}                          
\usepackage[top=1.15in, bottom=1.15in, left=1.25in, right=1.25in]{geometry}             
\usepackage{subcaption}
\usepackage{tikz}
\usetikzlibrary{patterns,shapes}

\usepackage{datetime}                            
\startlocaldefs
\def\journal@id{~}
\def\journal@name{~}
\def\journal@url{~}
\endlocaldefs

\usepackage[refpage]{nomencl}                     
\makenomenclature                                 
\usepackage{multicol}
\makeatletter
\@ifundefined{chapter}
  {\def\wilh@nomsection{section}}
  {\def\wilh@nomsection{chapter}}
\def\thenomenclature{%
  \begin{multicols}{2}[%
\setlength\columnsep{1cm}
    \csname\wilh@nomsection\endcsname*{\nomname}
    \if@intoc\addcontentsline{toc}{\wilh@nomsection}{\nomname}\fi
    \nompreamble]
  \list{}{%
    \labelwidth\nom@tempdim
    \leftmargin\labelwidth
    \advance\leftmargin\labelsep
    \itemsep\nomitemsep
    \let\makelabel\nomlabel}%
}
\def\endthenomenclature{%
  \endlist
  \end{multicols}
  \nompostamble}
\makeatother
\renewcommand*{\nompreamble}{\footnotesize}
\newcommand{\nom}{\nomenclature}                  
\renewcommand{\nomname}{List of notation}         
\makeatletter \newcommand \Dotfill {\leavevmode \leaders \hb@xt@ 0.5em{\hss .\hss }\hfill \kern \z@} \makeatother    

\newtheorem{theorem}{Theorem}[section]
\newtheorem{conjecture}[theorem]{Conjecture}
\newtheorem{lemma}[theorem]{Lemma}
\newtheorem{proposition}[theorem]{Proposition}

\theoremstyle{definition}
\newtheorem{definition}[theorem]{Definition}
\theoremstyle{remark}
\newtheorem{remark}[theorem]{Remark}

\numberwithin{equation}{section}

\def\R{{\mathbb R}}
\def\N{{\mathbb N}}

\def\1{{\mathbbm 1}}
\def\E{{\mathbb E}}
\def\P{{\mathbb P}}
\newcommand{\nrm}[1]{\Vert #1 \Vert}
\newcommand{\lpnrm}[1]{\nrm{#1}_{n,p}}
\newcommand{\ltwonrm}[1]{\nrm{#1}_{n,2}}
\def\ra{\rightarrow}
\def\eqdist{\stackrel{(d)}{=}}
\renewcommand{\|}{\Vert}

\def\lp{\ell^p}
\def\bnp{{\mathbb B}_{n,p}}

\def\sphn{\mathbb{S}^{n-1}}
\def\seq{\mathbb{S}}

\def\sx{X}
\def\xnp{\sx^{(n,p)}}

\def\bxn{{ \sx}^{(n,2)}}
\def\xng{\sx^{(n,\gamma)}}

\def\thn{\theta^{(n)}}
\def\bth{\Theta}
\def\bthn{\Theta^{(n)}}

\def\sw{W}
\def\wnpth{\sw_\theta^{(n,p)}}
\def\wnpbth{\sw^{(n,p)}}

\def\wnpz{\widehat{W}^{(n,p)}}
\def\twnpz{\widetilde{W}^{(n,p)}}

\def\sy{\mathsf{Y}}
\def\ynp{Y^{(n,p)}}
\def\boly{\sy}
\def\bolyp{\sy^{(p)}}
\def\sz{\mathsf{Z}}
\def\zn{Z^{(n)}}
\def\bolz{\sz}

\def\sfq{\mathsf{qu}}
\def\sfa{\mathsf{an}}
\def\iq{\mathbb{I}^\sfq}
\def\ia{\mathbb{I}^\sfa}
\def\jc{\mathbb{I}^{\mathsf{cr}}}
\def\lm{\Lambda} 
\def\lmq{\Psi}
\def\lma{\Phi}
\def\lmbar{\widetilde\Phi}

\def\onen{\iota^{(n)}}

\begin{document}

\begin{frontmatter}

\title{Large deviations for random projections of \lowercase{$\ell^p$} balls}
\runtitle{LDP\MakeLowercase{s} for random projections }

\begin{aug}
  \author{\fnms{Nina}  \snm{Gantert}\thanksref{t1} \ead[label=e1]{gantert@ma.tum.de}},
  \author{\fnms{Steven Soojin} \snm{Kim} \thanksref{t2,t4} \ead[label=e2]{steven\_kim@brown.edu}},
  \and
  \author{\fnms{Kavita}  \snm{Ramanan} \thanksref{t3,t4} \ead[label=e3]{kavita\_ramanan@brown.edu}}

  \thankstext{t1}{NG and KR would like to thank ICERM, Providence, for an invitation to the program ``Computational Challenges in Probability", where some of this work was initiated.}
  \thankstext{t2}{SSK was partially supported by a Department of Defense NDSEG fellowship.}  
  \thankstext{t3}{KR was partially supported by ARO grant W911NF-12-1-0222 and NSF grant DMS 1407504.}
  \thankstext{t4}{SSK and KR would also like to thank Microsoft Research New England for their hospitality during the Fall of 2014, when some of this work was carried out.}

  \runauthor{Nina Gantert, Steven Soojin Kim, and Kavita Ramanan}

\smallskip

  \address{Fakult\"at f\"ur Mathematik,\\
Technische Universit\"at M\"unchen\\
          \printead{e1}}

  \address{Division of Applied Mathematics,\\
          Brown University\\
          \printead{e2,e3}}
          
\end{aug}

\begin{abstract} 
Let $p\in[1,\infty]$. Consider the projection of a uniform random vector from a suitably normalized $\ell^p$ ball in $\mathbb{R}^n$ onto an independent random vector from the unit sphere. We show that sequences of such random projections, when suitably normalized, satisfy a large deviation principle (LDP) as the dimension $n$ goes to $\infty$, which can be viewed as an \emph{annealed} LDP. We also establish a \emph{quenched} LDP (conditioned on a fixed sequence of projection directions) and show that  for $p\in(1,\infty]$ (but \emph{not} for $p=1$), the corresponding rate function is ``universal", in the sense that it coincides for ``almost every" sequence of projection directions. We also analyze some  exceptional sequences of directions in the ``measure zero" set, including the directions corresponding to the classical Cram\'er's theorem, and show that those directions yield LDPs with rate functions that are distinct from the universal rate function of the quenched LDP. Lastly, we identify a variational formula that relates the annealed and quenched LDPs, and analyze the minimizer of this variational formula. These large deviation results complement the central limit theorem for convex sets, specialized to the case of sequences of $\ell^p$ balls.  
\end{abstract}

\begin{keyword}[class=MSC]
\kwd[Primary ]{60F10} 
\kwd{52A23}  
\kwd[; secondary ]{52A20} 
\kwd{60K37} 
\kwd{60D05.} 
\end{keyword}

\begin{keyword}
\kwd{large deviations}
\kwd{random projections}
\kwd{$\ell^p$-balls}
\kwd{annealed and quenched large deviations}
\kwd{self-normalized}
\kwd{central limit theorem for convex sets}
\kwd{variational formula}
\end{keyword}

\end{frontmatter}

\vspace*{-8mm}

\setcounter{tocdepth}{1}
\tableofcontents


\section{Introduction}\label{sec-intro} 

Consider the projection of a random $n$-dimensional vector $X^{(n)}$  onto some lower dimensional subspace. Our broad goal is to understand and analyze distributional properties of the projections of high-dimensional random vectors (i.e., large $n$), given certain natural assumptions on the law of $X^{(n)}$. In this paper, we focus on projections onto one-dimensional subspaces,  and we write the \emph{projection}  of $x\in \R^n$  onto the direction $v\in \sphn$  (the unit sphere in $\mathbb{R}^n$), to refer to $\langle x,v\rangle_n \doteq \sum_{i=1}^n x_i v_i \in \R$; this is for the sake of brevity, since to be precise, the preceding quantity is  the scalar component of the projected vector $\langle x,v\rangle_nv\in \R^n$.
\nom[0]{$\langle \cdot, \cdot\rangle_n$}{Euclidean inner product on $\R^n$}

One prior result in this vein is the central limit theorem (CLT) for projections of convex bodies: if $X^{(n)}$ is sampled from a log-concave measure (e.g., the uniform measure on a convex body) that is also isotropic, then for sufficiently large $n$, and ``most" $\thn\in\sphn$, the projection of $X^{(n)}$ onto $\thn$ satisfies $  \langle X^{(n)}, \thn \rangle_n \approx N(0,1)$ in some suitable sense. This result is established via a concentration estimate in \cite{klartag2007central}, drawing from a classical idea of \cite{diaconis1984asymptotics, sudakov1978typical} and a conjecture stated in \cite{anttila2003central, brehm2000asymptotics}. Similar central limit results hold for directions of projection $\bthn$ drawn from the unique rotationally invariant measure on $\sphn$, and for projections onto $k$-dimensional subspaces, for $1 \le  k \ll \sqrt{\log n}$ \cite{meckes2012approximation, meckes2012projections}. In this class of results, the source of the Gaussian approximation may be attributed to geometric properties (specifically, log-concavity) of the original measure.

It is natural to ask if existing CLT results for typical projections of high-dimensional random vectors from a convex set can be complemented by analyzing deviations beyond the central limit fluctuation scale. In this work, we initiate such an analysis by investigating \emph{large deviation principles (LDPs)} for sequences of random one-dimensional projections of a certain class of convex bodies, the so-called $\ell^p$ balls.

One of our motivations for investigating LDPs is to understand which aspects of random projections can be used to distinguish between different convex bodies. From a central limit perspective, convex bodies cannot be distinguished by their random projections; that is, given \emph{any} isotropic convex body in high dimension, its typical random projections will be approximately standard Gaussian. In fact, Gaussian asymptotics arise not only at the ``central limit" scale, but also across the ``moderate deviation" scale \cite{sodin2007tail}. These universal results are quite powerful, but from another point of view, it is also of interest to precisely identify how a random projection can encode distinct distributional information about the original vector. Our results demonstrate that the {large deviation} behavior of a random projection of a convex body depends on the geometry of the underlying convex body. In particular, we demonstrate sharply different large deviation behavior for random projections of $\ell^p$ balls for different values of $p\in[1,\infty)$. That is, compare Theorem \ref{th-aldp} for $p\in[2,\infty)$ against Theorem \ref{th-aldp12}  for $p\in[1,2)$ (which we comment on further in Remark \ref{rmk-p12}), and also compare Theorem \ref{th-qldp} against Theorem \ref{th-qldp1}, where the anomalous LDP is for $p=1$, the only $\ell^p$ ball  for $p\in[1,\infty)$ with a non-smooth ``corner".

Unlike for central limit theorems, where one can quantify the closeness of a random vector in a fixed high-dimensional space $\R^n$ to the $n$-dimensional Gaussian using various metrics, the statement of an LDP for random projections requires an infinite sequence of convex bodies defined for all dimensions $n\in\N$. This motivates our analysis of the uniform measures on $\ell^p$ balls, which offer a natural, fundamental, yet non-trivial (in particular, non-product) example of a \emph{sequence} of isotropic, log-concave measures. Specifically, we address the following question:

\begin{quote}\normalsize\itshape
Do LDPs hold for (suitably normalized) random projections of vectors uniformly distributed on the $\ell^p$ ball of $\R^n$? If so, then at which speed and with what rate function? Moreover, how do these LDPs vary with $p\in[1,\infty)$?
\end{quote}

These questions have the flavor of the study of LDPs in random environments (see, e.g., \cite{comets2000quenched}), where in our case the random ``environment" is governed by the random sequence of projection directions. In this setting, it is natural to consider both the case when one conditions on a fixed sequence of random projection directions (the so-called \emph{quenched} case) and also when one incorporates the randomness of the projection directions (the so-called \emph{annealed} case).

Our main results on this question are the following:
\begin{enumerate}[leftmargin=9em, rightmargin=2em, align=right, labelwidth=6em, itemsep=0.5em] 
\item [Theorems \ref{th-aldp} \& \ref{th-aldp12}:]   annealed LDPs,  for $p\in[2,\infty)$ and $p\in[1,2)$, respectively.
\item [Theorems \ref{th-qldp} \& \ref{th-qldp1}:] quenched LDPs, for $p\in(1,\infty)$ and $p=1$, respectively. Moreover, for $p\in(1,\infty)$, but \emph{not} for $p=1$, this LDP holds with a ``universal" rate function that coincides for ``almost every" sequence of directions.
\item [Theorem \ref{th-compar}:] for $p\in(2,\infty)$, a variational formula that relates the annealed and quenched rate functions via the entropy of an underlying measure.  
\item [Theorem \ref{th-atyp}:] a proof of the observation that the particular sequence of directions $(\iota^{(n)})_{n\in\N}$ defined in \eqref{onedefn} below (which corresponds to Cram\'er's theorem in the case of product measures) leads to an ``atypical" large deviation rate function.
\end{enumerate}

Observe that in the preceding summary of our main results (stated precisely in Sect. \ref{sec-main} and proved in Sects. \ref{sec-annealed}--\ref{sec-atyp}), we only discuss $p< \infty$, and omit the case $p=\infty$.  However, all of our results have corresponding versions for general product measures satisfying certain tail conditions (including the uniform measure on the $\ell^\infty$ ball),  in fact with simpler proofs than in the non-product ($p<\infty$) case. We compile all of the corresponding statements for product measures in Sect. \ref{sec-prod}, where we also provide brief sketches of the proofs.

We make the distinction between $p<\infty$ and $p=\infty$ because a secondary motivation for our work is to investigate to what extent large deviation results extend beyond the classical setting of sums of independent and identically distributed (i.i.d.) random variables to the more general setting of  generic projections of log-concave measures. More precisely, let $X^{(n,p)}$ be distributed uniformly on the $\ell^p$ ball of $\R^n$. Consider the direction  $\iota^{(n)}\in\sphn$ defined by
\begin{equation}\label{onedefn}
  \iota^{(n)} \doteq \tfrac{1}{\sqrt{n}}(\underbracket{1,1,\dots, 1}_{n \text{ times }}) \in \sphn.
\nom[iota]{$\iota^{(n)}$}{direction $\tfrac{1}{\sqrt{n}}(1,\dots,1) \in \sphn$}
\end{equation}
The classical Cram\'er's theorem yields an LDP for the sequence of suitably normalized projections $n^{-1/2} \langle X^{(n,\infty)}, \, \iota^{(n)}\rangle$, $n\in \N$. In contrast, our work establishes an LDP for $n^{-1/2} \langle X^{(n,p)}, \thn \rangle$, $n\in \N$, for $p\in[1,\infty)$ and general $\thn\in\sphn$. Figure \ref{fig-projwhole} illustrates our setup. 

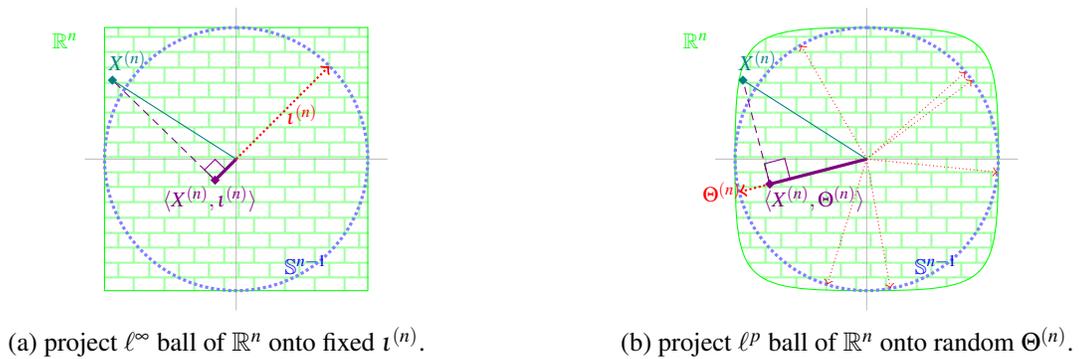
\begin{figure}[bht]
\begin{subfigure}{0.45\textwidth}
\centering
	\begin{tikzpicture}[scale=1.75]
	\node[green, scale = 0.90] at (-1.3,0.9) {$\R^n$};
\draw [gray!50]  (0,-1.15) -- (0,1.15) ;
\draw [gray!50]  (-1.15,0) -- (1.15,0) ;
    \draw [color=green, fill opacity=0.35, fill=green!25, pattern=bricks, pattern color = green] (-1,-1) -- (-1, 1) -- (1, 1) -- (1, -1) -- (-1, -1);
    
\draw [blue!50, very thick, densely dotted] plot [smooth cycle, tension=1] coordinates {(-1,0) (0,1) (1,0) (0,-1)};
    \node[blue, scale=0.90] at (0.53,-0.8) {$\sphn$};

    \draw [red, thick, densely dotted, ->] (0, 0) -- (0.70710,0.70710) ;
    \node[scale=0.90, red] at (0.5,0.34) { $\onen$};
    
    \draw [violet, densely dashed] (-0.16,-0.16) -- (-0.94,0.60);
    \draw [violet,very thick] (0,0) -- (-0.16,-0.16);
    \draw [violet](-0.16 ,-0.16 ) -- (-0.16 - 0.08,-0.16 + 0.08*38/39) --(-0.08 -0.08, -0.08 + 0.08*38/39)  --  (-0.08,-0.08) ;
    \node[fill=violet, thin, diamond, scale=0.3] at (-0.16,-0.16) { };
    \node[scale=0.90, violet] at (-0.2, -0.3) { $\langle X^{(n)}, \onen \rangle$ };
    
    \node[fill=teal, thin, diamond, scale=0.3] at (-0.94,0.60) { };
    \node[scale=0.90, teal] at (-0.825,0.75) { $X^{(n)}$};
    \draw [teal] (0,0) -- (-0.94,0.60);
	\end{tikzpicture}
 \caption{project $\ell^\infty$ ball of $\R^n$ onto fixed $\onen$.}
 \label{fig-projo}
\end{subfigure} \hfill
\begin{subfigure}{0.45\textwidth}
\centering
	\begin{tikzpicture}[scale=1.75]
	\node[green, scale = 0.90] at (-1.3,0.9) {$\R^n$};
\draw [gray!50]  (0,-1.15) -- (0,1.15) ;
\draw [gray!50]  (-1.15,0) -- (1.15,0) ;
    \draw [color=green, fill opacity=0.35, fill=green!25, pattern=bricks, pattern color = green] plot [smooth cycle, tension=1.7] coordinates {(-1,0) (0,1) (1,0) (0,-1)};
    
\draw [blue!50, very thick, densely dotted] plot [smooth cycle, tension=1] coordinates {(-1,0) (0,1) (1,0) (0,-1)};
    \node[blue, scale=0.90] at (0.53,-0.8) {$\sphn$};

    \draw [red, thick, densely dotted, ->] (0, 0) -- (-0.96824,-1/4) ;
    \node[scale=0.90, red] at (-1.1,-1/4) { $\bthn$};
    \draw [red, opacity=0.85 ,densely dotted, ->] (0, 0) -- (0.8, 0.6) ;
    \draw [red,opacity=0.85 ,densely dotted,  ->] (0, 0) -- (-0.3, -0.9539) ;
    \draw [red,opacity=0.85 ,densely dotted,  ->] (0, 0) -- (-0.5, 0.866) ;
    \draw [red,opacity=0.85 ,densely dotted,  ->] (0, 0) -- (0.75, 0.6614) ;
    \draw [red,opacity=0.85 ,densely dotted,  ->] (0, 0) -- (0.9949, -0.1) ;
    \draw [red,opacity=0.85 , densely dotted, ->] (0, 0) -- (+0.18,-0.9836) ;
    
    \draw [violet, densely dashed] (-0.736,-0.19) -- (-0.94,0.60);
    \draw [violet,very thick] (0,0) -- (-0.736,-0.19);
    \draw [violet] (-0.736, -0.19) -- (-0.736 -0.04 ,-0.19 + 3.8725*0.04) -- ( -0.736 -0.04 + 0.16*0.96, -0.19 + 3.8725*0.04 + 0.16/4 ) --  (-0.736 +0.16*0.96, -0.19+0.16/4) ;
    \node[fill=violet, thin, diamond, scale=0.3] at (-0.736,-0.19) { };
    \node[scale=0.90, violet] at (-0.4, -0.3) { $\langle X^{(n)}, \bthn \rangle$ };
    
    \node[fill=teal, thin, diamond, scale=0.3] at (-0.94,0.60) { };
    \node[scale=0.90, teal] at (-0.825,0.75) { $X^{(n)}$};
    \draw [teal] (0,0) -- (-0.94,0.60);
	\end{tikzpicture}
 \caption{project $\ell^p$ ball of $\R^n$ onto random $\bthn$.}
 \label{fig-projr}
\end{subfigure}
\caption{Projection of $X^{(n)}$ onto an element of $\sphn$.}
\label{fig-projwhole}
\end{figure}

Note that in the central limit setting, projections onto general $\thn\in\sphn$ exhibit the same properties (Gaussian fluctuations) as projections onto the specific direction $\iota^{(n)}\in\sphn$; in the large deviation setting,  Theorem \ref{th-atyp} indicates that this is not the case. We elaborate on this in Sect.\ \ref{ssec-atyp}.

Lastly,  we are also interested in LDPs because they can yield not only the asymptotic likelihood of a rare event, but also insight into \emph{how} a rare event occurs. In particular, large deviation analysis typically yields variational formulas whose minimizer(s) admit a probabilistic interpretation. This paper initiates an investigation of a particular kind of ``geometric" rare event (large value of a projection), and we establish an associated variational formula in Theorem \ref{th-compar}. We also provide some analysis of a simpler variational formula (for the case $p=\infty$) in Sect. \ref{sec-analysis}.

It would be interesting to extend this large deviation analysis to more general sequences of probability measures beyond the uniform measures on $\ell^p$ balls. An even broader goal is to determine precisely which geometric aspects of the underlying probability measures affect the large deviation behavior of random projections. We defer these questions for future work.

The outline of this paper is as follows. In the remainder of Sect.\ \ref{sec-intro}, we review related work and set up the preliminaries for our own results.  In Sect.\ \ref{sec-main}, we precisely state our main results and provide explicit formulas for the case of $p=2$. In Sect.\ \ref{sec-equiv}, we appeal to certain probabilistic representations of the $\ell^p$ balls which simplify our analysis.  Sects.\ \ref{sec-annealed}--\ref{sec-atyp} contain the proofs of our results. In Sect. \ref{sec-prod}, we discuss analogous results for product measures. Lastly, in Sect.\ \ref{sec-analysis}, we analyze the variational problem established in Theorem \ref{th-compar}. A list of notation can be found on page \pageref{notation}.

\subsection{Relation to prior work}\label{ssec-prior}

Random projections of high-dimensional random vectors arise in a variety of applications. In the statistics and machine learning literature, projections onto random lower-dimensional subspaces are employed for the purposes of dimensionality reduction \cite{bingham2001random, lin2003dimensionality}, clustering \cite{fern2003random}, regression \cite{maillard2012linear}, and topic discovery \cite{ding2013topic} in the setting of high-dimensional data. The main idea is that a practitioner would like to restrict statistical analysis to a low-dimensional space, but it may be computationally expensive to try to select an ``optimal" subspace (using, e.g., Principal Component Analysis), and that under certain assumptions, selecting a random subspace may perform ``nearly" as well. 

On a more theoretical side, there is significant interest in  $\ell^p$ balls due to their central role in convex geometry. As a small fraction of the extensive literature, we note results on sections \cite{meyer1988sections}, hyperplanes \cite{barthe2002hyperplane}, extremal slabs \cite{barthe2003extremal}, probabilistic representations \cite{barthe2005probabilistic}, and cone/surface measures \cite{naor2003projecting, naor2007surface, kr1}. The $\ell^p$ balls also arise in computer science in the context of sketches and low-distortion embeddings \cite{indyk2000stable}. An LDP for projections of $\ell^p$ balls onto canonical basis directions can be found in \cite{BarGamLozRou10}, but our work provides the first LDPs for projections onto general $\thn\in\sphn$ and random $\bthn$.

In Sect.\ \ref{sec-equiv}, we show how our LDPs are related to LDPs for weighted sums of certain i.i.d.\ random variables. For a partial survey of large deviation results in the setting of \emph{deterministically} weighted sums of i.i.d.\ random variables, we refer to  \cite[\S2.1]{gkr3}, which details the arguments for quenched LDPs for projections of product measures.

Also in Sect.\ \ref{sec-equiv}, it becomes apparent that our LDP is related to LDPs for \emph{self-normalized} sums, as developed in \cite{Shao97}. A similar question as our quenched LDP (Theorem \ref{th-qldp}) in the case $p=2$  can be found in \cite[\S3.3, Ex. 5,6]{de2009self}, but without specifying the form of the rate function (which can be found in Sect.\ \ref{ssec-p2}). We discuss these connections to self-normalized sums in greater detail in Remark\ \ref{rmk-shao}.

\subsection{Setup and notation} \label{ssec-setup}

Let $\mathbb{A} \doteq \prod_{n\in \N} \R^n$ denote the space of infinite triangular arrays. That is, $z\in\mathbb{A}$ if $z=(z^{(1)},z^{(2)},\dots)$ where $z^{(n)} \in \R^n$ for all $n\in \N$. 
\nom[aaaaa]{$\mathbb{A}$}{infinite triangular arrays}

We assume that all random variables  are defined on a common probability space $(\Omega, {\mathcal F}, \P)$, and let $\E$ denote the corresponding expectation. Let $\mathcal{X}$ be some measurable space, and let $\mathcal{P}(\mathcal{X})$ be the space of probability measures on $\mathcal{X}$. For a random variable $\xi:\Omega\ra \mathcal{X}$, and a measure $\mu\in\mathcal{P}(\mathcal{X})$, we write $\xi \sim \mu$ if the law of $\xi$ is $\mu$; that is, if $\mathbb{P}\circ \xi^{-1} = \mu$. 
\nom[px]{$\mathcal{P}(\mathcal{X})$}{probability measures on $\mathcal{X}$}

For  $p \in [1,\infty)$, $n \in \N$, and $x \in \R^n$,   let $\nrm{x}_{n,p} \doteq \left(\sum_{i=1}^n |x_i|^p \right)^{1/p}$    denote the $\lp$ norm on $\R^n$. For $p=\infty$, let $\nrm{x}_{n,\infty} \doteq \sup_{i \in \{1, \ldots, n\}} |x_i|$ denote the $\ell^{\infty}$ norm on $\R^n$. Let $\bnp$  be the unit $\lp$ ball in $\R^n$: 
\[  \bnp \doteq  \left\{  x \in \R^n: \lpnrm{x} \leq 1 \right\}, \quad p \in [1,\infty],  \]
and let $\xnp =(\xnp_1,\dots, \xnp_n)$  be a random vector that is  distributed according to the uniform probability measure on  $\bnp$. Whenever we define a probability measure on a subset $A\subset \R^n$, we mean a probability measure on the Borel subsets of $A$.

\nom[a]{$\lVert \cdot \rVert_{n,p}$}{$\ell^p$ norm of $\R^n$} 
\nom[bnp]{$\bnp$}{unit $\ell^p$ ball of $\R^n$}
\nom[xnp]{$\xnp$}{uniform point from $\bnp$}

Let $\sphn$ denote the unit sphere in $\R^n$:  
\[   \sphn \doteq \left\{ x \in \R^n:  x_1^2 + x_2^2 + \dots + x_n^2  = 1 \right\} =  \left\{ x \in \R^n:  \ltwonrm{x} = 1 \right\}, \]
We write  $\sigma_{n}$ for the unique rotationally invariant probability measure on $\sphn$. For $n\in \N$, let $\bthn$ denote a random vector that is distributed according to the uniform measure $\sigma_n$ on $\sphn$, independent of $\xnp$.
\nom[sn]{$\sphn$}{unit sphere in $\R^n$}
\nom[szsig]{$\sigma_n$}{rotation inv.\ measure on $\sphn$ }
\nom[theta1]{$\bthn$, $\thn$}{random / fixed direction in $\sphn$}

For background on large deviations, we refer to \cite{DemZeibook}. In particular, recall the definition of large deviation principles:
 
\begin{definition}
Let  $\Sigma$ be a topological space. A sequence of $\Sigma$-valued random variables $(\xi_n)_{n\in \N}$ is said to satisfy a \emph{large deviation principle (LDP)} with \emph{speed} $s:\N\ra \R$ and a \emph{rate function} $I:\Sigma\ra[0,\infty]$ if $I$ is lower semi-continuous, and for all Borel measurable subsets $\Gamma\subset \Sigma$, 
\begin{equation*}
  -\inf_{x\in \Gamma^\circ} I(x) \le  \liminf_{n\ra\infty} \tfrac{1}{s(n)} \log \P(\xi_n\in \Gamma^\circ) \le \limsup_{n\ra\infty} \tfrac{1}{s(n)} \log \P(\xi_n\in \bar \Gamma) \le -\inf_{x\in \bar{\Gamma}} I(x),
\end{equation*}
where $\Gamma^\circ$ and $\bar\Gamma$ denote the interior and closure of $\Gamma$, respectively. Furthermore, $I$ is said to be a \emph{good rate function} if it has compact level sets. When no speed is explicitly stated, we take the convention that the default speed is $s(n) = n$.
\end{definition}

In the large deviation setting, we are frequently interested in geometric properties of an LDP rate function, such as convexity, or the following weakened form of convexity.

\begin{definition}\label{def-quasi}
A function $f:\R\ra(-\infty,+\infty]$ is said to be \emph{quasiconvex} if its level sets $\{x \in \R : f(x) \le c\}$ are convex for all $c\in \R$.
\end{definition}

Practically, this definition is useful because quasiconvex functions have an equivalent characterization: $f$ is quasiconvex if and only if there exists some $x_0\in \R$ such that $f$ is non-increasing for $x < x_0$ and non-decreasing for $x > x_0$. A general review of quasiconvex functions can be found, for example, in \cite[\S 3.4]{boyd2004convex}. As a further link to convexity, we recall the following transform which arises in Cram\'er's theorem, and will also play a role in our results.

\begin{definition}
Given a function $\Lambda:\mathbb{R}^n\ra (-\infty,+\infty]$, the \emph{Legendre transform} of $\Lambda$ is the function $\Lambda^*: \mathbb{R}^n \ra (-\infty,+\infty]$ defined by
\begin{equation*}
  \Lambda^*(\tau) \doteq  \sup_{t \in \mathbb{R}^n} \{ \langle t, \tau\rangle_n - \Lambda(t) \}, \quad \tau\in \mathbb{R}^n.
\nom[aa]{$(\,\cdot\,)^*$}{Legendre transform}
\end{equation*}
\end{definition}

We also define a class of measures which are intimately tied to the $\ell^p$ balls, as we will demonstrate in Sect.\ \ref{sec-equiv}. For $p\in [1,\infty)$, let $\mu_p \in \mathcal{P}(\mathbb{R})$ have density $f_p$, where
\begin{equation}\label{fpdef}
f_p(y) \doteq  \frac{1}{2p^{1/p}\Gamma(1+\frac{1}{p})} e^{-|y|^p/p}, \quad y\in \R.
\end{equation}
This is the density of the \emph{generalized normal distribution} (also known as the \emph{exponential power distribution}) with location 0, scale $p^{1/p}$, and shape $p$. When $p=2$, $\mu_2$ corresponds to the standard Gaussian distribution. 
\nom[fp]{$f_p$}{generalized normal density}
\nom[mzup]{$\mu_p$}{generalized normal distribution}

\section{Main results}\label{sec-main} 

In Sects. \ref{ssec-ann}--\ref{ssec-p2}, we precisely state our main results.

\subsection{Annealed LDP}\label{ssec-ann}

 Let $\wnpbth$ be the normalized (scalar) projection of $\xnp$ onto a random direction $\bthn$, defined as
\begin{equation} \label{def-bwnpbth}
\wnpbth \doteq   \frac{n^{1/p}}{n^{1/2}}\langle  \xnp, \bthn   \rangle_n  = \frac{1}{n} \sum_{i=1}^n (n^{1/p}\xnp_i) (n^{1/2}\bthn_i), \quad n \in \N,
\end{equation}
where for $p=\infty$, we abide by the convention  $n^{1/\infty} \equiv 1$. Our first result establishes an LDP for $(\wnpbth)_{n\in\N}$.
\nom[wnpth]{$\wnpbth$}{$n^{(1/p)-(1/2)} \langle \xnp, \bthn\rangle_n$}

\begin{remark}\label{rmk-scaling}
As Theorem \ref{th-aldp} and Theorem \ref{th-qldp} below show, the scaling $n^{(1/p)-(1/2)}$ in \eqref{def-bwnpbth} --- and also later in \eqref{def-wthetan} --- turns out to be appropriate for large deviation analysis. The heuristic reasoning behind this scaling is that the variance of $\wnpbth$ should be of ``order $1/n$" in order to prove non-trivial large deviation principles. To this end, note that both $n^{1/p}\xnp_i$ and $n^{1/2}\bthn_i$  are typically of order $1$, since they are coordinates of points on $n^{1/p}\bnp$ and $n^{1/2}\sphn$, respectively. Thus, the sum over all $i=1,\dots,n$ is of order $n$, and upon multiplying by $1/n$ (which scales the variance by a factor of $1/n^2$), we find that $\wnpbth$ is of the appropriate scale. 

For an alternative perspective, recall the corresponding central limit results briefly discussed in Sect. \ref{sec-intro}. Note that $n^{1/p}$ is the scaling appropriate for central limit fluctuations. To be precise, let $c_{n,p}$ be the isotropic constant (see, e.g., \cite[p.71]{ball1988logarithmically} for a definition) for the law of $n^{1/p}X^{(n,p)}$. A straightforward calculation shows that $\lim_{n\ra\infty} c_{n,p} = [p^{1/p} \Gamma(3/p)/\Gamma(1/p)]^{1/2}$, a numerical constant depending on $p$. That is, the $n^{1/p}$ factor ensures that the isotropic constants of $n^{1/p}\bnp$ are normalized to be at the same scale for all dimensions $n\in \N$. From this point of view, the scaling $n^{(1/p)-(1/2)}$ is natural for large deviations, as it is just the CLT scaling multiplied by $n^{-1/2}$. 
\end{remark}

For classical sums of i.i.d. random variables, Cram\'er's theorem gives the LDP rate function as the Legendre transform of the logarithmic moment generating function (log mgf) of the common distribution. In our setting of random projections, certain analogs of the log mgf arise, which we now define. For $p\in[2,\infty)$, let
\begin{equation}\label{chklampdefn}
\lma_{p}(t_0, t_1, t_2) \doteq \log \int_\R \int_{\mathbb{R}} e^{t_0z^2 + t_1zy + t_2 |y|^p} \,  \mu_2(dz)\mu_p(dy), \quad t_0,t_1,t_2\in \R.
\nom[phip]{$\lma_p$}{log mgf for annealed}
\end{equation}
Note that $\lma_p(t_0,t_1,t_2) < \infty$ if and only if  $t_0 < \frac{1}{2}, t_1\in \R,t_2 < \frac{1}{p}$. Our rate function is defined in terms of the Legendre transforms of $\lma_{p}$: for $w\in \R$, let
\begin{align} \label{iadefn}
\ia_{p}(w) &\doteq  \inf_{\substack {\tau_0 > 0, \tau_1\in \R, \tau_2 > 0\,:\\ \tau_0^{-1/2}\tau_1\tau_2^{-1/p} = w}} \lma_{p}^*(\tau_0,\tau_1,\tau_2).
\nom[ipa]{$\ia_p$}{annealed rate function}
\end{align}

\begin{theorem}[Annealed LDP, $p\in[2,\infty)$]\label{th-aldp}
Let $p \in [2,\infty)$. The sequence $(\wnpbth)_{n\in\N}$ satisfies an LDP with the quasiconvex, symmetric, good rate function  $\ia_{p}$.
\end{theorem}

The proof of Theorem \ref{th-aldp} is given in Sect. \ref{ssec-anng2}.

\smallskip

For $p <2$, random projections display significantly different large deviation behavior. For $p\in [1,2)$, define 
\begin{align}
  \ia_p(w) &\doteq \tfrac{1}{r_p}|w|^{r_p}, \quad w\in \R, \label{anp12}\\
  r_p &\doteq \tfrac{2p}{2+p}. \label{rp}
\end{align}
Note that $r_p < 1$ for $p < 2$, so the following large deviation principle holds with a speed $n^{r_p}$, slower than the speed $n$ associated with the case $p\ge 2$.
\nom[rp]{$r_p$}{the exponent/scale $2p/(2+p)$}

\begin{theorem}[Annealed LDP, $p\in[1,2)$]\label{th-aldp12}
Let $p\in[1,2)$. The sequence $(\wnpbth)_{n\in \N}$ satisfies an LDP with speed $n^{r_p}$ and the quasiconvex, symmetric, good rate function $\ia_p$.
\end{theorem}

The proof of Theorem \ref{th-aldp12} is given in Sect. \ref{ssec-annpl2}.

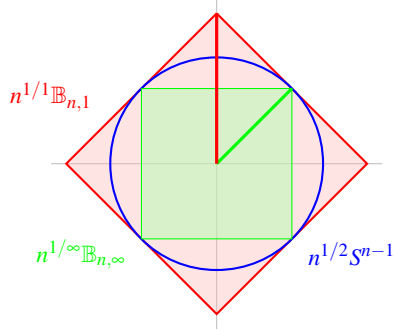
\begin{figure}[htb]
	\begin{tikzpicture}
	\node [red] at (-2.2,0.9) {\footnotesize$n^{1/1}\mathbb{B}_{n,1}$};
	\node [color=green] at (-1.8,-1.2) {\footnotesize$n^{1/\infty}\mathbb{B}_{n,\infty}$};
	\node [blue] at (1.8,-1.2) {\footnotesize$n^{1/2}S^{n-1}$};
\draw [gray!50]  (0,-2.2) -- (0,2.2) ;
\draw [gray!50]  (-2.2,0) -- (2.2,0) ;
    \draw [red, fill opacity=0.35, fill=red!30, thick] (0,2) -- (2, 0) -- (0,-2) -- (-2,0) -- (0,2) ;
    \draw [blue, thick] plot [smooth cycle, tension=1] coordinates {(-1.4142,0) (0,1.4142) (1.4142,0) (0,-1.4142)};
    \draw [green, fill opacity=0.5, fill=green!30] (-1,-1) -- (-1, 1) -- (1, 1) -- (1, -1) -- (-1, -1);
    \draw [green, very thick] (0,0) -- (1,1);
    \draw [red, very thick] (0,0) -- (0, 2);
	\end{tikzpicture}
\caption{Scaled $\ell^1$ ball vs. scaled $\ell^\infty$ ball.}
\label{fig-1inf}
\end{figure}

\vspace*{-1em}

\begin{remark}\label{rmk-p12}
Note that Theorem \ref{th-aldp} and Theorem \ref{th-aldp12} reveal a sharp difference between the LDPs for $p > 2$ and $p <2$. Due to the rotational invariance of the law of $\bthn$, a large deviation of $\langle \xnp, \bthn \rangle_n$ depends crucially on  a large deviation of the Euclidean norm of $\xnp$. The difference between the cases $p>2$ and $p <2$ is a consequence of the geometry of the $\ell^p$ balls, highlighted in Figure \ref{fig-1inf}, which portrays the scaled balls $n^{1/p}\bnp$, with  \textcolor{red}{$p=1$ in red}, and \textcolor{green}{$p=\infty$ in green}. For $p>2$, the vectors in $n^{1/p}\bnp$ that attain maximal Euclidean norm are the ``corners" $(\pm 1,\pm 1,\dots,\pm 1)$.  Meanwhile, for $p <2$, the vectors in $n^{1/p}\bnp$ that attain maximal Euclidean norm are again ``corners", but this time the corners are in canonical basis directions $(\pm 1,0,\dots,0)$, $(0,\pm1, 0,\dots, 0)$, etc. In particular, this means that for $p>2$, a large deviation of the Euclidean norm  occurs due to a combined large deviation of each coordinate. On the other hand, for $p<2$, the large deviation event is caused by the large deviation of a single coordinate. The behavior in the $p<2$ case is similar to the observation that for random walks with heavy-tailed increments, a large deviation is caused by an extreme of the sample \cite[\S4]{mikosch1998large}, which is also referred to as the ``principle of the big jump" \cite{foss2007discrete}.
\end{remark}

\subsection{Quenched LDP}\label{ssec-que} 

We now consider the case where we condition on a fixed sequence of directions $\bthn=\thn$, $n\in \N$. Let $\seq\doteq \prod_{n \in \N} \sphn$. Given a sequence of projection directions $\theta = (\theta^{(1)}, \theta^{(2)}, \dots) \in \seq$,  consider the sequence of random variables $\wnpth, n \in \N,$ defined by 
\begin{equation}\label{def-wthetan}
  \wnpth \doteq   \frac{n^{1/p}}{n^{1/2}}\langle \xnp, \thn   \rangle_n  = \frac{1}{n} \sum_{i=1}^n (n^{1/p}\xnp_i)(n^{1/2}\thn_i), \quad n \in \N.
\end{equation}
 Observe that $\wnpth$ denotes the normalized (scalar) projection of $\xnp$ onto a \emph{particular} direction $\thn$, whereas $\wnpbth$ of \eqref{def-bwnpbth} denotes the normalized (scalar) projection of $\xnp$ onto a \emph{random} direction $\bthn$. The scaling $n^{(1/p)-(1/2)}$ follows from the same rationale as in the annealed case, discussed in Remark \ref{rmk-scaling}. 
\nom[wnpthq]{$\wnpth$}{$n^{(1/p)-(1/2)} \langle \xnp, \thn\rangle_n$}
\nom[snn]{$\seq$}{sequences $\seq= \prod_{n\in\N} \sphn$}

In the case of \emph{fixed} directions of projection $\theta\in\seq$ (or conditioning on $\bth = \theta$), the corresponding analog of the log mgf is as follows. For $p\in(1,\infty)$, $\nu\in\mathcal{P}(\R)$, define
\begin{align}
\lm_{p}( t_1, t_2) &\doteq  \log\left( \int_\mathbb{R} e^{ t_1 y +  t_2 |y|^p} \mu_p(dy)  \right); \label{lampdefn0} \\
\lmq_{p,\nu}( t_1, t_2) &\doteq  \int_\mathbb{R} \lm_p(t_1u,t_2)  \nu(du)\, , \quad t_1,t_2\in \R. \label{lampdefn} 
\nom[lzzamp]{$\lm_p$}{log mgf of $(Y,\lvert Y \rvert^p)$ for $Y\sim\mu_p$}
\nom[psip]{$\lmq_{p,\nu}$}{log mgf for quenched}
\end{align}
Note that $\lmq_{p,\nu}(t_1,t_2) < \infty$ for $t_2 < 1/p$, and is equal to infinity, otherwise. We define the associated rate function in terms of the Legendre transform of $\lmq_{p,\nu}$: for $w\in \mathbb{R}$, let
\begin{align}\label{ieqdefn}
\iq_{p,\nu}(w) &\doteq  \inf_{\substack{ \tau_1\in \R, \tau_2 > 0\,: \\  \tau_1\tau_2^{-1/p} = w }} \lmq_{p,\nu}^*(\tau_1,\tau_2).
\nom[ipqu]{$\iq_{p,\nu}$}{quenched rate function}
\end{align}

Let $\pi_n:\seq\ra \sphn$ be the coordinate map such that for $\theta \in \seq$, we have $\pi_n(\theta) = \thn$.  
Let $\sigma$ be any probability measure on (the Borel sets of) $\seq$ such that for all $n\in \N$,
\begin{equation}\label{sigproj}
\sigma \circ \pi_n^{-1} = \sigma_{n}.
\end{equation}
For example, the product measure $\sigma=\bigotimes_{n\in \N} \sigma_{n}$ satisfies \eqref{sigproj}.  Our second result establishes an LDP for $(\wnpth)_{n\in \N}$ which holds for $\sigma$-a.e.\ $\theta \in \seq$.
\nom[szsig2]{$\sigma$}{ a measure on $\seq$  }

\begin{theorem}[Quenched LDP, $p \in (1,\infty)$]\label{th-qldp}
Let $p\in(1,\infty)$. For  $\sigma$-a.e.\ $\theta \in \seq$,   the sequence $(\wnpth)_{n\in\mathbb{N}}$  satisfies an LDP with the quasiconvex, symmetric, good rate function $\iq_{p,\mu_2}$.
\end{theorem}

The proof of Theorem \ref{th-qldp} is given in Sect.\ \ref{sec-quenched}. 

\smallskip

Interestingly, note that  almost every sequence of directions of projection yields the same exponential rate of decay! That is, for $\sigma$-a.e.\ $\theta \in \seq$, the rate function $\iq_{p,\mu_2}$ does not depend on  the particular choice of $\theta\in \seq$. This is not obvious at first sight, because in principle, the rate function for $(\wnpth)_{n\in\N}$ should depend on the particular choice of $\theta\in\seq$. Note that the rate function is measurable with respect to the tail sigma-algebra generated by the sequence $(\theta^{(1)},\theta^{(2)},\dots)$. Hence, if $\sigma$ were the product measure $\sigma=\bigotimes_{n\in \N} \sigma_{n}$, then  the lack of dependence of $\iq_{p,\mu_2}$ on $\theta$ would follow from the Kolmogorov 0--1 law. However, our result holds for any $\sigma\in\mathcal{P}(\seq)$ satisfying \eqref{sigproj}. We refer to \cite[Remark 3.3]{gkr3} for further comment. The key is that the $\sigma$-a.e.\ asymptotic behavior of $\sqrt{n}\thn$ which is relevant for the proof of Theorem \ref{th-qldp} depends only on the row-wise behavior of the array $\theta$ specified by \eqref{sigproj}. 

A natural question to ask is whether there exists a subset of $\seq$ of measure zero that displays ``atypical" behavior; that is, for which an LDP still holds, but with a rate function that is different from the universal quenched rate function $\iq_{p,\mu_2}$. We address this question in Sect.\ \ref{ssec-atyp}, for $p \in(1,\infty)$.

On another note, for $p=2$, we can strengthen Theorem \ref{th-qldp} to hold for \emph{all} $\theta \in \seq$, not just for $\sigma$-a.e.\ $\theta \in \seq$. This and other unique aspects of the $p=2$ case will be explored further in Sect.\ \ref{ssec-p2}.

The preceding discussion applies only to the case $p\in(1,\infty)$. For $p=1$,  the integrated log mgf $\lmq_{1,\mu_2}(t_1,t_2)$ is infinite if $t_1\ne 0$, and the same techniques as in the case $p\in(1,\infty)$ do not apply. Instead, for $p=1$ and $c > 0$, define
\begin{equation}\label{iq1def}
  \iq_{1,c}(w) \doteq \frac{|w|}{c}, \quad w\in \R.
\end{equation}

\begin{theorem}[Quenched LDP, $p =1$]\label{th-qldp1}
Fix $\theta\in\seq$ such that 
\begin{equation}\label{thmax}
\lim_{n\ra\infty} \sqrt{\frac{n}{\log n}} \max_{1\le i \le n} \theta_i^{(n)}   = c.
\end{equation}
Then, $(\sw_\theta^{(n,1)})_{n\in\mathbb{N}}$  satisfies an LDP with speed $n/\sqrt{\log n}$ and the good rate function $\iq_{1,c}$.
\end{theorem}

The proof of Theorem \ref{th-qldp1} is given in Sect. \ref{ssec-qp1}. Note that unlike the ``universal" rate function $\iq_{p,\mu_2}$ of Theorem \ref{th-qldp}, which  is the LDP rate function for $\sigma$-a.e $\theta\in\seq$ and any $\sigma$ satisfying \eqref{sigproj}, the quenched LDP for $p=1$ (Theorem \ref{th-qldp1}) depends on the particular sequence $\theta\in\seq$ through the condition \eqref{thmax}. We discuss the condition \eqref{thmax} further in Remark \ref{rmk-whysig}.

\subsection{Relationship between the annealed and quenched LDPs}\label{ssec-rel}

Let $m_q$ denote the $q$-th absolute moment of a measure,
\begin{equation}\label{qmomdef}
  m_q(\nu) \doteq \int_\mathbb{R} |x|^q \nu(dx), \quad \nu \in \mathcal{P}(\mathbb{R}).
\end{equation}
\nom[mq]{$m_q(\cdot)$}{$q$-th  absolute moment}
Let $H(\cdot | \cdot)$ denote the relative entropy between two measures; that is, for $\nu,\mu\in \mathcal{P}(\R)$,
\begin{equation*}
  H(\nu | \mu) \doteq \left\{\begin{array}{cl}
\displaystyle\int_\R \log \left(\frac{d\nu}{d\mu}\right) d\nu, & \text{ if } \nu \ll \mu, \\
+\infty, & \text{ else.}
\end{array}\right. 
\end{equation*}
We identify a variational formula that relates the annealed and quenched rate functions.
\nom[Hnumu]{$H(\cdot | \cdot)$}{relative entropy}

\begin{theorem}[Relationship between annealed and quenched LDPs]\label{th-compar}
Let $p\in [2,\infty)$. Then, for all $w\in \R$,
\begin{align}
  \ia_{p}(w) &= \inf_{\substack{\nu \in \mathcal{P}(\R):\\ m_2(\nu)\le 1}} \left\{\iq_{p,\nu}(w) + H(\nu | \mu_2) + \tfrac{1}{2}\left(1- m_2(\nu)\right)  \right\}. \label{varform1}
\end{align}In particular, this implies that  $\ia_{p}(w) \le \iq_{p,\mu_2}(w)$ for all $w\in \R$.
\end{theorem}

We prove this theorem in Sect. \ref{ssec-varadh}, as a consequence of the groundwork laid in Sect.\ \ref{ssec-rnp} and Sect. \ref{ssec-empcone}. We also discuss the minimizers of this variational problem in Sect.\ \ref{sec-analysis}.

\smallskip

As established in Proposition \ref{prop-quegen}, the term $\iq_{p,\nu}$ in \eqref{varform1} is the large deviation rate function for  projections of the random point $\xnp$ onto a particular outcome of fixed directions of projection $\Theta = \theta$ (i.e., a quenched ``environment") corresponding to the measure $\nu$. On the other hand, we will see in Sect.\ \ref{ssec-empcone} that $H(\cdot | \mu_2) + \frac{1}{2}(1-m_2(\cdot))$  is the large deviation rate function for the underlying environment $\Theta$. That is, an annealed large deviation arises precisely due to the combination of: (i) a deviation of the environment; and (ii) the deviation of a projection within such an environment.

\begin{remark}
Although quenched and annealed LDPs have been considered in other contexts such as random walks in random environments (RWRE), with the exception of \cite[Eqn.\ (9)]{comets2000quenched}, there appear to be relatively few results that relate quenched and annealed rate functions via a variational formula in the spirit of Theorem \ref{th-compar}. See also \cite[Eqn. (1.9)]{aidekon2010large} for a weaker comparison. As one would expect due to the different contexts, the proofs in the RWRE setting are quite different in nature from our proof.
\end{remark}

\subsection{Atypical directions of projection}\label{ssec-atyp}

As noted in the discussion following Theorem \ref{th-qldp}, the LDP rate function $\iq_{p,\mu_2}$ is the same for $\sigma$-a.e.\ sequence of directions $\theta \in \seq$. In this section, we compare  the $\sigma$-a.e.\ sequences of directions with sequences in the set of measure zero for which Theorem \ref{th-qldp} does not hold.

One particular sequence to consider is $\iota = (\iota^{(1)},\iota^{(2)},\dots) \in \seq$, where $\iota^{(n)}$ is the vector of 1's as in \eqref{onedefn}. Then, $\sw_\iota^{(n,p)}$ denotes the projection of $\xnp$ onto a particular ``corner" direction. In order to make a comparison between the particular sequence $\iota$ and ``generic" sequences $\theta$ for which the quenched LDP holds, we define the following rate functions:
\begin{align}\label{cramratedefn}
  \jc_p(w) &\doteq \inf_{\substack{\tau_1\in \R, \tau_2 > 0 : \\ \tau_1\tau_2^{-1/p} = w}} \lm_p^*(\tau_1,\tau_2), \quad  w\in \R. 
  \nom[ipcr]{$\jc_p$}{Cram\'er rate function}
\end{align}
As we elaborate in Remark \ref{rmk-shao}, this rate function is related to large deviations for \emph{self-normalized} random variables.

\begin{theorem}[Atypicality, $p\in (1,\infty)$]\label{th-atyp}
For $p\in (1,\infty)$, the sequence $(\sw_\iota^{(n,p)})_{n\in \N}$ satisfies an LDP with the quasiconvex, symmetric, good rate function $\jc_p$. Moreover, for $w\in (-1,1)$, we have the following:
\begin{enumerate}[itemsep=4pt]
\item for $p > 2$, $\iq_{p,\mu_2}(w) \ge \jc_p(w)$, with equality if and only if $w=0$;
\item for $p =2$, $\iq_{p,\mu_2}(w) = \jc_p(w)$;
\item for $p <2$, $ \iq_{p,\mu_2}(w) \le \jc_p(w)$, with equality if and only if $w =0$.
\end{enumerate}
\end{theorem}

The proof of Theorem \ref{th-atyp} is given in Sect.\ \ref{sec-atyp}. 

\smallskip

An analogous result for product measures is the focus of  \cite{gkr3}, as we briefly recall in Sect. \ref{ssec-queprod}. See Figure \ref{fig-comp} for a sketch of how the universal quenched rate function compares to the exceptional rate function associated with $\iota$,  in the case of projections of a random variable uniformly distributed on the $\ell^\infty$ ball.

\begin{figure}[hbt]
\includegraphics[scale=0.55, trim=2.5in 3.85in 2.5in 3.85in]{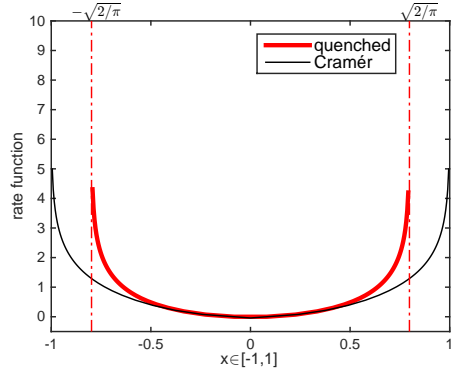}
\caption{Quenched vs. Cram\'er rate functions for $p=\infty$.}
\label{fig-comp}
\end{figure}

\begin{remark}
A similar notion of atypicality can be found in the work of \cite{barthe2003extremal}, where the authors were interested in slabs of convex bodies. In particular, the authors proved a Cram\'er-type upper bound for the volume of slabs of certain convex bodies, and noted that this upper bound was asymptotically attained by the sequence of ``extremal" slabs orthogonal to the main diagonal $(1,1,\dots,1)$. Our result Theorem \ref{th-atyp} shows that the sequence of directions $(1,1,\dots,1)$ is not only extremal, but also particularly distinct, in that almost every other sequence of directions yields a universal rate function different from that of the extremal direction. 
\end{remark}

\begin{remark}
Another particular sequence of directions to consider is the sequence of canonical basis vectors $e_1 = (e_1^{(1)}, e_1^{(2)},\dots) \in \seq$, where
\begin{equation*}
  e_1^{(n)}  \doteq (1,\underbracket{0,\dots, 0}_{n-1 \text{ times }}) \in \sphn.
\end{equation*}
That is, project $\xnp$ onto its first coordinate. In the language of \cite{barthe2003extremal}, this is the volume of the ``canonical slab" of the $\ell^p$ ball. It is known due to \cite[Theorem 3.4]{BarGamLozRou10} that the sequence $(\langle \xnp, e_1^{(n)}\rangle_n)_{n\in \N}$ satisfies an LDP with speed $n$ and rate $J_p(x) = -\frac{1}{p} \log(1-x^p)$. Note, however, that this sequence lacks the $\frac{n^{1/p}}{n^{1/2}}$ scaling found in $\wnpth$, so the sequence $e_1$ is also atypical in its own sense, at least for $p\ne 2$.
\end{remark}

\nom[e1]{$e_1^{(n)}$}{first coord. $(1,0,\dots,0) \in \sphn$}

\subsection{Special case of $p=2$ }\label{ssec-p2}

As a brief digression, we consider  the special case of $p=2$. First, define the rate function for $w\in \R$:
\begin{equation}\label{j2defn}
J_2(w) \doteq \left\{\begin{array}{cl} - \frac{1}{2} \log(1-w^2), & w\in (-1,1);\\ +\infty, & \text{ else.}\end{array}\right.
\end{equation}
\nom[j2]{$J_2$}{rate function for $p=2$}
Then, our results can be summarized as follows:
\begin{theorem}\label{th-p2}
For $p=2$, the quenched LDP of Theorem \ref{th-qldp} holds for all $\theta \in \seq$. Moreover,
\begin{equation}\label{equalrates}
\ia_{2} = \iq_{2,\mu_2} = J_2.
\end{equation}
\end{theorem}

\begin{proof}
Note that $\sx^{(n,2)}$ is distributed uniformly over the Euclidean ball, so its distribution is spherically symmetric in the sense that for all $n$ and all $\eta,\eta'\in \sphn$,
\begin{equation}\label{sphersymm}
  \langle \bxn, \eta\rangle_n  \eqdist \langle \bxn, \eta' \rangle_n.
\end{equation}
In particular, this implies that for $e_1^{(n)}= (1,0,\dots,0) \in \sphn$, 
\begin{equation*}
  \P(\langle \bxn, \bthn\rangle_n \in \cdot ) = \P(\langle \bxn, \thn \rangle_n \in \cdot ) = \P(\langle \bxn, e_1^{(n)} \rangle_n \in \cdot \, ).
\end{equation*}
The upshot is that to analyze either the annealed LDP for $(\sw^{(n,2)})_{n\in\N}$, or the quenched LDP for $(\sw_\theta^{(n,2)})_{n\in\N}$, it suffices to consider the LDP of $(\sw_{e_1}^{(n,2)})_{n\in\N}$, the sequence of projections onto the first coordinate. In this case, it is known from \cite[Theorem 3.4]{BarGamLozRou10} that this sequence satisfies an LDP with good rate function $J_2$. It is also possible to prove the equality \eqref{equalrates} by direct calculation.
\end{proof}

Note that the key part in the preceding proof is spherical symmetry, a property which we will use again to a different end in Sect.\ \ref{ssec-rnp}. It is this spherical symmetry which leads to the ``for all" claim in Theorem \ref{th-p2}, as opposed to the ``$\sigma$-a.e." claim in Theorem \ref{th-qldp}. 

\begin{remark}
While Theorem \ref{th-p2} shows that the quenched and annealed rate functions are identical when $p=2$, Proposition \ref{prop-nongsn} shows that the quenched and annealed rate functions do \emph{not} coincide when $p=\infty$.
\end{remark}

\section{An equivalent formulation}\label{sec-equiv} 

When $p< \infty$, the non-trivial dependence between the coordinates that is induced by the uniform measure on $\bnp$ makes a direct large deviation analysis difficult. To resolve this, we invoke a more convenient representation for the uniform measure on $\bnp$ to reduce the analysis of $\wnpbth$ and $\wnpth$ to that of more tractable objects. Furthermore, this representation will also clarify the role of the density $f_p$ introduced in \eqref{fpdef}.

\subsection{A probabilistic representation for the uniform measure on $\bnp$}\label{ssec-repn}

Let $n\in \N$ and  $p \in [1,\infty)$.  Consider the following random variables, defined on the same common probability space $(\Omega,\mathcal{F}, \P)$ as in Sect.\ \ref{ssec-setup}:
\begin{itemize}
\item  $U$  is uniformly distributed on $[0,1]$;
\item $ \bolyp = ( \ynp )_{n\in\mathbb{N}} = (\,(\ynp_1,\dots,\ynp_n)\,)_{n\in\mathbb{N}}$ is a triangular array of i.i.d. real-valued random variables, with common distribution $\mu_p$ defined by \eqref{fpdef};
\item $ \bolz = ( \zn )_{n\in\mathbb{N}} = (\, (\zn_1,\dots,\zn_n)\, )_{n\in\mathbb{N}}$ is a triangular array of independent $N(0,1)$ random variables;
\item $U$, $\bolyp$, and $\bolz$ are independent.
\end{itemize}
Then, the following properties are well known --- see, e.g., \cite[Lemma 1]{schechtman1990volume}, \cite[\S3]{rachev1991approximate}.
\nom[u]{$U$}{uniform r.v.\ on $[0,1]$}
\nom[ya]{$\boly$, $\ynp$}{array / vector of i.i.d.  $\sim \mu_p$}
\nom[za]{$\bolz$, $\zn$ }{array / vector of i.i.d. $N(0,1)$}

\begin{lemma}\label{lem-jointrep}
For $p\in[1,\infty)$,
\begin{equation}
\label{bthn-rep}
\left(  \xnp, \bthn\right) \eqdist  \left(U^{1/n}  \frac{\ynp}{\lpnrm{\ynp}}, \frac{\zn}{\ltwonrm{\zn}}\right).
\end{equation}
Moreover, $\ynp / \lpnrm{\ynp}$ is independent of $\lpnrm{\ynp}$, and $\zn/\ltwonrm{\zn}$ is independent of $\ltwonrm{\zn}$.
\end{lemma}

Define the sequences of random variables $(\wnpz)_{n\in\N}$  and $(\wnpz_\theta)_{n\in\N}$ as follows: for $n\in \N$ and $\theta \in \seq$,
\begin{align}
  \wnpz &\doteq \frac{n^{1/p}}{n^{1/2}}  U^{1/n} \frac{ \sum_{i=1}^n  \ynp_i \zn_i}{\lpnrm{\ynp}\ltwonrm{\zn}}; \label{altseq}\\
  \wnpz_\theta  &\doteq \frac{n^{1/p}}{n}  U^{1/n} \frac{ \sum_{i=1}^n  \ynp_i \sqrt{n}\thn_i}{\lpnrm{\ynp}} \label{altseq2}.
\nom[wnpz1]{$\wnpz$}{representation of $W^{(n,p)}$}
\end{align}
The definitions \eqref{def-bwnpbth} and \eqref{def-wthetan} together with \eqref{bthn-rep}, \eqref{altseq} and \eqref{altseq2} show that for $n\in \N$ and $\theta\in\seq$,
\begin{align}
 \wnpbth &\eqdist \wnpz \label{bwthn-rep};\\
 \wnpth &\eqdist \wnpz_\theta \label{bwthn-rep2}.
\end{align}

\subsection{Concentration on the boundary}\label{ssec-bdry}

Continue to assume $p\in [1,\infty)$. In this section, we show that, for the purposes of both the annealed and quenched LDPs, it is possible to ignore the contribution of the ``radial" term $U^{1/n}$ in the definition of $\wnpz$ given by \eqref{altseq}. This is related to the fact that the uniform measure on high-dimensional isotropic convex bodies concentrates strongly on the boundary. Note that unlike in the central limit setting, our asymptotic result  as $n\rightarrow \infty$ does not rely on the delicate ``thin-shell" estimates derived for finite $n$ dimensions \cite{klartag2007central}.

\begin{lemma}\label{lem-unifequiv}
Suppose that a sequence of $\R$-valued random variables $(\xi_n)_{n\in\N}$ satisfies an LDP with a good rate function $I_\xi(\cdot)$.  Let $U$ be an independent random variable uniformly distributed on $[0,1]$. If $I_\xi$ is quasiconvex and symmetric, then the sequence $(U^{1/n}\xi_n)_{n\in \N}$ satisfies an LDP with good rate function $I_\xi$.
\end{lemma}

To prove Lemma \ref{lem-unifequiv}, we begin by appealing to the large deviation behavior of $U^{1/n}$ as $n\ra\infty$.

\begin{lemma}\label{lem-uniform}
The sequence $(U^{1/n})_{n\in \N}$ satisfies an LDP with the good rate function
\begin{equation*}
  I_U(u) \doteq \left\{ \begin{array}{ll}
 - \log u & u\in(0,1] ; \\
   +\infty & \textnormal{ else. }
 \end{array}\right.
\end{equation*}
\end{lemma}

\begin{proof}
Let $A$ be a Borel set in $\mathbb{R}$. First, we prove the large deviation upper bound; that is,  $\limsup_{n\ra\infty} \frac{1}{n} \log \P(U^{1/n} \in A) \le -\inf_{u\in \bar{A}} I_U(u)$. If $1\in \bar{A}$, then $\inf_{u \in \bar A} I_U(u) = 0$, so the upper bound in this case is automatic. Otherwise, let $u_1 = \sup\{ u \in A : u < 1 \}$ and  $u_2 = \inf\{ u \in A : u > 1\}$. Since $I_U$ is convex with minimum at 1, and infinite outside $(0,1]$,
\begin{align*}
\inf_{u\in \bar A} I_U(u) &= I_U(u_1).
\end{align*}
Using the fact that $\bar{A} \subset (-\infty,u_1] \cup [u_2, \infty)$, and $\P(U^{1/n} \ge u_2) = 0$, we find that
\begin{align*}
\limsup_{n\ra\infty} \frac{1}{n} \log \P(U^{1/n} \in \bar{A}) &\le \limsup_{n\ra\infty} \frac{1}{n} \log \left( \P(U^{1/n} \in (-\infty,u_1]) + \P(U^{1/n} \in [u_2,\infty))\right)\\
   &= \limsup_{n\ra\infty} \frac{1}{n} \log \P(U \le u_1^n)\\
   &= \left\{ \begin{array}{cc}
 \log u_1 & \text{ if } u_1 > 0\\
 -\infty & \text{ else }
 \end{array}\right.\\
   &= -\inf_{u\in \bar A} I_U(u).
\end{align*}

Now we prove the large deviation lower bound, $\liminf_{n\ra\infty} \frac{1}{n} \log \P(U^{1/n} \in A) \ge -\inf_{u\in A^\circ} I_U(u)$. If $A^\circ \cap (0,1] = \emptyset$, then $\inf_{u\in A^\circ} I_U(u) = \infty$, so the lower bound in this case is automatic. Otherwise, let $\delta > 0$, and let $\bar{u} \in A^\circ \cap (0,1]$ such that $I_U(\bar{u}) \le \inf_{u \in A^\circ} I_U(u) + \delta$. Since $A^\circ$ is open, there exists $\epsilon \in (0, \bar{u})$ such that $(\bar u- \epsilon, \bar u] \subset A^\circ \cap (0,1]$. Thus,
\begin{align*}
\liminf_{n\ra\infty} \frac{1}{n} \log \P(U^{1/n} \in A^\circ ) &\ge \liminf_{n\ra\infty} \frac{1}{n} \log \P(U^{1/n} \in (\bar u - \epsilon , \bar{u}])\\
    &=  \liminf_{n\ra\infty} \frac{1}{n} \log \P(U \in ((\bar{u}-\epsilon)^n, \bar{u}^n])\\
    &= \liminf_{n\ra\infty} \frac{1}{n} \log (\bar{u}^n - (\bar{u} - \epsilon)^n)\\
    &= \log \bar{u}\\
    &= -I_U(\bar{u})\\
    &\ge - \inf_{u\in A^\circ} I_U(u) - \delta.
\end{align*}
This holds for arbitrary $\delta > 0$, so the lower bound follows.
\end{proof}

\begin{proof}[Proof of Lemma \ref{lem-unifequiv}]
By independence, the sequence $(U^{1/n},\,\,\xi_n)_{n\in \N}$ satisfies a joint LDP with rate function $I_{U, \xi }(u,w) =I_U(u) + I_\xi(x)$, where $I_U$ is the rate function computed in Lemma \ref{lem-uniform}. By the contraction principle, the sequence of products $(U^{1/n}\xi_n)_{n\in \N}$ satisfies an LDP with the rate function $I$, where for $\tilde{x}\in \R$,
\begin{align*}
  I(\tilde{x}) &= \inf\{ I_U(u) + I_{\xi}(x) :u,x\in \R, ux = \tilde x\},\\
       &= \inf\{ -\log u + I_{ \xi }(x) :  u\in(0,1], x\in \R, ux = \tilde x\}.
\end{align*}
Let $\tilde x > 0$. By assumption, $I_{\xi}$ is quasiconvex and symmetric, so it is minimized at $x=0$ and non-decreasing for $x > 0$. Using the fact that $x\mapsto \log x$ is increasing, the infimum is attained at $u=1$ and $x=\tilde x$. Therefore, $I(\tilde x) = I_\xi(\tilde x)$. Likewise, when $\tilde{x} < 0$, similar calculations show once again that $I(\tilde{x}) = I_\xi(\tilde{x})$.
\end{proof}

For $p<\infty$, the equivalence of the LDPs given by Lemma \ref{lem-unifequiv} motivates the analysis of the sequences $(\twnpz)_{n\in \N}$ and $(\twnpz_\theta)_{n\in \N}$ defined as follows: for $n\in \N$ and $\theta\in\seq$,
\begin{align} 
  \twnpz &\doteq  \frac{n^{1/p}}{n^{1/2}}  \frac{ \sum_{i=1}^n  \ynp_i \zn_i}{\lpnrm{\ynp}\ltwonrm{\zn}}, \label{wnou}\\
\twnpz_\theta &\doteq \frac{n^{1/p}}{n}  \frac{ \sum_{i=1}^n  \ynp_i \sqrt{n}\thn_i}{\lpnrm{\ynp}}.  \label{wnou2}
\nom[wnpz2]{$\twnpz$}{$\widehat{W}^{(n,p)}$ sans $U^{1/n}$ factor}
  \end{align}
In the following lemma, we claim that it suffices to analyze the sequences defined by \eqref{wnou} and \eqref{wnou2}.

\begin{lemma}\label{lem-reduction}
If the sequence $(\twnpz)_{n\in\N}$ satisfies an LDP with good rate function $\ia_p$, then $(\wnpbth)_{n\in \N}$ satisfies an LDP with the same rate function. Similarly, if  the sequence $(\twnpz_\theta)_{n\in \N}$ satisfies an LDP with good rate function $\iq_{p,\mu_2}$ for $\sigma$-a.e.\ $\theta \in \seq$, then the sequence $(\wnpth)_{n\in\N}$ satisfies an LDP with the same rate function for $\sigma$-a.e.\ $\theta \in \seq$.
\end{lemma}
\begin{proof}
Due to \eqref{bwthn-rep} and \eqref{bwthn-rep2}, $\wnpbth$ and $\wnpth$ are equal in distribution to $\wnpz$ and $\wnpz_\theta$, respectively. Thus, it suffices to show that an LDP for $(\twnpz)_{n\in\N}$ (resp., $(\twnpz_\theta)_{n\in\N}$) implies an LDP for $(\wnpz)_{n\in\N}$ (resp., $(\wnpz_\theta)_{n\in\N}$) with the same rate function. However, this would follow from Lemma \ref{lem-unifequiv} if  $\ia_p$ and $\ia_{p,\mu_2}$ could be shown to be quasiconvex and symmetric.

For $\ia_p$, note that by \eqref{iadefn},
\begin{align}
\ia_{p}(w) &= \inf_{\tau_0,\tau_2 > 0} \lma_{p}^*(\tau_0,w\tau_0^{1/p}\tau_2^{1/p},\tau_2), \quad w\in \R. \label{iarepn}
\end{align}
Since $\mu_p$ and $\mu_2$ are symmetric distributions, $\lma_p$ (and thus, $\lma_p^*$) is symmetric in the second variable. Then, the representation  \eqref{iarepn} implies that $\ia_{p}$ is symmetric. As for quasiconvexity, we know that $\lma_{p}^*$ is convex by definition of the Legendre transform. Combined with the symmetry of $\lma_{p}^*$ in the second argument, we see that for fixed $\tau_0,\tau_2> 0$, $\lma_{p}^*(\tau_0,\tau_1, \tau_2)$ is minimized at $\tau_1 = 0$, non-decreasing for $\tau_1 > 0$, and non-increasing for $\tau_1 < 0$. Thus, for $w' > w > 0$, \eqref{iarepn} shows that
\begin{align*}
\ia_{p}(w') &= \inf_{\tau_0,\tau_2 > 0} \lma_{p}^*(\tau_0,w'\tau_0^{1/2}\tau_2^{1/p},\tau_2) \ge  \inf_{\tau_0,\tau_2 > 0} \lma_{p}^*(\tau_0,w\tau_0^{1/2}\tau_2^{1/p},\tau_2) = \ia_{p}(w).
\end{align*}
Similar calculations for $w' < w < 0$ show that for all $c>0$, the set $\{w \in \R: \ia_p(w) \le c\}$ is a closed interval containing 0.  Thus, $\ia_p$ is quasiconvex (see Definition \ref{def-quasi}). The argument is essentially identical for $\iq_{p,\mu_2}$, and hence, left to the reader.
\end{proof}

\section{The annealed LDP}\label{sec-annealed} 

In this section, we prove Theorem \ref{th-aldp}, the annealed LDP for random projections of $\ell^p$ balls. When $p\in[2,\infty)$, the recipe  is roughly as follows: we employ the representations of $X^{(n,p)}$ and $\bthn$ given in Sect.\ \ref{sec-equiv}, apply Cram\'er's theorem for a sum of i.i.d.\ random variables in $\R^3$, and then complete the proof with the contraction principle.  The case $p\in[1,2)$ is slightly different, in that we must prove an LDP at a different speed; for this case, we still employ the representations of Sect.\ \ref{sec-equiv}, but show that deviations of the ``numerator" are relevant for the LDP, whereas the deviations of the ``denominator" do not matter.

\subsection{Annealed proof for $p\in[2,\infty]$}\label{ssec-anng2}

For $p\in [2,\infty)$, we define the following sum of i.i.d. $\mathbb{R}^3$-valued random variables,
\begin{equation*}
 S^{(n,p)} \doteq  \frac{1}{n}\sum_{i=1}^n \left( |Z^{(n)}_i|^2,   \, \ynp_i Z_i^{(n)}, \,  |\ynp_i|^p   \right), \quad n \in \N.
\end{equation*}
Note that $\lma_p$ of \eqref{chklampdefn} is the log mgf of the summands, $( |Z^{(n)}_i|^2,   \, \ynp_i Z_i^{(n)}, \,  |\ynp_i|^p   )$. We write out this sum because $\twnpz$ of \eqref{wnou} can be written  as a function of $S^{(n,p)}$. In our proof below, we will have to recall the following definition.

\begin{definition}\label{def-domain}
Consider a convex function $\Lambda:\mathbb{R}^d\ra (-\infty,\infty]$. The \emph{effective domain} of $\Lambda$ is the set
\begin{equation*}
  D_\Lambda \doteq \{x \in \R^d  :  \Lambda(x) < \infty \}.
\end{equation*}
When there is no confusion, we refer to $D_\Lambda$ as the \emph{domain} of $\Lambda$.
\nom[D]{$D_{\cdot}$}{(effective) domain of a function}
\end{definition}

\begin{proof}[Proof of Theorem \ref{th-aldp}]
For $t_0 < \tfrac{1}{2}$, $t_1\in \R$, $t_2 < \tfrac{1}{p}$,
\begin{align*}
  \lma_p(t_0,t_1,t_2) &=  \log \int_{\R}\int_{\R} e^{t_1 zy } \tfrac{1}{2p^{1/p}\Gamma(1+\tfrac{1}{p})} e^{-(1-pt_2)|y|^p/p} dy\,  \tfrac{1}{\sqrt{2\pi}}e^{-(1-2t_0)z^2/2} dz \\
    &= -\tfrac{1}{p} \log(1-pt_2) - \tfrac{1}{2}\log(1-2t_0)\\
    & \quad \quad +  \log \int_{\R}  \exp\left(\tfrac{1}{2} t_1^2(1-pt_2)^{-2/p}(1-2t_0)^{-1}y^2 \right)   \mu_p(dy) 
\end{align*}
Note that the preceding quantity is finite for $p > 2$. Thus, $D_{\lma_p}^\circ = (-\infty,\tfrac{1}{2})\times \R \times (-\infty, \tfrac{1}{p}) \ni 0$. Thus, by  Cram\'er's theorem, the sequence $(S^{(n,p)})_{n\in \N}$ satisfies an LDP in $\R^3$ with the good rate function given by the Legendre transform $\lma_{p}^*(\tau_0,\tau_1,\tau_2)$. Note that $D_{\lma_p^*} \subset (0,\infty)\times \R \times (0,\infty)$, and the map $T_p: (0,\infty)\times \R  \ra \R$ defined by
\begin{equation*}
T_p(\tau_0,\tau_1,\tau_2) \doteq  \tau_0^{-1/2}\tau_1\tau_2^{-1/p},
\end{equation*}
is continuous. Since $\twnpz = T(S^{(n,p)})$, we can apply the contraction principle to obtain an LDP for $(\twnpz)_{n\in\N}$ with the rate function
\begin{equation*}
 \inf_{\tau_0^{-1/2}\tau_1\tau_2^{-1/p} = w}  \lma_{p}^*(\tau_0,\tau_1,\tau_2)  = \ia_{p}(w), \quad w\in \R.
\end{equation*}
Due to Lemma \ref{lem-reduction}, this implies that the same LDP holds for $(W^{(n,p)})_{n\in\N}$.
\end{proof}

\subsection{Annealed proof for $p\in [1,2)$} \label{ssec-annpl2}

First, note that we cannot approach Theorem \ref{th-aldp12} in the same way as Theorem \ref{th-aldp} due to the fact that for $p < 2$,  $\lma_p(t_0,t_1,t_2) = \infty$ for $t_1 \ne 0$. This suggests that the LDP, if it exists, occurs at a different speed  slower than $n$. To identify the appropriate speed, we begin with a lemma giving upper and lower bounds for the tails of $\mu_p$.

\begin{lemma}\label{lem-ulbd}
Let $p \in[1,2)$. Then, for all $x \ge  0$,
\begin{equation*}
   \frac{x}{x^p+1}e^{-x^p/p} \le \int_x^\infty e^{-y^p/p}dy \le \frac{1}{x^{p-1}} e^{-x^p/p}.
\end{equation*}
\end{lemma}

\begin{proof}
First, we prove the upper bound. For $x\ge 0$,
\begin{equation*}
  \int_x^\infty e^{-y^p/p}dy \le \int_x^\infty \frac{y^{p-1}}{x^{p-1}} e^{-y^p/p}dy \le \frac{1}{x^{p-1}} e^{-x^p/p}.
\end{equation*}
As for the lower bound, let
\begin{equation*}
  f(x) \doteq \int_x^\infty e^{-y^p/p}dy - \frac{x}{x^p+1} e^{-x^p/p}, \quad x\ge 0.
\end{equation*}
Note that $f(0) > 0$ and $\lim_{x\ra\infty} f(x) = 0$. Lastly, since $p<2$,
\begin{equation*}
  f'(x) = -\frac{e^{-x^p/p}}{(x^p+1)^2}\left( (2-p)x^p + 2\right) < 0,
\end{equation*}
and so $f(x) \ge 0$ for all $x \ge 0$, thus proving the lower bound.
\end{proof}

\begin{lemma}\label{lem-tailyz}
Let $p \ge 1$, $Y\sim \mu_p$, and $Z\sim \mu_2$, and let $Y$ and $Z$ be independent. Then,
\begin{equation*}
\lim_{t\ra\infty}  \frac{1}{t^{r_p}} \log \P(YZ \ge t) = - r_p^{-1},
\end{equation*}
where $r_p = \frac{2p}{2+p}$ as in \eqref{rp}.
\end{lemma}
\begin{proof}
First, we prove the lower bound. Fix $t > 0$. For all $s$ such that $0 < s < t$, by the independence of $Y$ and $Z$, and the lower bound of Lemma \ref{lem-ulbd},
\begin{equation*}
  \P(YZ \ge t) \ge \P(Y \ge s) \P(Z \ge \tfrac{t}{s})  \ge C_p\,\frac{s}{s^p+1} e^{-s^p/p} \frac{1}{(t/s) + (s/t)} e^{-t^2/(2s^2)},
\end{equation*}
where $C_p <\infty$ represents a constant that depends on $p$ but not on $t$ nor $s$. Then, pick the optimal $s$ for the lower bound,
\begin{equation*}
s_t = \arg\min_s\left\{ \tfrac{s^p}{p} + \tfrac{t^2}{2s^2}\right\} =t^{2/(2+p)} = t^{r_p/p}.
\end{equation*}
Therefore, 
\begin{equation*}
  \liminf_{t\ra\infty} \frac{1}{t^{r_p}} \log \P(YZ\ge t) \ge \liminf_{t\rightarrow\infty} \frac{1}{t^{r_p}} \left(-\tfrac{s_t^p}{p} - \tfrac{t^2}{2s_t^2}\right) = -\left(\tfrac{1}{p} + \tfrac{1}{2}\right) = - r_p^{-1}.
\end{equation*}

Now  we prove the upper bound. By Lemma \ref{lem-ulbd}, for some different  constant $\tilde{C}_p < \infty$,
\begin{align*}
\P(YZ \ge t) &= \tilde{C}_p\int_0^\infty \P(Y \ge \tfrac{t}{s}) \exp\left(-\tfrac{s^2}{2}\right) ds\\
  &\le \tilde{C}_p\int_0^\infty \frac{1}{(t/s)^{p-1} } \exp\left(-\tfrac{1}{p}(\tfrac{t}{s})^p - \tfrac{s^2}{2}\right) ds \\
\text{\footnotesize($s^p = t^{p^2/(2+p)}u$)} \quad  &= \tilde{C}_p\frac{1}{pt^{p-1}t^{p^2/(2+p)}}  \int_0^\infty  \exp\left(-\tfrac{t^p}{pt^{p^2/(2+p)} u} - \tfrac{t^{2p/(2+p)} u^{2/p}}{2}\right) du.
\end{align*}
Then, using Laplace's method,
\begin{align*}
\limsup_{t\ra\infty} \frac{1}{t^{r_p}} \log\P(YZ \ge t) &\le \limsup_{t\ra\infty} \frac{1}{t^{r_p}} \log \int_0^\infty \exp\left( -t^{r_p}\left(\tfrac{1}{pu} + \tfrac{u^{2/p}}{2}\right)\right) du\\
   &= -\min_{u> 0} \left\{ \tfrac{1}{pu} + \tfrac{u^{2/p}}{2}\right\} \\
   &= -\left(\tfrac{1}{p} + \tfrac{1}{2}\right) = -r_p^{-1}.
\end{align*}

\end{proof}

We state an intermediate large deviation result. As in Sect.\ \ref{ssec-repn} (but, for ease of notation, omitting the superscripts $^{(n)}$ and $^{(n,p)}$), let $Y_1,\dots, Y_n$ be i.i.d. with common distribution $\mu_p$, and let $Z_1,\dots, Z_n$ be i.i.d. with common distribution $\mu_2$. Define the empirical mean of i.i.d.\ random variables,
\begin{equation*}
V^{(n,p)} \doteq \frac{1}{n}\sum_{i=1}^n Y_i\,Z_i.
\end{equation*}

\begin{proposition}\label{prop-stretched}
Let $p\in [1,2)$. Then, with $r_p =\frac{2p}{2+p}$, the sequence $(V^{(n,p)})_{n\in \N}$ satisfies an LDP with speed $n^{r_p}$ and the quasiconvex good rate function $\ia_p(w) = \frac{1}{r_p} |w|^{r_p}$.
\end{proposition}
\begin{proof}
This follows from \cite[Theorem 2.1]{arcones2002large}, where $p$, $b_n$, and $a$ there correspond to $r_p$, $n$, and $r_p^{-1}$ here, respectively. The condition $\frac{n}{n^{2-r_p}} \rightarrow 0$ as $n\rightarrow\infty$ holds since $r_p < 1$ for $p < 2$, and the condition  $\frac{n}{n+1}\rightarrow 1$ as $n\rightarrow\infty$ holds trivially. Then, the symmetry of $\mu_p$ and the tail asymptotics of  Lemma \ref{lem-tailyz} imply the desired LDP. Note that this result can also be deduced from \cite[Theorem 1]{gantert2014large}.
\end{proof}

We now show that at the large deviation scale, $\twnpz$ of \eqref{wnou} is comparable to $V^{(n,p)}$ in the following sense.

\begin{definition}\label{def-expequiv}
Let $(\xi_n)$  and $(\tilde{\xi}_n)$ be two sequences of $\R$-valued random variables such that for all $\delta > 0$, and some speed $s(n)$,
\begin{equation*}
  \limsup_{n\ra\infty} \frac{1}{s(n)} \log \P( |\xi_n - \tilde{\xi}_n| > \delta ) = -\infty;
\end{equation*}
then, $(\xi_n)$ and $(\tilde{\xi}_n)$ are said to be \emph{exponentially equivalent} with speed $s(n)$.
\end{definition}

\begin{proposition}[\cite{DemZeibook}]
\label{prop-expeq}
If $(\xi_n)$ is a sequence of random variables that satisfies an LDP with speed $s(n)$ and good rate function $I$, and $(\tilde{\xi}_n)$ is another sequence that is exponentially equivalent to $(\xi_n)$ with speed $s(n)$, then $(\tilde{\xi}_n)$ satisfies an LDP with speed $s(n)$ and good rate function $I$.
\end{proposition}

\begin{proof}[Proof of Theorem \ref{th-aldp12}]
We will prove that $(\twnpz)_{n\in \N}$ and $(V^{(n,p)})_{n\in \N}$ are exponentially equivalent with speed $n^{r_p}$. For $\delta > 0$, $\epsilon > 0$, 
\begin{align*}
\P & (|V^{(n,p)}- \twnpz| > \delta) \\
  &= 2\, \P\left( \frac{1}{n}\sum_{i=1}^n Y_i\,Z_i \cdot \left(1 - \tfrac{n^{1/2}n^{1/p}}{\|Z^{(n)}\|_{n,2}\|Y^{(n,p)}\|_{n,p} } \right)  > \delta \right) \\
   &\le  2\, \P\left( \frac{1}{n}\sum_{i=1}^n Y_i Z_i  > \frac{\delta}{\epsilon} \right)   + 2\, \P\left( 1 - \tfrac{n^{1/2}n^{1/p}}{\|Z^{(n)}\|_{n,2}\|Y^{(n,p)}\|_{n,p} }  > \epsilon\right)\\
   &\le  2\, \P\left( \frac{1}{n}\sum_{i=1}^n Y_i Z_i  > \frac{\delta}{\epsilon} \right)   + 2\,\P\left(\frac{1}{n}\sum_{i=1}^n Z_i^2 > (1-\epsilon)^{-1}  \right) + 2\,\P\left( \frac{1}{n}\sum_{i=1}^n |Y_i|^p > (1-\epsilon)^{-p/2}  \right).
\end{align*}
Note that by Cram\'er's theorem, the second and third terms decay exponentially with speed $n$ since $\E[|Y_1|^p]^{1/p} = \E[|Z_1|^2]^{1/2}= 1$. Thus, for $p\in [1,2)$, the first term is dominant with speed $n^{r_p}$,  yielding the limit
\begin{align*}
  \limsup_{n\ra \infty} \frac{1}{n^{r_p}} \log \P(|V^{(n,p)}- \twnpz| > \delta)  &\le  \limsup_{n\ra \infty} \frac{1}{n^{r_p}} \log \P\left( \frac{1}{n}\sum_{i=1}^n Y_i Z_i  > \frac{\delta}{\epsilon} \right) = -\tfrac{1}{r_p}\left\lvert\tfrac{\delta}{\epsilon}\right\rvert^{r_p}
\end{align*}
where the last equality follows from Proposition \ref{prop-stretched} and quasiconvexity. Sending $\epsilon \rightarrow 0$, we see that $(V^{(n,p)})_{n\in \N}$ and $(\twnpz)_{n\in\N}$ are exponentially equivalent with speed $n^{r_p}$. The LDP for $(\wnpbth)_{n\in\N}$ then follows from Proposition \ref{prop-stretched}, Proposition \ref{prop-expeq}, and the fact that the $U^{1/n}$ factor  in \eqref{def-bwnpbth} can be ignored since $(U^{1/n})_{n\in\N}$ satisfies a large deviation principle with good rate function at speed $n$ (as given by Lemma \ref{lem-uniform}).
\end{proof}

\section{The quenched LDP}\label{sec-quenched} 

In this section, we prove Theorem \ref{th-qldp}, the quenched LDP for random projections of $\ell^p$ balls. To do so, we prove LDPs for the weighted sum \eqref{wnou2}, which has \emph{deterministic} weights. This task reduces to proving an LDP for sums of  random variables which are independent but not identically distributed (in our case due to the inhomogeneous weights $\thn_i$), for which the G\"artner-Ellis theorem is well suited (see \cite[\S2.3]{DemZeibook}). We first show  in Sect.\ \ref{ssec-pressure} that the convergence of a certain empirical measure implies the convergence of a certain limiting  log mgf which arises in the G\"artner-Ellis theorem. Then, in Sect.\ \ref{ssec-glivenko}, we prove a slight extension of the Glivenko-Cantelli theorem which establishes convergence of the empirical measure in general settings. We specialize to our case of the surface measure $\sigma$ and complete the proof of the quenched LDP  in Sect.\ \ref{ssec-surface}.

\subsection{Convergence of log mgfs}\label{ssec-pressure}

In what follows, we require two notions of convergence of probability measures. Let $\Rightarrow$ denote weak convergence, and also recall the Wasserstein topology of probability measures. 
\nom[aaa]{$\Rightarrow$}{weak convergence}

\begin{definition}
Let $r\in[1,\infty)$, let $m_r$ be the $r$-th moment as in \eqref{qmomdef}, and let $\mathcal{P}_r(\R) \doteq \{\mu \in \mathcal{P}(\R) : m_r(\mu) < \infty\}$. The \emph{Wasserstein}-$r$ topology on $\mathcal{P}_r(\R)$ is induced  by the following metric:
\begin{equation*}
  \mathcal{W}_r(\mu,\nu)\doteq  \inf_{ \pi \in \Pi(\mu,\nu)} \iint_{\mathbb{R}^2} |x-y|^r\, \pi(dx,dy),
\end{equation*}
where $\Pi(\mu,\nu)$ denotes the set of probability measures on $\mathbb{R}^2$ with first and second marginals $\mu$ and $\nu$, respectively. 
\nom[pxr]{$\mathcal{P}_r(\R)$}{probability measures on $\R$ with finite $r$-th moment}
\nom[wr]{$\mathcal{W}_r$}{Wasserstein-$r$ metric}
\end{definition}

\begin{lemma}[see, e.g., Definition 6.8 and Theorem 6.9 of \cite{villani2008optimal}]\label{lem-wass}
Let $(\mu_n) \subset \mathcal{P}_r(\mathbb{R})$ and $\mu\in \mathcal{P}_r(\R)$. The following are equivalent:
\begin{enumerate}
\item $\mathcal{W}_r(\mu_n,\mu) \ra 0$;
\item $\mu_n\Rightarrow \mu$ and $m_r(\mu_n) \ra m_r(\mu)$;
\item for all continuous functions $\varphi:\R\ra\R$ bounded by $|\varphi(x)| \le C(1+|x|^r)$, $x\in \R$ for some constant $C \in \R$, we have
\begin{equation*}
  \int_\R \varphi(x) \mu_n(dx) \xrightarrow{n\ra \infty} \int_\R \varphi(x) \mu(dx).
\end{equation*}
\end{enumerate}
\end{lemma}

For $\theta\in\seq$, let $L_{n,\theta}$ denote the empirical measure,
\begin{equation*}
  L_{n,\theta} \doteq \frac{1}{n}\sum_{i=1}^n \delta_{\sqrt{n}\thn_i}.
\end{equation*}
The goal of this subsection is to prove the following statement that convergence of $(L_{n,\theta})_{n\in \N}$ implies a quenched LDP.

\begin{proposition}\label{prop-quegen}
Let $p\in(1,\infty)$. Let $\rho\in\mathcal{P}(\mathbb{A})$ be a probability measure on the space of triangular arrays $\mathbb{A}$, let $\nu\in\mathcal{P}_{p/(p-1)}(\R)$, and suppose that for $\rho$-a.e.\ $\theta \in \seq$, we have as $n\ra\infty$,
\begin{equation*}
\mathcal{W}_{p/(p-1)}(  L_{n,\theta}, \nu) \rightarrow 0.
\end{equation*}
Then, for  $\rho$-a.e.\ $\theta \in \seq$,   the sequence $(\wnpth)_{n\in\mathbb{N}}$  satisfies an LDP with the quasiconvex, symmetric, good rate function $\iq_{p,\nu}$ of \eqref{ieqdefn}.
\end{proposition}

We defer the proof of Proposition \ref{prop-quegen} to the end of this subsection (see p.\pageref{page-propproof}).

\begin{remark}
A slightly different approach to proving the ``product" version of Theorem \ref{th-qldp} can be found in \cite{gkr3}; that argument does not appeal to the convergence of empirical measures assumed by Proposition \ref{prop-quegen}. However, Proposition \ref{prop-quegen} has the benefit of giving a concrete interpretation of the quenched rate function $\iq_{p,\nu}$ for any $\nu\in\mathcal{P}_{p/(p-1)}(\R)$ associated with a conditioned ``environment" $\theta$.
\end{remark}

We now establish some notation and several preliminary lemmas. For  $\gamma \in \mathcal{P}(\R)$, let 
\begin{equation}\label{logmgfdef}
\mathrm{M}_\gamma(t) \doteq \int_\R e^{ty} \gamma(dy)
\end{equation}
denote the moment generating function (mgf) of $\gamma$. Let $\mathcal{T}_q$ denote the set of probability measures on $\R$ with tails dominated by the tails of $\mu_q$, in the following sense.
\begin{equation}\label{tpdef}
\mathcal{T}_q \doteq \left\{ \gamma\in\mathcal{P}(\R) : \exists \, C < \infty  \text{ s.t. } \forall\, t\in \R, \quad \log\mathrm{M}_\gamma(t) < C|t|^{q/q-1} + C  \right\}.
\end{equation}
Note that $\mathcal{T}_p \supset \mathcal{T}_q$ for $p < q$, and $\mathcal{T}_2$ consists of subgaussian measures.
\nom[mynu]{$\mathrm{M}_\gamma$}{moment generating function}
\nom[tp]{$\mathcal{T}_q$}{measures w/ $f_q$-dominated tails}

\begin{lemma}\label{lem-taildecay}
Suppose $\gamma\in\mathcal{P}(\R)$ has density $f$ and $q\in[1,\infty)$ is such  that there exist constants  $0 < c_\gamma, d_\gamma < \infty$ such that for all $|x| > d_\gamma$,
\begin{equation*}
  f(x) \le  c_\gamma e^{-c_\gamma |x|^q/q}.  
\end{equation*}
Then, $\gamma\in\mathcal{T}_q$. In particular,  for $q\in[1,\infty)$, we have $\mu_q\in\mathcal{T}_q$.
\end{lemma}

\begin{proof}
The first assertion of the lemma follows from a simple application of Young's inequality (see \cite[Lemma 2.3]{gkr3} for details). The second assertion is a simple consequence of the first.
\end{proof}

\begin{lemma}\label{lem-subgsn}
Let $p\in (1,\infty)$. For $\gamma\in\mathcal{T}_p$ and $t \in\R$, the map 
\begin{equation}\label{gammap}
  \mathcal{P}_{p/(p-1)}(\R)\ni \nu\mapsto \int_\R \log \mathrm{M}_\gamma(tu)\,\nu(du)  \in \R
\end{equation}
is continuous with respect to the Wasserstein-$\tfrac{p}{p-1}$ topology.
\end{lemma}
\begin{proof}
Fix $t \in \R$. Then, the map $u \mapsto \log M_\gamma (tu)$ is clearly 
continuous and the definition of $\mathcal{T}_p$ implies that 
\begin{equation*}
\log \mathrm{M}_\gamma(tu) < C |u|^{p/(p-1)} + C, \quad u\in \R,
\end{equation*}
for some constant $C$ depending on $t$ and $\gamma$, but not $u$. The continuity of  \eqref{gammap} with respect to the Wasserstein-$\tfrac{p}{p-1}$ topology follows from the equivalent formulation of Wasserstein convergence given by Lemma \ref{lem-wass}(3).
\end{proof}

\begin{lemma}\label{lem-tailcont}
Let $p\in [1,\infty)$, and let $\lm_p$ and $\lmq_{p,\nu}$ be as defined in \eqref{lampdefn0} and \eqref{lampdefn}, respectively. Then, 
\begin{equation*}
 \lm_p(t_1,t_2) = -\tfrac{1}{p}\log(1-pt_2) + \log \mathrm{M}_{\mu_p}\left(\tfrac{t_1}{(1-pt_2)^{1/p}}\right), \quad t_1\in \R, t_2<\tfrac{1}{p}. 
\end{equation*}
As a consequence, for $p\in(1,\infty)$, $t_1\in\R$, $t_2<\tfrac{1}{p}$, the map 
\begin{equation*}
  \mathcal{P}_{p/(p-1)}(R)\ni \nu\mapsto \lmq_{p,\nu}(t_1,t_2) \in \R
\end{equation*}
is continuous with respect to the Wasserstein-$\tfrac{p}{p-1}$ topology.
\end{lemma}
\begin{proof}
By the change of variables $x = (1-pt_2)^{1/p}y$ and the form of the density of $\mu_p$ given by \eqref{fpdef}, we write
\begin{align*}
\lm_p(t_1,t_2)   &=  \log\int_{\mathbb{R}} e^{t_1y} \tfrac{1}{2p^{1/p}\Gamma(1+\tfrac{1}{p})}   e^{-(1-pt_2) |y|^p/p} dy \\
   &=\log\left( \tfrac{1}{(1-pt_2)^{1/p}} \int_{\mathbb{R}} \exp \left(\tfrac{t_1}{(1-pt_2)^{1/p}}x\right)  \tfrac{1}{2p^{1/p}\Gamma(1+\tfrac{1}{p})}e^{-|x|^p/p} dx\right) \\
   &=-\tfrac{1}{p}\log(1-pt_2) + \log \mathrm{M}_{\mu_p}\left(\tfrac{t_1}{(1-pt_2)^{1/p}}\right).
\end{align*}
We now prove the continuity part of the lemma. Due to Lemma \ref{lem-taildecay}, $\mu_p \in \mathcal{T}_p$, and therefore by Lemma \ref{lem-subgsn}, $\nu\mapsto\int_\R \log \mathrm{M}_{\mu_p}(tu)\,\nu(du)$ is continuous for all $t\in \R$. Combined with the preceding display, this implies that $\nu\mapsto \lmq_{p,\nu}(t_1,t_2)$ is continuous for all $t_1\in\R$ and $t_2< \frac{1}{p}$.
\end{proof}

Whereas Lemma \ref{lem-tailcont} will be applied to establish the convergence of certain log mgfs, Lemma \ref{lem-lamfacts} and Lemma  \ref{lem-quenchedess} will be used to show that the limit log mgf satisfies the hypotheses of the G\"artner-Ellis theorem. We refer to Theorem 2.3.6 of \cite{DemZeibook} for a precise statement, and Definition 2.3.5 of \cite{DemZeibook} for the definition of \emph{essentially smooth}.

\begin{lemma}\label{lem-lamfacts}
Let $p\in(1,\infty)$. Then, $D_{\lm_p} = \R\times (-\infty, \frac{1}{p})$ and $\lm_p$ is strictly convex on its effective domain, lower semi-continuous, and essentially smooth. Furthermore, $\lm_p$ is symmetric in its first argument. Lastly, $\lm_p$ is non-decreasing in its second argument; that is, for fixed $t_1 \in \R$ and $t_2 < t_2'$, 
we have $\lm_p(t_1,t_2) \le \lm_p(t_1,t_2')$.
\end{lemma}

\begin{proof} 
This is a basic consequence of standard properties of mgfs and the representation of Lemma \ref{lem-tailcont}. 
\end{proof}

\begin{lemma}\label{lem-quenchedess}
Let  $p\in(1,\infty)$ and $\nu\in\mathcal{P}_{p/(p-1)}(\mathbb{R})$. Then, $\lmq_{p,\nu}$ is essentially smooth and lower semi-continuous, and $0 \in D_{\lmq_{p,\nu}}^\circ$.
\end{lemma}

\begin{proof}
Recall from Lemma \ref{lem-lamfacts} that $D_{\lm_p} = \R \times (-\infty,\frac{1}{p})$. For $(t_1,t_2)\not\in D_{\lm_p}$, note that $\lmq_{p,\nu}(t_1,t_2) = +\infty$ for all $\nu\in \mathcal{P}(\R)$. Due to Lemmas \ref{lem-subgsn} and \ref{lem-tailcont}, there exists a constant $C < \infty$ such that for all $t_1\in \mathbb{R}$ and $t_2 < \frac{1}{p}$, 
\begin{align*}
  \lmq_{p,\nu}(t_1,t_2) &\le - \tfrac{1}{p}\log\left(  1- pt_2\right)  + \int_\mathbb{R} \left( C\left|\tfrac{ t_1z}{(1- pt_2)^{1/p}}\right|^{p/(p-1)}   + C\right) \nu(dz)\\
     &= - \tfrac{1}{p}\log\left(  1- pt_2\right)  + C\tfrac{|t_1|^{p/(p-1)}}{(1- pt_2)^{1/(p-1)}} m_{p/(p-1)}(\nu)  + C< \infty.
\end{align*}
That is,
\begin{equation*}
D_{\lmq_{p,\nu}}^\circ = \mathbb{R}\times (-\infty,\tfrac{1}{p}) \ni 0.
\end{equation*}

As for essential smoothness, first note that differentiability of $\lmq_{p,\nu}$  in $D_{\lmq_{p,\nu}}^\circ$ follows from the differentiability of $(t_1,t_2)\mapsto \lm_p(t_1u,t_2)$ for all $u\in \R$ and an application of the dominated convergence theorem with the dominating function 
\begin{equation*}
  g_{t_1,t_2}(u) \doteq |\nabla\lm_p((t_1-1)u,t_2)| +  |\nabla\lm_p((t_1+1)u,t_2)|.
\end{equation*}
We refer to Lemma 3.8 of \cite{gkr3} for a similar argument in greater detail. Note by Lemma \ref{lem-lamfacts} that $\partial_{t_2} \lm_p \ge 0$, which implies
\begin{equation*}
 | \nabla \lmq_{p,\nu}(t_1,t_2)| \ge  | \partial_{t_2} \lmq_{p,\nu}(t_1,t_2) |  = \left|\int_\mathbb{R} \partial_{t_2} \lm_p(t_1u,t_2) \nu(du)\right| =  \int_\mathbb{R} \partial_{t_2} \lm_p(t_1u,t_2) \nu(du),
\end{equation*}
Then, by Fatou's Lemma, for $t' \in \mathbb{R}$, 
\begin{align*}
\liminf_{(t_1,t_2) \ra (t',  1/p)}  | \nabla \lmq_{p,\nu}(t_1,t_2)| &\ge   \liminf_{(t_1,t_2) \ra (t',  1/p)} \int_\mathbb{R} \partial_{t_2} \lm_p(t_1u,t_2) \nu(du) \\ 
   &\ge \int_\mathbb{R} \liminf_{(t_1,t_2) \ra (t',  1/p)} \partial_{t_2} \lm_p(t_1u,t_2) \nu(du) = \infty,
\end{align*}
where the last equality follows from the steepness of $\lm_p$ established in Lemma \ref{lem-lamfacts}. This shows that $\lmq_{p,\nu}$ is steep and hence, completes the proof of essential smoothness of $\lmq_{p,\nu}$.

For lower semi-continuity, suppose $(t_1^{(n)},t_2^{(n)})\ra (t_1,t_2)$ as $n\ra\infty$. Then,
\begin{equation*}
  \lmq_{p,\nu}(t_1,t_2) \le  \int_\mathbb{R} \liminf_{n\ra\infty} \lm_p(t_1^{(n)}u,t_2^{(n)}) \nu(du) \le \liminf_{n\ra\infty} \,\lmq_{p,\nu}(t_1^{(n)},t_2^{(n)}),
\end{equation*}
where the first inequality is due to the lower semi-continuity of $\lm_p$ (from Lemma \ref{lem-lamfacts}), and the second inequality is due to Fatou's Lemma.
\end{proof}

\label{page-propproof}
\begin{proof}[Proof of Proposition \ref{prop-quegen}]
We begin by proving a  $\rho$-a.e.\ LDP for the sequence $(R_\theta^{(n,p)})_{n\in\N}$ in $\R^2$, defined as
\begin{equation}\label{rnp2}
  R_\theta^{(n,p)} \doteq \left(\frac{1}{n}\sum_{i=1}^n \sqrt{n}\thn_i\ynp_i, \, \, \frac{1}{n}\sum_{i=1}^n  |\ynp_i|^p \right).
\end{equation}
Consider the G\"artner-Ellis limit log mgf: for $t=(t_1,t_2)\in \R^2$,
\begin{align*}
\lim_{n\ra\infty}\frac{1}{n}\log\E\left[\exp(n\,\langle t, R_\theta^{(n,p)}\rangle)\right] &=  \lim_{n\ra\infty} \frac{1}{n} \log \E\left[ \exp\left(\sum_{i=1}^n  t_1\sqrt{n}\thn_i \ynp_i+ t_2 |\ynp_i|^p  \right)  \right]\\
  &= \lim_{n\ra\infty} \frac{1}{n} \log \prod_{i=1}^n \E\left[ \exp\left( t_1\sqrt{n}\thn_i \ynp_i+ t_2 |\ynp_i|^p  \right)  \right]\\
  & = \lim_{n\ra\infty} \frac{1}{n} \sum_{i=1}^n \lm_p(t_1\sqrt{n}\thn_i, t_2),
\end{align*}
with $\lm_p$ given by \eqref{lampdefn0}. Due to Lemma \ref{lem-tailcont}, for all $t_1\in \R$ and $t_2<\frac{1}{p}$, the map $\nu\mapsto \int \lm_p(t_1 u,t_2) \nu(du)$ is continuous with respect to the Wasserstein-$\tfrac{p}{p-1}$ topology. Since by assumption, the empirical measure $L_{n,\theta}$ converges to $\nu$ in in the Wasserstein-$\tfrac{p}{p-1}$ topology, we have that for $\rho$-a.e.\ $\theta \in \seq$, for all $t_1\in \R$ and $t_2<\frac{1}{p}$,
\begin{equation}\label{gcllnmgf}
  \lim_{n\ra\infty} \frac{1}{n}\sum_{i=1}^n \lm_p(t_1\sqrt{n}\thn_i, \, t_2) = \int_\mathbb{R} \lm_p(t_1u,t_2)\nu(du).
\end{equation}
The same claim holds for all $t_2\ge\frac{1}{p}$, with both sides in the preceding equality valued at $+\infty$.

Due to the lower semi-continuity and essential smoothness of $\lmq_{p,\nu}$ as established in Lemma \ref{lem-quenchedess}, for $\rho$-a.e.\ $\theta\in\seq$, the G\"artner-Ellis theorem (see, e.g., \cite[Theorem 2.3.6]{DemZeibook}) yields the LDP for the sequence $(R_\theta^{(n,p)})_{n\in \N}$, with the good rate function $\lmq_{p,\nu}^*$. 

Note that $D_{\lmq_{p,\nu}^*}\subset \R \times (0,\infty)$, and the map $\bar T: \R \times (0,\infty)\ra\R$ defined by 
\begin{equation}\label{bartp}
\bar T_p(\tau_1,\tau_2) \doteq  \tau_1\tau_2^{-1/p}.
\end{equation}
is continuous. Since $\twnpz_\theta = \bar T_p (R_\theta^{(n,p)})$, we can apply the contraction principle to obtain an LDP for $(\twnpz_\theta)_{n\in \N}$  with the rate function $\iq_{p,\nu}$. Due to Lemma \ref{lem-reduction}, this implies that an identical LDP holds for $(W_\theta^{(n,p)})_{n\in\N}$.
\end{proof}

\begin{remark}
In Proposition \ref{prop-quegen}, we make the assumption $p > 1$ so that the right-hand side of \eqref{gcllnmgf} is well defined. In the case of $p=1$, the effective domain is $D_{\Lambda_1} = (-1,1) \times (-\infty, 1)$, so the integral over $\R$ on the RHS of \eqref{gcllnmgf} is infinite. This issue does not arise for $p >1$ due to Lemma \ref{lem-tailcont}.
\end{remark}

\subsection{An extension of the Glivenko-Cantelli theorem}\label{ssec-glivenko}

In view of Proposition \ref{prop-quegen}, it is natural to investigate when the empirical measure convergence holds. Recall the classical Glivenko-Cantelli theorem, which concerns weak convergence of the empirical measure of an i.i.d\ sequence. That is, for $\xi_1,\xi_2,\dots,$ i.i.d.\ with common distribution $\mu$,
\begin{equation*}
  \frac{1}{n}\sum_{i=1}^n \delta_{\xi_i} \Rightarrow \mu, \quad \P\text{-a.s.}.
\end{equation*}
In the lemma below, we state a slight extension of the Glivenko-Cantelli theorem, to triangular arrays with some dependence within rows, and to Wasserstein convergence instead of weak convergence.

\begin{lemma}\label{lem-gliv}
Let $\mu \in \mathcal{P}(\mathbb{R})$, and for $n\in\N$, suppose $(\xi^{(n)})_{n\in \N}$ is a sequence of random variables defined on a common probability space $(\Omega,\mathcal{F},\P)$ such that $\xi^{(n)}\sim \mu^{\otimes n}$. Next, let $f_n:\R^n\rightarrow \R$ be such that
\begin{equation}
  f_n(\xi^{(n)}) \xrightarrow{n\rightarrow\infty} 1, \quad \mathbb{P}\text{-a.s.} \label{convass}
\end{equation}
Let $\eta^{(n)} \doteq \xi^{(n)} / f_n(\xi^{(n)})$, and consider its   empirical measure, 
\begin{equation*}
  L_{n,\eta} \doteq \frac{1}{n}\sum_{i=1}^n\delta_{\eta_i^{(n)}}.
\end{equation*}
If $m_p(\mu) < \infty$ for some $p\in[1,\infty)$, then
\begin{equation*}
 \mathcal{W}_{p/4}( L_{n,\eta} ,\mu) \rightarrow 0, \quad \mathbb{P}\text{-a.s.}.
\end{equation*}
\end{lemma}

\begin{proof}
Let $\mathbf{F}$ be the cumulative distribution function (cdf) of $\mu$. Let $\mathbb{F}_n$ and $\mathbb{G}_n$, respectively, denote the empirical distribution functions of the samples $\xi^{(n)}$ and $\eta^{(n)}$:
\begin{align*}
  \mathbb{F}_n(t) &\doteq \frac{1}{n}\, \#\left\{ \xi_i^{(n)} \le t; \,\, i=1\,\dots,n \right\},\\
  \mathbb{G}_n(t) &\doteq \frac{1}{n}\, \#\left\{ \eta_i^{(n)} \le t ; \,\, i =1,\dots, n\right\}.
\end{align*}
First, we prove $\P$-a.s.\ weak convergence of $\mathbb{G}_n$ to $\mathbf{F}$. In other words, we prove that $\P$-a.s., for any point of continuity $t$ of $\mathbf{F}$,
\begin{equation*}
\lim_{n\rightarrow\infty} \mathbb{G}_n(t) = \mathbf{F}(t).
\end{equation*}
Note that we can decompose the preceding difference as follows:
\begin{align}
\mathbb{G}_n(t) - \mathbf{F}(t) &= \left[\mathbb{F}_n(f_n(\xi^{(n)})\, t) - \mathbf{F}(f_n(\xi^{(n)})\, t)\right] + \left[\mathbf{F}(f_n(\xi^{(n)})\,t) - \mathbf{F}(t)\right] \notag \\
   &\le \sup_{x\in \R} |\mathbb{F}_n(x) - \mathbf{F}(x) | + \left|\mathbf{F}(f_n(\xi^{(n)})\,t) - \mathbf{F}(t)\right|\label{goalweak}
\end{align}
The first term of \eqref{goalweak} converges to zero by the extension of the Glivenko-Cantelli theorem to row-wise independent triangular arrays \cite[p.106, Theorem 1]{shorack2009empirical}. The second term of \eqref{goalweak} converges to 0 due to the assumption \eqref{convass}. Therefore, we have that $L_{n,\eta}\Rightarrow \mu$, $\P$-a.s.

Next, we prove convergence of suitable moments in order to strengthen the result to Wasserstein convergence. Due to Lemma \ref{lem-wass}, it suffices to show $\P$-a.s. convergence of the $p/4$-th moments of $L_{n,\eta}$. That is, $\P$-a.s.,
\begin{equation}\label{momconv}
\lim_{n\ra\infty} m_{p/4}(L_{n,\eta}) = m_{p/4}(\mu).
\end{equation}
Note that 
\begin{equation*}
m_{p/4}(L_{n,\eta}) = \frac{1}{n}\sum_{i=1}^n |\eta_i^{(n)}|^{p/4} = \frac{\frac{1}{n}\sum_{i=1}^n |\xi_i^{(n)}|^{p/4}}{f_n(\xi^{(n)})^{p/4}}.
\end{equation*}
Due to the assumption \eqref{convass}, in order to prove \eqref{momconv}, it suffices to show that, $\P$-a.s.,
\begin{equation}\label{xipfour}
\frac{1}{n}\sum_{i=1}^n |\xi_i^{(n)}|^{p/4} \rightarrow  m_{p/4}(\mu).
\end{equation}
Note that the strong law of large numbers (SLLN) does not extend (in general) to row-wise means of i.i.d.\ triangular arrays, but a standard Borel-Cantelli argument shows that the SLLN \emph{does} hold if the common law of the i.i.d.\ elements has finite fourth moment  \cite[Example 5.41]{romano1986counterexamples}. Since $\mu$ has finite $p$-th moment, we have
\begin{equation*}
  \int_\R (|x|^{p/4})^4  \mu(dx) = m_p(\mu) < \infty,
\end{equation*}
and thus, \eqref{xipfour} holds, implying Wasserstein-$p/4$ convergence.
\end{proof}

\begin{remark} 
A weaker version of Lemma \ref{lem-gliv} can be found in \cite[p.235]{spruill2007asymptotic}, where $\mu = \mu_p$ and $f_n = n^{-1/p}\|\cdot\|_{n,p}$, so that $\eta^{(n)} \sim\text{Unif}(\bnp)$; the difference is that the statement in \cite{spruill2007asymptotic} is for convergence in probability (instead of $\P$-a.s.), and weak convergence of measures (instead of Wasserstein). 
\end{remark}

\subsection{The measure $\sigma \in \mathcal{P}(\seq)$}\label{ssec-surface}

Recall the measure $\sigma\in\mathcal{P}(\seq)$ which was assumed to satisfy \eqref{sigproj}. It remains to show how  $\sigma$  fits into the framework of Proposition \ref{prop-quegen} and Lemma \ref{lem-gliv}. To do so, we further explore the probabilistic representation for the surface measure on $\sphn$ given in Sect.\ \ref{sec-equiv}.

Let $\mathcal{R}:\mathbb{A}\rightarrow\mathbb{A}$ be the map such that for $z \in \mathbb{A}$, the $n$-th row of $\mathcal{R}(z)$ is
\begin{equation}\label{rdef}
  \mathcal{R}(z)^{(n)} \doteq \frac{z^{(n)}}{\|z^{(n)}\|_{n}}.
\end{equation}
Let $\pi_n:\mathbb{A}\rightarrow \mathbb{R}^n$ denote the coordinate map such that $\pi_n(z) = z^{(n)}$, outputting the $n$-th row of a triangular array.
\nom[RR]{$\mathcal{R}$}{Gaussian to spherical map}

\begin{definition}\label{zetdef}
Let $\zeta\in\mathcal{P}(\mathbb{A})$ be such that $\zeta \circ \pi_n^{-1}$ is the standard Gaussian measure on $\R^n$.
\nom[zzeta]{$\zeta$}{a Gaussian measure on $\mathbb{A}$}
\end{definition}

\begin{proof}[Proof of Theorem \ref{th-qldp}]
Fix $r<\infty$. Then, for $\sigma$-a.e.\ $\theta \in \seq$, we claim that $\mathcal{W}_r(L_{n,\theta},\mu_2)\ra 0$ as $n\rightarrow\infty$. The proof of the quenched LDP follows immediately from the preceding claim and Proposition \ref{prop-quegen} with $\nu=\mu_2$.

To prove the claim, first note that a straightforward application of Lemma \ref{lem-jointrep} shows that if $\sigma$ satisfies \eqref{sigproj}, then for some $\zeta$ as in Definition \ref{zetdef}, we have $\sigma = \zeta \circ \mathcal{R}^{-1}$. The upshot is that  $\sigma$-a.e.\ claims about $\theta \in \seq$ (i.e., Theorem \ref{th-qldp}) can be reduced  to $\zeta$-a.e.\ claims about $\mathcal{R}(z)$ for $z\in\mathbb{A}$.  Thus, it suffices to show that for $\zeta$-a.e\ $z\in\mathbb{A}$, we have
\begin{equation*}
\mathcal{W}_r\left(  \frac{1}{n}\sum_{i=1}^n \delta_{\sqrt{n} z_i^{(n)} / \|z^{(n)}\|_{n,2}} , \, \mu_2\right) \ra 0.
\end{equation*}
 This is a consequence of Lemma \ref{lem-gliv}, with $\mu= \mu_2$ (which has finite moments of all order) and $f_n = n^{-1/2} \|\cdot\|_{n,2}$.

\end{proof}

\subsection{Quenched proof for $p=1$}\label{ssec-qp1}

\begin{proof}[Proof of Theorem \ref{th-qldp1}]
For $\theta\in \seq$ satisfying \eqref{thmax} with  limit $c> 0$, let
\begin{equation*}
  V_\theta^{(n)} \doteq \frac{1}{n}\sum_{i=1}^n Y_i\,\sqrt{n}\theta_i^{(n)},
\end{equation*}
where $Y_1,Y_2,\dots$ are i.i.d.\ random variables with distribution $\mu_1(dy) \doteq \frac{1}{2} e^{-|y|}dy$. To prove the LDP for $(\sw_\theta^{(n,1)})_{n\in \N}$, it suffices to show that  $(V_\theta^{(n)})_{n\in\N}$ satisfies an LDP with speed $n/\sqrt{\log n}$ and the good rate function $\iq_{1,c}$ (see proof of Theorem \ref{th-aldp12} for a similar argument). In fact, due to the symmetry of $\mu_1$ and the monotonicity of $w\mapsto \iq_{1,c}(w)$ for $w > 0$, it suffices to show that for $w> 0$, we have the following upper and lower bounds:
\begin{equation}\label{ldclaim1}
 \limsup_{n\ra\infty} \frac{\sqrt{\log n}}{n}\log \P\left( V_\theta^{(n)} \ge w \right) \le  -\frac{w}{c}; \quad   \quad    \quad \liminf_{n\ra\infty} \frac{\sqrt{\log n}}{n}\log \P\left( V_\theta^{(n)} \ge w \right) \ge  -\frac{w}{c}.
 \end{equation}

First we prove the upper bound in \eqref{ldclaim1}. For $\epsilon \in(0,1)$, let
\begin{equation*}
  t_{n,\epsilon} \doteq \frac{1-\epsilon}{c(1+\epsilon)\sqrt{\log n}}.
\end{equation*}
Due to  \eqref{thmax}, for all $\epsilon >0$, there exists $N_{\epsilon} < \infty$ such that for $n\ge N_\epsilon$, we have  $ \sqrt{n}t_{n,\epsilon}  \max_{1\le i \le n} \theta_i^{(n)} \le 1-\epsilon$. Recall that for $t\in \R$, the mgf of $\mu_1$ is $\E[e^{tY_1}] = (1-t^2)^{-1}$ for $|t| <1$, and equals $+\infty$ otherwise. Combined with the Chernoff bound and the elementary bound $-\log (1-x) \le x + \tfrac{x^2}{2}$ for $x\in[0,1)$, we find that for $n\ge N_\epsilon$,
\begin{align*}
\frac{1}{nt_{n,\epsilon} } \log \P(V_\theta^{(n)} \ge w) &\le \frac{1}{nt_{n,\epsilon} } \sum_{i=1}^n -\log(1-n\,t_{n,\epsilon}^2 (\theta_i^{(n)})^2) - w \\ 
 &\le t_{n,\epsilon}  \sum_{i=1}^n (\theta_i^{(n)})^2 +\frac{\,t_{n,\epsilon}}{2} \sum_{i=1}^n (\theta_i^{(n)})^2 (\sqrt{n} t_{n,\epsilon}\theta_i^{(n)})^2 - w   \\
 &\le   t_{n,\epsilon} +\frac{\,t_{n,\epsilon}}{2}  (1-\epsilon)^2 - w.
\end{align*}
It follows that $\limsup_{n\ra\infty} \frac{\sqrt{\log n}}{n } \log \P(V_\theta^{(n)} \ge w) \le  -\frac{w(1-\epsilon)}{c(1+\epsilon)}$. Letting $\epsilon \ra 0$ yields the upper bound.

Now we prove the corresponding lower bound in \eqref{ldclaim1}. Again due to \eqref{thmax}, there exists some $N_\epsilon < \infty$ such that for $n\ge N_\epsilon$, we have $\sqrt{n} \max_{1\le i \le n} \theta_i^{(n)} \ge c(1-\epsilon) \sqrt{\log n}$. For $n\in \N$, let $j_n \doteq \arg\max_{1\le i\le n} \theta_i^{(n)}$. Then, for $n\ge N_\epsilon$, 
\begin{equation}\label{infbd1}
\P(V_\theta^{(n)} \ge w) \ge \P\left( Y_{j_n} \ge \tfrac{wn}{c(1-\epsilon)\sqrt{\log n}} \right)\cdot \P\left( \sum_{i\ne {j_n}} Y_i \sqrt{n}\theta_i^{(n)} \ge 0 \right),
\end{equation}
The second term in \eqref{infbd1} equals $1/2$, due to the symmetry of $\mu_1$. As for the first term, it follows from Lemma \ref{lem-ulbd} with $p=1$ that
\begin{equation*}
\lim_{\epsilon \rightarrow 0}  \lim_{n\ra\infty} \tfrac{\sqrt{\log n}}{n} \log \P\left( Y_{j_n} \ge \tfrac{wn}{c(1-\epsilon)\sqrt{\log n}} \right) = \lim_{\epsilon \rightarrow 0}   -\frac{w}{c(1-\epsilon)} = -\frac{w}{c}.
\end{equation*}
Combining this with \eqref{infbd1}, one obtains the lower bound.
\end{proof}

\begin{remark}\label{rmk-whysig}
Until now, we have not clarified why the condition \eqref{thmax} is natural, nor why it is not possible to make the same $\sigma$-a.e.\ claim as in  the quenched LDP for $p\in(1,\infty)$. Roughly speaking, ``almost everywhere" statements on row-wise \emph{sums} of triangular arrays are essentially identical to the corresponding statements for sequences; this is clarified in the proof of Lemma \ref{lem-gliv}, and crucial to the proof of Theorem \ref{th-qldp}, the quenched LDP for $p\in(1,\infty)$. This is not the case for  ``almost everywhere" statements on row-wise \emph{maxima} of triangular arrays, which is relevant for the $p=1$ case.

To be precise, first note that the following ``in probability" statement is classical \cite[p.430]{gnedenko1943distribution}: for any distribution on triangular arrays $\zeta\in\mathcal{P}(\mathbb{A})$ such that the law of the $n$-th row is the $n$-dimensional standard Gaussian measure (as in Definition \ref{zetdef}), and for all $\epsilon > 0$, 
\begin{equation}\label{zetprob}
\lim_{n\ra\infty} \zeta\left( z\in \mathbb{A} : \left|\tfrac{1}{\sqrt{\log n}}\max_{1\le i \le n} z_i^{(n)} - \sqrt{2} \right| > \epsilon  \right) = 0.
\end{equation}
In fact, for any $\sigma$ satisfying \eqref{sigproj}, there exists some $\zeta$ as in Definition \ref{zetdef} such that $\sigma = \zeta \circ \mathcal{R}^{-1}$ for $\mathcal{R}$ as in \eqref{rdef}. Thus, for all $\epsilon > 0$,
\begin{equation}\label{sigprob}
\lim_{n\ra\infty} \sigma\left( \theta\in \seq : \left|\sqrt{\tfrac{n}{\log n}} \max_{1\le i \le n} \theta_i^{(n)} - \sqrt{2} \right| > \epsilon  \right) = 0.
\end{equation}
The scaling in this limit motivates the condition  \eqref{thmax}.

We now consider whether the ``almost everywhere" version of \eqref{zetprob} is satisfied:
\begin{equation}\label{zetae}
 \zeta\left( z\in \mathbb{A} : \lim_{n\ra\infty} \tfrac{1}{\sqrt{\log n}}\max_{1\le i \le n} z_i^{(n)} = \sqrt{2}  \right) \stackrel{?}{=} 1.
\end{equation}
The equality in \eqref{zetae} holds for \emph{some} $\zeta\in\mathcal{P}(\mathbb{A})$ satisfying Definition \ref{zetdef}, but is not satisfied for others. 
\begin{enumerate}[label=(\alph*)]
\item Suppose $\zeta$ is such that $\zeta(z\in\mathbb{A} : z_i^{(n)} = z_i^{(i)}, \forall\,  i,n\in \N)$.  That is, for $\zeta$-a.e.\ $z$, the array is constant within columns. Then, the maximum of the $n$-th row of the array $z$ is equivalent to the maximum of the first $n$ terms of the sequence $z_1^{(1)},z_2^{(2)}, \dots$. Under this law, the $\zeta$-a.e.\ convergence in \eqref{zetae} is known to hold \cite[Remark (viii)]{resnick1973almost}.
\item On the other hand, suppose $\zeta$ is such that for a random triangular array $Z\sim \zeta$, the rows of $Z$ are independent (and hence, the elements of $Z$ are i.i.d.\ standard Gaussian random variables). Then, the limit \eqref{zetae} can be shown \emph{not} to hold since the $\zeta$-a.e.\  limit inferior and limit superior differ \cite[p.123]{jiang2005maxima}. In particular, it is possible to show that for $\zeta$-a.e. $z$, all of the points in $[\sqrt{2},2]$ are limit points of the sequence $\max_{1\le i \le n}z_i^{(n)} / \sqrt{\log n}$, $n\in \N$.
\end{enumerate}
Similarly, the ``almost everywhere" analog of \eqref{sigprob} holds for some $\sigma$ satisfying \eqref{sigproj}, but not others. Recall the map $\mathcal{R}$ of \eqref{rdef}, and for $\zeta$ satisfying Definition \ref{zetdef}, let $\sigma = \zeta \circ \mathcal{R}^{-1}$.
\begin{enumerate}[label=(\alph*')]
\item If $\zeta$  is as in example (a) above,  then condition \eqref{thmax} of  Theorem \ref{th-qldp1} holds for $\sigma$-a.e\ $\theta\in \seq$, with $c=\sqrt{2}$.
\item If $\zeta$  is as in example (b) above, then the proof of Theorem \ref{th-qldp1} (which goes through for subsequences) shows that  for $\sigma$-a.e.\ $\theta\in \seq$, the sequence $(\sqrt{\log n}/n)\log \P(V_\theta^{(n)} \ge w)$ has all of the points in $[-w/2,-w/\sqrt{2}]$ as limit points, and hence, does not converge.
\end{enumerate}
The upshot of the two preceding examples is that, unlike for the quenched LDP when $p\in(1,\infty)$, it is not possible to state Theorem \ref{th-qldp1} as a result for $\sigma$-a.e.\ $\theta\in \seq$ and any $\sigma$ satisfying \eqref{sigproj}. Instead, the large deviation behavior of $(W_\theta^{(n,1)})_{n\in\N}$ depends on the particular sequence $\theta$ of projection directions, via the limit \eqref{thmax}.
 \end{remark}

\section{The relationship between the annealed and quenched LDPs}\label{sec-rel} 

Fix $p\in (2,\infty)$. In this section, we prove Theorem \ref{th-compar}, which establishes a connection between the quenched rate function $\iq_{p,\nu}$ and  the annealed rate function $\ia_p$. Additional analysis of this variational problem is deferred to Sect.\ \ref{sec-analysis}.

In Sect.\ \ref{sec-quenched}, we obtained the quenched rate function by establishing an LDP for $R_\theta^{(n,p)}$ of \eqref{rnp2} and then using the fact that $\twnpz_\theta = \bar T_p(R_\theta^{(n,p)})$, where $\bar T_p:\R\times\R_+ \ra \R$ is the map defined in \eqref{bartp}. To establish the variational formula \eqref{varform1}, we will find it convenient to use an exactly analogous representation for the annealed case (as opposed to the approach originally adopted in Sect.\ \ref{sec-annealed}). Let $R^{(n,p)}$ be defined similarly to $R_\theta^{(n,p)}$ of \eqref{rnp2}, but with the deterministic deterministic $\thn$ replaced by random $\bthn$,
\begin{equation}\label{rnpdefn}
  R^{(n,p)} \doteq \left(\frac{1}{n}\sum_{i=1}^n \sqrt{n}\bthn_i\ynp_i, \, \, \frac{1}{n}\sum_{i=1}^n  |\ynp_i|^p \right).
\end{equation}
Then, we have
\begin{equation*}
  \wnpbth \eqdist \bar T_p(R^{(n,p)}).
\end{equation*}
We will prove an LDP for $(\bar T_p(R^{(n,p)}))_{n\in\N}$, and use it to obtain an alternate form for the annealed LDP that directly relates the annealed and quenched rate functions.

In Sect.\ \ref{ssec-rnp}, we establish an LDP for $(R^{(n,p)})_{n\in \N}$ using certain spherical invariance properties similar to those discussed in Sect.\ \ref{ssec-p2}. Then, in Sect.\ \ref{ssec-empcone}, we recall a large deviation principle for the empirical measure induced by the coordinates of a random point on the scaled $\ell^q$ sphere $n^{1/q}\partial \mathbb{B}_{n,q}$. Lastly, in Sect.\ \ref{ssec-varadh}, we apply the aforementioned empirical measure LDP in order to obtain variational formulas for the limit log mgfs associated with $R^{(n,p)}$. Here, we will repeatedly make use of the tail bounds obtained in Lemma \ref{lem-subgsn}.

\subsection{An LDP for $(R^{(n,p)})_{n\in \N}$  with a convex rate function}\label{ssec-rnp}

In this subsection, we prove that $(R^{(n,p)})_{n\in\N}$ satisfies an LDP with some convex good rate function. For our purposes, although the explicit form of the rate function is irrelevant, its convexity is important. We begin with two elementary lemmas involving convex analysis.

\begin{lemma}[Theorem 5.3 or comment on p.54 of \cite{rockafellar1970convex}] \label{lem-infconvex}
Let $\mathcal{X}$,$\mathcal{Y}$ be real vector spaces. Let $D_F\subset \mathcal{X}\times \mathcal{Y}$ be a convex set, and suppose $F:D_F\ra \R$ is a convex function. Let
\begin{equation*}
  \tilde{F}(x) \doteq \inf_{y \in \mathcal{Y}: (x,y)\in D_F} F(x,y).
\end{equation*}
 Then, $\tilde{F}$ is a convex function.
\end{lemma}

\begin{lemma}\label{lem-jointcon}
The map
\begin{equation*}
  \R^2 \ni (x,y) \mapsto J_2(\tfrac{x}{y^{1/2}}) = -\tfrac{1}{2}\log(1-\tfrac{x^2}{y}) \in \R
\end{equation*}
is convex on its domain $\{(x,y) \in \R^2 : y > x^2\}$. 
\end{lemma}

\begin{proof}
Let $f(x,y) \doteq -\tfrac{1}{2}\log(1-\tfrac{x^2}{y})$. We compute the Hessian matrix. 
\begin{align*}
(Hf)(x,y) &= \frac{1}{(y-x^2)^2}\begin{pmatrix}
 y+x^2 &  -x \\
 -x & \frac{1}{2y^2}x^2(2y-x^2)
 \end{pmatrix}.
\end{align*}
Note that for $(x,y)$ such that $y > x^2$,
\begin{equation*}
  \det (Hf) = \frac{1}{(y-x^2)^4}\frac{x^4}{2y^2} \left( y- x^2\right) > 0,
\end{equation*}
and also 
\begin{equation*}
  \frac{y+x^2}{(y-x^2)^2} > 0.
\end{equation*}
By Sylvester's criterion, since all leading principal minors are positive, $Hf$ is a positive definite matrix, so $f$ is convex.
\end{proof}

Next, we exploit the spherical symmetry of $\bthn$ in the following lemma, as we did previously in Sect. \ref{ssec-p2}, which will then allow us to prove the desired LDP.

\begin{lemma}\label{lem-spherindep}
Fix $n\in \N$, and let $X^{(n)}=(X_1,\dots, X_n)$ be a random vector in $\R^n$ independent of $\bthn$ which is uniformly distributed on $\sphn$.  Then, 
\begin{equation*}
 \left\langle \sqrt{n}\bthn, \tfrac{X^{(n)}}{\|X^{(n)}\|_{n,2}}\right\rangle_n  \eqdist  \left\langle \sqrt{n}\bthn, e_1^{(n)}\right\rangle_n.
\end{equation*}
Moreover, $\langle \sqrt{n}\bthn, X^{(n)}/\|X^{(n)}\|_{n,2}\rangle_n$ is independent of $X^{(n)}$. 
\end{lemma}

\begin{proof}
Due to the spherical symmetry of $\sqrt{n}\bthn$ and since $X^{(n)}/\|X^{(n)}\|_{n,2} \in \sphn$,
\begin{equation}\label{margeq}
 \left\langle \sqrt{n}\bthn, \tfrac{X^{(n)}}{\|X^{(n)}\|_{n,2}} \right\rangle_n \eqdist \langle \sqrt{n}\bthn, x\rangle_n \eqdist \langle \sqrt{n}\bthn, e_1^{(n)}\rangle_n,
 \end{equation}
 for any $x\in \sphn$. It remains to show independence. Let $\pi(\cdot,\cdot)$ denote the joint distribution of $\left( \tfrac{X^{(n)}}{\|X^{(n)}\|_{n,2}}, X^{(n)}\right)$, with first and second marginals $\pi_1$ and $\pi_2$, respectively. For $A \in \mathcal{B}(\R)$ and $B \in \mathcal{B}(\R^n)$,
\begin{align*}
\P\left( \left\langle \sqrt{n}\bthn, \tfrac{X^{(n)}}{\|X^{(n)}\|_{n,2}}\right \rangle_n \in A, X^{(n)} \in B\right ) &= \int_{\R \times \R^n} \P( \langle \sqrt{n}\bthn, x_1\rangle_n \in A) \1_{\{x_2\in B\}} \, \pi(dx_1,dx_2)\\
   &= \int_{\R \times B} \P( \langle \sqrt{n}\bthn, e_1^{(n)}\rangle_n \in A) \, \pi(dx_1,dx_2)\\
   &= \P( \langle \sqrt{n}\bthn, e_1^{(n)}\rangle_n \in A)\, \P(X^{(n)} \in B),\\
   &=\P\left( \left\langle \sqrt{n}\bthn, \tfrac{X^{(n)}}{\|X^{(n)}\|_{n,2}}\right \rangle_n \in A\right) \,\P(X^{(n)} \in B),
\end{align*}
where the second and last equality follow from \eqref{margeq}.
\end{proof}

\begin{proposition}\label{prop-altldpconvex}
Let $p\in(2,\infty)$. Then, the sequence $(R^{(n,p)})_{n\in \N}$ defined by \eqref{rnpdefn}   satisfies an LDP with a convex good rate function. 
\end{proposition}
\begin{proof}
Due to the independence given by Lemma \ref{lem-spherindep},
\begin{align}
 R^{(n,p)} &=  \left(\frac{1}{n}\sum_{i=1}^n \sqrt{n}\bthn_i \frac{\ynp_i}{\|Y^{(n,p)}\|_{n,2}} \|Y^{(n,p)}\|_{n,2}, \frac{1}{n}\sum_{i=1}^n  |\ynp_i|^p \right)\notag\\
  &= \left(\frac{1}{n}\left\langle  \sqrt{n}\,\bthn,  \tfrac{\ynp}{\|Y^{(n,p)}\|_{n,2}} \right\rangle_n  \|Y^{(n,p)}\|_{n,2}, \frac{1}{n}\sum_{i=1}^n  |\ynp_i|^p \right)\notag\\ 
  &\eqdist \left(\frac{1}{n}\sqrt{n}\,\bthn_1 \|Y^{(n,p)}\|_{n,2}, \frac{1}{n}\sum_{i=1}^n  |\ynp_i|^p \right).\label{rnprep}
\end{align}
Define the following $\R^3$-valued sequence of random variables,
\begin{equation*}
  Q^{(n,p)} \doteq \left(\bthn_1,  \frac{1}{n}\sum_{i=1}^n |\ynp_i|^2, \frac{1}{n}\sum_{i=1}^n  |\ynp_i|^p \right), \quad n \in \N.
\end{equation*}
By Cram\'er's theorem in $\mathbb{R}^2$, the sequence $Q_{2,3}^{(n,p)} \doteq \left(  \frac{1}{n}\sum_{i=1}^n |\ynp_i|^2, \frac{1}{n}\sum_{i=1}^n  |\ynp_i|^p \right)$, $n\in\N$, satisfies an LDP with some convex good rate function, call it $\widehat{J}_p$, with domain $D_{\widehat{J}_p} = \mathbb{R}_+^2$. As obtained in \cite{BarGamLozRou10} and described in Sect.\ \ref{ssec-p2}, $(\bthn_1)_{n\in \N}$ satisfies an LDP with the convex good rate function $J_2(a) = -\frac{1}{2}\log(1-a^2)$ for $|a| <1$ (and $+\infty$ elsewhere). Since $\bthn$ and $\ynp$ are independent, the sequence $(Q^{(n,p)})_{n\in \N}$ satisfies an LDP with the convex good rate function 
\begin{equation*}
J_{Q,p}(a,b,c) \doteq J_2(a) + \widehat{J}_p(b,c)\,, \quad a,b,c\in \R .
\end{equation*}
By \eqref{rnprep} and the contraction principle, $(R^{(n,p)})_{n\in \N}$ satisfies an LDP with the good rate function $J_{R,p}$ defined as follows: for $x\in \R$ and $z \ge 0$,
\begin{align*}
J_{R,p}(x,z) &\doteq \inf \left\{ J_{Q,p}(a,b,c) : |a| <1, b \ge 0, c \ge 0, x= ab^{1/2} , z = c \right\}\\ 
  &= \inf_{y:y>x^2 \ge 0} J_{Q,p}(\tfrac{x}{y^{1/2}},y,z).
\end{align*}
We now show that $J_{R,p}$ is convex. By Lemma \ref{lem-infconvex}, it suffices to prove that
\begin{equation*}
  (x,y,z) \mapsto J_{Q,p}(\tfrac{x}{y^{1/2}}, y,z) = -\tfrac{1}{2}\log(1-\tfrac{x^2}{y}) + \widehat{J}_p(y,z) \,, \quad 0 \le x^2 < y;
\end{equation*}
is (jointly) convex, which follows from Lemma \ref{lem-jointcon}, the convexity of $\widehat{J}_p$, and the fact that the sum of two convex functions is convex.

\end{proof}

\subsection{LDP for the empirical measure under the cone measure on $n^{1/q}\partial\mathbb{B}_{n,q}$}\label{ssec-empcone}

The connection between the annealed and quenched LDPs will make critical use of a particular LDP for the following empirical measures. Let $L_{n,\Theta}$ denote the empirical measure of $\sqrt{n}\bthn$,
\begin{equation}\label{empir}
  L_{n,\Theta} \doteq \frac{1}{n}\sum_{i=1}^n\delta_{\sqrt{n}\bthn_i}.
\nom[lnalpha]{$L_{n,\Theta}$}{empirical measure of  $\sqrt{n}\bthn$}
\end{equation}
In Proposition \ref{prop-sanovcone} below, we state a Sanov-type LDP for this sequence of empirical measures, with the rate function $\mathbb{H}:\mathcal{P}(\mathbb{R})\ra[0,\infty]$ defined to be a perturbed version of relative entropy: for $\nu\in\mathcal{P}(\R)$, let 
 \begin{equation} \label{hpdefn}
  \mathbb{H}(\nu) \doteq \left\{ \begin{array}{ll}
 H(\nu | \mu_2) + \tfrac{1}{2} (1 - m_2(\nu)) & \textnormal{ if } m_2(\nu) \le 1, \\
  +\infty & \textnormal{ else,} 
 \end{array}\right.
 \nom[hq]{$\mathbb{H}$}{rate function for $(L_{n,\Theta})_{n\in \N}$}
  \end{equation}
  where $m_2(\nu)$ is the second moment of $\nu$.

\begin{proposition}\label{prop-sanovcone}
Let $r<2$. Then, the empirical measure $(L_{n,\Theta})_{n\in N}$ satisfies an LDP in $\mathcal{P}_r(\R)$ (equipped with the Wasserstein-$r$ topology) with the strictly convex good rate function $\mathbb{H}$ of \eqref{hpdefn}.
\end{proposition}

This LDP can be found in \cite[Theorem 6.6]{arous2001aging} with respect to the weak topology. A strengthening to the Wasserstein topology (and in fact, a mild extension to the surface measure on $\ell^q$ spheres for $q\in[1,\infty]$ other than $q=2$) can be found in \cite[Theorem 1.4]{kr1}.

\subsection{Application of Varadhan's integral formula}\label{ssec-varadh} In this section, in order to obtain an expression for the rate function, we will apply the G\"artner-Ellis theorem. In view of this, we introduce the limit log mgf $\lmbar_p : \R^2 \ra \R$. For $t_2 \ge \frac{1}{p}$, let $\lmbar_p(t_1,t_2) \doteq  +\infty$  and for $t_1\in \R$, $t_2 < \frac{1}{p}$, let
\begin{equation} \label{pressfunc}
  \lmbar_{p}(t_1,t_2) \doteq \lim_{n\ra\infty} \frac{1}{n} \log \mathbb{E} \left[\exp\left( \sum_{i=1}^n \left( t_1 \sqrt{n}\bthn_i \ynp_i + t_2 |\ynp_i|^p\right) \right)\right],
\end{equation}
where $\sqrt{n}\bthn$ is  distributed according to the cone measure on $n^{1/2}\partial \mathbb{B}_{n,2}$ (i.e., the rotationally invariant probability measure on $n^{1/2} \sphn$). Before applying Varadhan's lemma, we introduce the following technical lemma.
\nom[phibar]{$\lmbar_p$}{limit log mgf}

\begin{lemma}[Theorem 2.11(2) of \cite{BarGamLozRou10}]\label{lem-subindep}
For all $n\in \N$, the collection of random variables $(|\bthn_1|,\dots, |\bthn_n|)$ is sub-independent. That is, for non-negative non-decreasing functions $g_1,\dots,g_n$,
\begin{equation*}
  \E\left[\prod_{i=1}^n g_i(|\bthn_i|)\right] \le \prod_{i=1}^n \E\left[g_i(|\bthn_i|)\right].
\end{equation*}
\end{lemma}

Using the preceding technical result, the following lemma introduces the connection between the limit log mgfs of \eqref{pressfunc} and \eqref{pressfunc2} and the entropy-like rate function $\mathbb{H}$ of \eqref{hpdefn}.

\begin{lemma} \label{lem-lmgfvar}
Let $p\in (2,\infty)$. Then, 
\begin{equation} \label{mgfvar}
   \lmbar_{p}(t_1,t_2) = \sup_{\nu \in \mathcal{P}(\mathbb{R})} \left\{ \lmq_{p,\nu}(t_1,t_2) - \mathbb{H}(\nu)\right\}, \quad t_1,t_2\in \R.
\end{equation}
\end{lemma}

\begin{proof}
The equality in \eqref{mgfvar} is clear for $t_2 \ge \frac{1}{p}$, since then both sides of \eqref{mgfvar} equal $+\infty$. Thus, fix $t_1\in \R$, $t_2 < \frac{1}{p}$. Conditioning on $\Theta$, and using the assumed independence of $\Theta$ and $\bolyp$, as well as the definition of $\lm_p$ from \eqref{lampdefn0},  the expectation on the right-hand side of \eqref{pressfunc} can be rewritten as
\begin{align*}
 \lmbar_{p}(t_1,t_2) &= \lim_{n\ra\infty} \frac{1}{n} \log \mathbb{E} \left[\prod_{i=1}^n\mathbb{E}\left[\left.  \exp\left( \left( t_1 \sqrt{n}\bthn_i \ynp_i + t_2 |\ynp_i|^p\right) \right) \right| \sqrt{n}\bthn \right]\right]  \\
  &= \lim_{n\ra\infty} \frac{1}{n} \log \mathbb{E}\left[\exp\left(   \sum_{i=1}^n \lm_p(t_1\sqrt{n}\bthn_i, t_2) \right)\right]\\
  &= \lim_{n\ra\infty} \frac{1}{n} \log \E\left[ \exp(n\psi_{p,t_1,t_2}(L_{n,\Theta}))\right],
\end{align*}
where $\psi_{p,t_1,t_2}:\mathcal{P}(\R) \ra \R$ is defined as
\begin{equation*}
  \psi_{p,t_1,t_2} (\nu) \doteq \int\lm_p(t_1 a,t_2)\nu(da) = \lmq_{p,\nu}(t_1,t_2), \quad \nu\in\mathcal{P}(\R).
\end{equation*}
Recall from Proposition \ref{prop-sanovcone} that for all $r <2$, the sequence $(L_{n,\Theta})_{n\in \N}$ satisfies an LDP in $\mathcal{P}_r(\R)$ equipped with the Wasserstein-$r$ topology, with the good rate function $\mathbb{H}$. Thus, the variational formula \eqref{mgfvar} would follow from Varadhan's integral formula  \cite[Theorem 4.3.1]{DemZeibook} if we can show that the following hypotheses hold: \begin{enumerate}[label=(\alph*)]
\item for some $r<2$,  $\psi_{p,t_1,t_2}$ is continuous with respect to the Wasserstein-$r$ topology; 
\item for some $\kappa >1$, $\psi_{p,t_1,t_2}$ satisfies the exponential moment condition 
\begin{equation}\label{expmomcond}
  \limsup_{n\ra\infty} \frac{1}{n} \log \E\left[ e^{\kappa n \psi_{p,t_1,t_2} (L_{n,\Theta})} \right] < \infty.
\end{equation}
\end{enumerate}

We first check  condition (a). The continuity of $\psi_{p,t_1,t_2}$   with respect to the Wasserstein-$\tfrac{p}{p-1}$ topology follows from Lemma \ref{lem-tailcont}. Condition (a) follows since for $p>2$, we have $\frac{p}{p-1} < 2$.

We now establish a strong version of condition (b) that shows the exponential moment is finite for any $\kappa >1$. Let $\kappa > 1$. Because $\mu_p$ is symmetric, $\lm_p$ of \eqref{lampdefn0} is symmetric in its first argument, so $\lm_p(t_1a,t_2)$ depends on $a$ only through $|a|$. Moreover, for fixed $t_1\in \R$ and $t_2 < \frac{1}{p}$, the mapping $|a| \mapsto \lm_p(t_1|a|,t_2)$ is non-negative and non-decreasing, as can be seen from the expression for $\lm_p$ given in Lemma  \ref{lem-tailcont}. Thus,  the sub-independence property of Lemma \ref{lem-subindep} and the definition \eqref{tpdef} of $\mathcal{T}_p$ imply that for a constant $C_{p,t_1,t_2}$ not depending on $n,\kappa, \bthn$,
 \begin{align}
 \E\left[\exp\left(\kappa n \psi_{p,t_1,t_2}(L_{n,\Theta}) \right)\right]  &\le \prod_{i=1}^n \E\left[ \exp(\kappa \lm_p(t_1\sqrt{n}\bthn_i, t_2) )\right] \notag\\
    &\le \prod_{i=1}^n \E\left[ \exp( \kappa C_{p,t_1,t_2} +  \kappa C_{p,t_1,t_2} |\sqrt{n}\bthn_i|^{p/(p-1)} )\right] \notag\\
    &=  \exp(n\kappa C_{p,t_1,t_2}) \, \E\left[ \exp(  \kappa C_{p,t_1,t_2} |\sqrt{n}\bthn_1|^{p/(p-1)} )\right]^n.\label{expcond}
\end{align}
Let $(Z_1,Z_2,\dots)$ be a sequence of i.i.d. standard Gaussian random variables, and note that due to Lemma \ref{lem-jointrep} and the strong law of large numbers,
\begin{equation*}
  \sqrt{n}\bthn_1 \eqdist \frac{\sqrt{n} Z_1}{\|Z^{(n)}\|_{n,2}}\xrightarrow[\textnormal{$\P$-a.s.}]{ n\ra\infty } Z_1.
\end{equation*}
Applying logarithms, dividing by $n$, and taking limits in \eqref{expcond}  shows that for $r=\frac{p}{p-1}$,
\begin{align*}
   \limsup_{n\ra\infty} \frac{1}{n} \log  \E\left[\exp\left(\kappa n \psi_{p,t_1,t_2} (L_{n,\Theta}) \right)\right]  &\le  \kappa  C_{p,t_1,t_2}  +  \limsup_{n\ra\infty} \log  \E\left[ \exp(  \kappa C_{p,t_1,t_2} |\sqrt{n}\bthn_1|^{r} )\right]\\
   &\le   \kappa C_{p,t_1,t_2}+  \log  \E\left[ \exp(  \kappa C_{p,t_1,t_2} |Z_1|^{r} )\right] < \infty,
\end{align*}
where the interchange of $\limsup$ and expectation is due to Fatou's lemma, and the last display is finite since $r < 2$ for $p > 2$.
\end{proof}

In the following two lemmas, we establish some properties of the minimizers of the variational problem of Lemma \ref{lem-lmgfvar}. We later massage these results to obtain the variational formula of Theorem \ref{th-compar}.

\begin{lemma}[{\cite[Lemma 2.4]{kr1}}]\label{lem-compact}
Let $K_2 \doteq \{ \nu \in \mathcal{P}(\mathbb{R}) : m_2(\nu) \le 1 \}$. The set $K_2$ is convex, non-empty, and  compact with respect to the Wasserstein-$r$ topology for all $r <2$.
\end{lemma}

\begin{lemma}\label{lem-optprops}
Let $p\in (2,\infty)$ and for fixed $(t_1,t_2)\in\R^2$, let $\phi:\mathcal{P}(\R)\ra\R$ denote the functional being maximized in \eqref{mgfvar} ,
\begin{equation*}
  \phi(\nu) \doteq \lmq_{p,\nu}(t_1,t_2) - \mathbb{H}(\nu).
\end{equation*}
Then, $\phi$ is strictly concave and upper semi-continuous (with respect to the Wasserstein-$\tfrac{p}{p-1}$ topology on $\mathcal{P}_{p/(p-1)}(\R)$). As a consequence, the supremum  in \eqref{mgfvar} is uniquely attained at some optimal $\nu^\circ$ such that $m_2(\nu^\circ) \le 1$.
\end{lemma}
\begin{proof}
From the definition \eqref{hpdefn}, it follows that the domain of $\mathbb{H}$ is the compact set $K_2$ of Lemma \ref{lem-compact}, so it suffices to restrict the supremum in the variational problem \eqref{mgfvar} to $K_2\subset\mathcal{P}(\R)$. For $\nu \in K_2$, we see that  $\phi$ is the sum of a linear functional $\nu\mapsto \lmq_{p,\nu}(t_1,t_2)$, and the negative of the strictly convex rate function $\mathbb{H}$ of \eqref{hpdefn}. As for upper semi-continuity, first note that $\nu\mapsto \lmq_{p,\nu}(t_1,t_2)$ is continuous due to Lemma \ref{lem-tailcont}. Since $\frac{p}{p-1} < 2$ for $p>2$, it follows from Proposition \ref{prop-sanovcone} that  $-\mathbb{H}$ is upper semi-continuous with respect to Wasserstein-$\tfrac{p}{p-1}$. This shows that $\phi$ is strictly concave and upper semi-continuous on the compact convex non-empty set $K_2$, so the supremum of $\phi$ is uniquely attained on $K_2$.
\end{proof}

\begin{theorem}[Minimax Theorem, see Corollary 3.3 of \cite{sion1958general}]\label{th-minmax} 
Let $\mathcal{X},\mathcal{Y}$ be topological vector spaces. Suppose $C\subset \mathcal{X}$ is a compact convex nonempty subset, and $D\subset \mathcal{Y}$ is a convex subset. Let $F:\mathcal{X}\times \mathcal{Y} \ra \mathbb{R}$ be a function such that:
\begin{itemize}
\item for all $y\in D$, $F(\cdot, y)$ is lower semi-continuous and convex on $C$;
\item for all $x\in C$, $F(x,\cdot)$ is upper semi-continuous and concave on $D$.
\end{itemize}
Then,
\begin{equation*}
\inf_{x\in C} \sup_{y\in D} F(x,y) =   \sup_{y\in D} \inf_{x\in C} F(x,y).
\end{equation*}

\end{theorem}

\begin{lemma} \label{lem-varforbarstar}
Let $p\in (2,\infty)$. Then, for $\tau_1,\tau_2\in\mathbb{R}$ ,
\begin{equation}\label{mimap}
  \lmbar_{p}^*(\tau_1,\tau_2) = \inf_{\nu \in \mathcal{P}(\mathbb{R})}  \left\{ \lmq_{p,\nu}^*(\tau_1,\tau_2) + \mathbb{H}(\nu)\right\}.
\end{equation}
\end{lemma}
\begin{proof} 
 We apply the Minimax Theorem \ref{th-minmax} to the following:
\begin{itemize}
\item $\mathcal{X} = \mathcal{M}(\mathbb{R})$, the space of finite signed measures on $\mathbb{R}$, equipped with the Wasserstein-$\tfrac{p}{p-1}$ topology;
\item $\mathcal{Y} = \mathbb{R}^2$;
\item $C = K_2 = \{ \nu \in \mathcal{X} : \nu \in \mathcal{P}(\mathbb{R}), m_2(\nu) \le 1\}$;
\item $D =  \mathbb{R} \times (-\infty,\tfrac{1}{p})$.
\item Fix $\tau_1,\tau_2$. For $\nu \in C$ and $(t_1,t_2) \in D$, let
\begin{equation} \label{minmaxf}
F(\nu, (t_1,t_2)) \doteq t_1\tau_1 + t_2\tau_2 -  \lmq_{p,\nu}(t_1,t_2) + \mathbb{H}(\nu),
\end{equation}
where $\lmq_{p,\nu}$ and $\mathbb{H}$ are defined as in \eqref{lampdefn} and \eqref{hpdefn}, respectively, for $\nu\in\mathcal{P}(\R)$, and set equal to $+\infty$ for $\nu \in \mathcal{M}(\R) \setminus \mathcal{P}(\R)$. 
\end{itemize}
It is clear that $\mathcal{X},\mathcal{Y}, D$ satisfy the hypotheses of the minimax theorem. The hypotheses for $C=K_2$ follow from Lemma \ref{lem-compact}, since $\frac{p}{p-1}< 2$.

To verify the desired properties of $F$, we first fix $(t_1,t_2) \in D$. Then, the lower semi-continuity and convexity of $F(\,\cdot\,,(t_1,t_2))$ follow from Lemma \ref{lem-optprops}. Next, fix $\nu \in C$.  Lower semi-continuity of $\lmq_{p,\nu}$ follows from Lemma \ref{lem-quenchedess}. As for convexity, Lemma \ref{lem-lamfacts} says that $\lm_p$ is convex on $D$, and hence, by linearity of expectation, $\lmq_{p,\nu}$ is convex on $D$. Since $(t_1,t_2) \mapsto t_1\tau_1 + t_2\tau_2$ is continuous and linear, it follows that $F(\nu, \cdot)$ is upper semi-continuous and concave on $D$.

Lastly, substitute the representation obtained in Lemma \ref{lem-lmgfvar} into the expression for the Legendre transform $\lmbar_{p}^*$, and then apply Theorem \ref{th-minmax} to $F$ as defined in \eqref{minmaxf}. 
\begin{align*}
\lmbar_{p}^*(\tau_1,\tau_2) &= \sup_{t_1\in \R, t_2\in R} \left\{t_1\tau_1 + t_2\tau_2 - \lmbar_{p}(t_1,t_2) \right\}\\ 
   &= \sup_{t_1 \in \R, t_2 \in \R} \left\{t_1\tau_1 + t_2\tau_2 - \sup_{\nu \in \mathcal{P}(\mathbb{R})} \left\{ \lmq_{p,\nu}(t_1,t_2) - \mathbb{H}(\nu)\right\} \right\}\\
   &= \sup_{(t_1,t_2) \in D} \inf_{\nu \in C} \left\{t_1\tau_1 + t_2\tau_2 -  \lmq_{p,\nu}(t_1,t_2) + \mathbb{H}(\nu)\right\} \\
   &=  \inf_{\nu \in C} \sup_{ (t_1,t_2) \in D} \left\{t_1\tau_1 + t_2\tau_2 -  \lmq_{p,\nu}(t_1,t_2) + \mathbb{H}(\nu)\right\} \\
   &= \inf_{\nu \in \mathcal{P}(\mathbb{R})}  \left\{ \lmq_{p,\nu}^*(\tau_1,\tau_2) + \mathbb{H}(\nu)\right\},
\end{align*}
where the third and fifth equalities hold since  $\lmq_{p,\nu}(t_1,t_2) = +\infty$ for $t_2 > \frac{1}{p}$, and $\mathbb{H}(\nu) = +\infty$ if either $\nu \in \mathcal{M}(\R) \setminus \mathcal{P}(\R)$ or  $m_2(\nu) > 1$.
 \end{proof}

\begin{lemma}\label{lem-var2}
The variational formula \eqref{varform1} of Theorem \ref{th-compar} holds for $p=2$, and the infimum is attained at $\nu=\mu_2$.
\end{lemma}
\begin{proof}
For $p=2$, it follows from elementary calculations that for $w\in \R$ such that $w^2 \ge m_2(\nu)$, we have $\iq_{2,\nu}(w) = +\infty$, and for $w^2 < m_2(\nu)$,
\begin{align*}
\iq_{2,\nu}(w) &= -\tfrac{1}{2}\log(1-\tfrac{w^2}{m_2(\nu)}).
\end{align*}
It is clear that for all $w\in \R$, $\iq_{2,\nu}(w)$ is non-increasing  in $m_2(\nu)\in(w^2,\infty]$. Observe from \eqref{hpdefn} that $\mathbb{H}(\nu)\ge 0$, with equality if and only if $\nu = \mu_2$. Hence, for $w\in \R$,
\begin{equation*}
\iq_{2,\mu_2}(w) = \iq_{2,\mu_2}(w) + \mathbb{H}(\mu_2) \ge  \inf_{\substack{\nu\in \mathcal{P}(\R):\\ m_2(\nu)\le 1}} \left\{ \iq_{2,\nu}(w) + \mathbb{H}(\nu) \right\} \ge  \inf_{\substack{\nu\in \mathcal{P}(\R):\\ m_2(\nu)\le 1}}\iq_{2,\nu}(w)  = \iq_{2,\mu_2}(w).
\end{equation*}
Thus, $\mu_2$ minimizes the variational formula \eqref{varform1}.
\end{proof}

\begin{proof}[Proof of Theorem \ref{th-compar}]

For $p=2$, the theorem follows from Lemma \ref{lem-var2}. As for $p\in(2,\infty)$ consider the quantity $R^{(n,p)}$  of \eqref{rnpdefn}. By Proposition \ref{prop-altldpconvex}, the sequence $(R^{(n,p)})_{n\in\N}$ satisfies an LDP with a convex good rate function, which we denote here by $J_{R,p}$. Note that $\lmbar_p$ of \eqref{pressfunc} satisfies
\begin{equation*}
\lmbar_p(t_1,t_2) = \lim_{n\ra\infty}\frac{1}{n}\log\E\left[\exp\left(n\,\langle (t_1,t_2), R^{(n,p)}\rangle\right)\right], \quad t_1\in \R, t_2 < \tfrac{1}{p}.
\end{equation*}
For $t_1\in \R$, $t_2<\frac{1}{p}$, there exist $\epsilon_1,\epsilon_2 > 0$ such that  $\lmbar_p(t_1(1+\epsilon_1),t_2(1+\epsilon_2)) < \infty$. Therefore, we can apply Varadhan's lemma --- see, e.g., Theorem 4.3.1 and condition (4.3.3) of \cite{DemZeibook}.
\begin{equation}\label{varadh1}
  \lmbar_p(t_1,t_2) = \sup_{\tau_1,\tau_2} \{t_1\tau_1+t_2\tau_2 - J_{R,p}(\tau_1,\tau_2)\}, \quad t_1\in \R, t_2 < \tfrac{1}{p}.
\end{equation}
We claim that the equality \eqref{varadh1} in fact holds for \emph{all} $t_1,t_2\in \R$. It remains to show that the right hand side is infinite for $t_2 \ge \frac{1}{p}$. Due to Cram\'er's theorem, the sequence $(R_2^{(n,p)})_{n\in\N}$, defined by
\begin{equation*}
R_2^{(n,p)} \doteq  \frac{1}{n}\sum_{i=1}^n |\ynp_i|^p,
\end{equation*}
satisfies an LDP with the good rate function $\hat\lm_p^*$,  where
\begin{align*}
\hat\lm_p^*(\tau_2) &\doteq \sup_{t_2\in \R} \{t_2\tau_2 - \hat\lm_p(t_2)\}, \quad \tau_2\in \R;\\
\hat\lm_p(t_2) &\doteq \log\int_\R e^{t_2|y|^p}\mu_p(dy), \quad t_2 \in \R.
\end{align*}
Due to the contraction principle and the continuity of the projection map $(\tau_1,\tau_2)\mapsto \tau_2$, we have
\begin{equation*}
\hat\lm_p^*(\tau_2) = \inf_{\tau_1} J_{R,p}(\tau_1,\tau_2).
\end{equation*}
Note that the infimum is attained at some $\tau_1^*\in\R$, because $J_{R,p}$ is lower semi-continuous (since it is a rate function), and has compact level sets (since it is a good rate function).  Then, we find that for all $t_1,t_2\in \R$,
\begin{align*}
\sup_{\tau_1,\tau_2}\{t_1\tau_1 + t_2\tau_2  -J_{R,p}(\tau_1,\tau_2)\} &\ge \sup_{\tau_2}\{t_2\tau_2 + t_1\tau_1^*-J_{R,p}(\tau_1^*,\tau_2) \} \\
   &= t_1\tau_1^*  + \sup_{\tau_2}\{t_2\tau_2 - \hat \lm_p^*(\tau_2)\}\\
   &= t_1\tau_1^* + \hat\lm_p(t_2).
\end{align*}
But from the definition of $\hat\lm_p$, it is clear that $\hat\lm_p(t_2) = \infty$ for $t_2 \ge \frac{1}{p}$. Thus, \eqref{varadh1} is true for \emph{all} $t_1,t_2\in \R$, so due to the convexity of $J_{R,p}$, Legendre duality (see, e.g., \cite[Lemma 4.5.8]{DemZeibook}) implies that $J_{R,p} = \lmbar_p^*$.

Applying the contraction principle, \eqref{mimap}, and the definition \eqref{ieqdefn} of $\iq_{p,\nu}$, we write the annealed rate function as
\begin{align*}
\ia_{p}(w) &= \inf_{\substack{\tau_1 \in \R,\tau_2 \in \R \,:\\
\tau_1\tau_2^{-1/p}=w}} \lmbar_{p}^*(\tau_1,\tau_2) \\
   &= \inf_{\substack{\tau_1 \in \R,\tau_2 >0 \,:\\ \tau_1\tau_2^{-1/p}=w}}\inf_{\nu \in \mathcal{P}(\mathbb{R})}  \left\{ \lmq_{p,\nu}^*(\tau_1,\tau_2) + \mathbb{H}(\nu)\right\}\\
   &= \inf_{\nu \in \mathcal{P}(\mathbb{R})} \inf_{\substack{\tau_1 \in \R,\tau_2 >0 \,:\\ \tau_1\tau_2^{-1/p}=w}} \left\{   \lmq_{p,\nu}^*(\tau_1,\tau_2) + \mathbb{H}(\nu)\right\}\\
   &= \inf_{\nu \in \mathcal{P}(\mathbb{R})}  \left\{ \iq_{p,\nu}(w) + \mathbb{H}(\nu)\right\}.
\end{align*}
The definition of $\mathbb{H}$ in \eqref{hpdefn} allows the restriction of the variational problem to measures $\nu$ satisfying $m_2(\nu)\le 1$, which completes the proof.
\end{proof}

\begin{remark}
The essence of the proof of Theorem \ref{th-qldp} is a strong law of large numbers for $\tfrac{1}{n} \sum_{i=1}^n \lm_p(t_1\sqrt{n}\bthn_i,t_2)$. Similarly, the essence of the proof of Theorem \ref{th-compar} is a large deviation principle for $\tfrac{1}{n} \sum_{i=1}^n \lm_p(t_1\sqrt{n}\bthn_i,t_2)$. The non-triviality of establishing such an LDP is due to the fact that 
\begin{equation*}
[\lm_p(t_1\sqrt{n}\bthn_i,t_2)]_{i=1,\dots,n;\,n\in\mathbb{N}}
\end{equation*}
 is an infinite triangular array of \emph{dependent} random variables. Similar random structures have previously been analyzed, for example, in \cite{frank1985strong} for LLN and in \cite{eichelsbacher1998exponential, trashorras2002large} for LDP. Note that this triangular array does have rich structure. For example, each row is a finite exchangeable vector. In addition, the Maxwell-Poincar\'e-Borel Lemma (see, e.g., \cite[Lemma 1.2]{ledoux1996isoperimetry}) says that for fixed $k$, the random variables
\begin{equation*}
  (\sqrt{n}\bthn_1, \dots, \sqrt{n}\bthn_k)
\end{equation*}
are asymptotically independent as $n\ra\infty$. Nonetheless, none of the existing literature on  general triangular arrays with such structure appears to be immediately applicable in our setting, which is why we appealed to the empirical measure versions (i.e., for $L_{n,\Theta}$) of the LLN (in the proof of Theorem \ref{th-qldp}) and LDP (Proposition \ref{prop-sanovcone}). As a side note, observe that the corresponding CLT for $L_{n,\Theta}$ (a Donsker-type theorem) can be found in \cite{spruill2007asymptotic}.
\end{remark}

\section{Atypical directions of projection}\label{sec-atyp} 

The goal of this section is to prove Theorem \ref{th-atyp}. We state some preliminary results in Sect. \ref{ssec-prel}. Then, we address the quenched case in Sect.\ \ref{ssec-atypq} and the annealed case in Sect.\ \ref{ssec-atypa}.

\begin{remark}\label{rmk-shao}
For $\theta = \iota$ (that is, the sequence of directions $(1,1,\dots,1)\in \sphn$, $n\in \N$), note that $\twnpz_\iota$ corresponds to the following ``self-normalized sum", 
\begin{equation}\label{shaosum}
  \twnpz_\iota =  \frac{\frac{1}{n} \sum_{i=1}^n \ynp_i}{\left( \frac{1}{n} \sum_{i=1}^n |\ynp_i|^p\right)^{1/p} }.
\end{equation}
The quantity $\twnpz_\iota$ when $\ynp_i$ has a general law (not necessarily $\mu_p$) has been analyzed in \cite{bercu2002concentration, dembo1998self, jing2003self, Shao97}. In particular,  \cite{Shao97} establishes upper-tail large deviation asymptotics for $(\twnpz_\iota)_{n\in\N}$ even if the law of $\ynp_i$ does not satisfy any exponential moment conditions. In our setting where $\ynp_i\sim \mu_p$, it is natural to ask how the rate function of \cite{Shao97}  compares with our universal rate function $\iq_{p,\mu_2}$.  A  consequence of Theorem \ref{th-atyp} is that the large deviation rate function for the self-normalized sums $(\twnpz_\iota)_{n\in\N}$ is atypical, in the sense that it does not coincide with   $\iq_{p,\mu_2}$.
 \end{remark}

\subsection{Preliminary properties of the rate functions}\label{ssec-prel}

In this section, we establish some elementary properties of our various rate functions. 

\begin{lemma}\label{lem-domains} The domains of the rate functions $\iq_{p,\mu_2}$ of \eqref{ieqdefn}, $\lm_p^*$ the Legendre transform of the function $\lm_p$ of \eqref{lampdefn0}, and $\jc_p$ of \eqref{cramratedefn} satisfy the following:
\begin{enumerate}
\item  For $p\in [2,\infty)$,  $ D_{\iq_{p,\mu_2}}^\circ  \subset  (-1,1)$;
\item  For $p\in(1,\infty)$, $D_{\lm_p^*}^\circ  = \{(\tau_1,\tau_2)\in \R\times\R_+ : |\tau_1|^p < \tau_2\}$;
\item For $p\in (1,\infty)$, $D_{\jc_p}^\circ = (-1,1)$. 
\end{enumerate}
\end{lemma}
\begin{proof}
We first prove 1. By H\"older's inequality, for $x\in \R^n$ and $p > 2$,
\begin{align*}
  \sum_{i=1}^n x_i^2 \le \left(\sum_{i=1}^n |x_i|^p\right)^{2/p} \left(\sum_{i=1}^n 1 \right)^{(p-2)/p} \quad \Rightarrow \quad  \|x\|_{n,2} \le \|x\|_{n,p}\, n^{\tfrac{1}{2} - \tfrac{1}{p}}.
\end{align*}
Thus, $n^{1/p}\bnp \subseteq n^{1/2} \mathbb{B}_{n,2}$, and they intersect at ``corners" of the form $(\pm 1, \pm 1, \dots, \pm 1) \in \R^n$. As a consequence,
\begin{equation*}
  \sup \{|\langle x,y\rangle_n| : x\in n^{1/p} \mathbb{B}_{n,p}  , y \in n^{1/2} \mathbb{B}_{n,2} \} = n.
\end{equation*}
This shows that the supports of the laws of $\wnpbth$ of \eqref{def-bwnpbth} and $\wnpth$ of \eqref{def-wthetan} are both equal to $[-1,1]$, and hence,  $D_{\iq_{p,\mu_2}}^\circ \subset  (-1,1)$.

Now we prove 2. Note that $\lm_p^*$ is the Cram\'er rate function for the  sequence of sums of i.i.d.\ random variables $\frac{1}{n}\sum_{i=1}^n \eta_i$, where in our case $\eta_i = (Y_i,|Y_i|^p)\in\R^2$ for $Y_i \sim \mu_p$, $i\in \N$. A classical fact from large deviation theory says that the closure of the domain of the Legendre transform of the log mgf of a probability measure $\nu\in\mathcal{P}(\R^d)$ is equal to the closure of the convex hull of the support of $\nu$ \cite[Lemma 2.4]{de1985large}. In our setting, this says that 
\begin{equation*}
  \overline{D_{\lm_p^*}}  = \overline{\text{conv}\{(\tau_1,\tau_2)\in\R\times\R_+ : |\tau_1|^p = \tau_2 \} } = \{(\tau_1,\tau_2)\in\R\times\R_+ : |\tau_1|^p \le \tau_2 \}.
\end{equation*}
Since $D_{\lm_p^*}$ is convex, this implies 2.

Lastly, the fact that $\jc_p$ is obtained from $\lm_p^*$ via the contraction principle under the map $(\tau_1,\tau_2)\mapsto \tau_1\tau_2^{-1/p}$ shows that $D_{\jc_p}^\circ = (-1,1)$.
\end{proof}

The preceding lemma explains why in Theorem \ref{th-atyp}, we limit our results to the case $w\in (-1,1)$. In the following lemma, we show that the relevant variational problems achieve their optima within these domains.

\begin{lemma}\label{lem-attained} Fix $p\in (1,\infty)$.
\begin{enumerate}
\item Let $(\tau_1,\tau_2)\in D_{\lm_p^*}$. Then, in the variational problem (i.e., the Legendre transform) that defines $\lm_p^*$,
\begin{equation}\label{lmtr}
  \lm_p^*(\tau_1,\tau_2) \doteq \sup_{t_1 \in \R, \,t_2<1/p} \left\{t_1\tau_1 + t_2\tau_2 - \lm_p(t_1,t_2) \right\},
\end{equation}
the supremum is uniquely attained.

\item Let $w\in D_{\iq_{p,\mu_2}}$. Then, in the variational problem of \eqref{ieqdefn} that defines $\iq_{p,\mu_2}$, the infimum is (not necessarily uniquely) attained.
\end{enumerate}
\end{lemma}

\begin{proof}\quad
\begin{enumerate}
\item A classical result in convex analysis is that the supremum defining the Legendre transform of a strictly convex function is uniquely attained  at a single point (see, e.g., Theorem 23.5 and Theorem 26.3 of \cite{rockafellar1970convex}). Since Lemma \ref{lem-lamfacts} states that $\lm_p$ is strictly convex, the supremum in \eqref{lmtr} is uniquely attained.

\item For $\tau > 0$ and $w\in (-1,1)$, let $g_w(\tau) \doteq \lmq_{p,\mu_2}^*(w\tau^{1/p},\tau)$. Note that for $w\in(-1,1)$, we can rewrite \eqref{ieqdefn} as 
\begin{equation*}
\iq_{p,\mu_2}(w) = \inf_{\tau_2 > 0} g_w(\tau_2).
\end{equation*}
Note that $g_w$ is lower semi-continuous due to the lower semi-continuity of $\lmq_{p,\mu_2}^*$ and the continuity of the map $\tau\mapsto (w\tau^{1/p},\tau)$. To show that the infimum is attained, it suffices to show boundedness of lower level sets, which implies compactness since $g_w$ is lower semi-continuous. Note that for all $\tau > 0$, the function $\lmq_{p,\mu_2}^*(\cdot,\tau)$ is symmetric about 0 and convex, and hence minimized at 0. As a consequence, for all $w\in (-1,1)$, 
\begin{align*}
g_w(\tau) \ge \lmq_{p,\mu_2}^*(0,\tau) &= \sup_{t_1\in \R, t_2<\frac{1}{p}}\left\{t_2\tau - \lmq_{p,\mu_2}(t_1,t_2)\right\}\\
   &\ge \sup_{t_2<\frac{1}{p}} \left\{ t_2\tau - \lmq_{p,\mu_2}(0,t_2)\right\} \\
   &= \sup_{t_2 < \frac{1}{p}} \left\{ t_2 \tau + \tfrac{1}{p}\log(1-pt_2)\right\}\\ 
   &= \tfrac{1}{p} (\tau + 1 - \log \tau),
\end{align*}
where the equality in the third line follows from Lemma \ref{lem-tailcont} and the fact that $\lmq_{p,\mu_2}(0,\cdot) = \lm_p(\cdot)$ by the definition of  $\lmq_{p,\mu_2}$ in \eqref{lampdefn}. Since $\lim_{\tau\rightarrow\infty}(\tau + 1 - \log \tau) = \infty$, we find that $\lim_{\tau\ra \infty} g_w(\tau) = \infty$, so $g_w$ has bounded level sets.
\end{enumerate}

\end{proof}

\subsection{Comparison of quenched and unweighted LDPs}\label{ssec-atypq}

In this section, we present the proof of Theorem \ref{th-atyp}, which entails a comparison of the log mgfs for the quenched and ``Cram\'er"-type LDPs. We begin by setting some notation that will allow us to state two lemmas that identify conditions under which  a log mgf is ``more" or ``less" convex than $t\mapsto t^2$.

Let $\beta >0$, and let $\mu_{p,\beta}(dy)$ be the absolutely continuous probability measure on $\R$ with density $f_{p,\beta}(y) \doteq C_{p,\beta} f_p(y /\beta^{1/p})$ for the appropriate normalization constant $C_{p,\beta}$. Note that $\mu_p =\mu_{p,1}$ and elementary calculations show that $C_{p,\beta} = C_{p,1} \beta^{-1/p}$. For $p\in[1,\infty)$, we have $\mu_{p,\beta}(dy) =  C_{p,\beta} e^{-|y|^p / (p\beta)} dy$, and for $p=\infty$, we have $\mu_{\infty,\beta}(dy)=\mu_\infty(dy)=\1_{[-1,1]}(y)\,dy$.

\begin{lemma}[Theorem 8 of \cite{barthe2003extremal}]\label{lem-mommon}
The map $\R_+\ni t \mapsto \log \mathrm{M}_{\mu_{p,1/p}}(\sqrt{t})$ is concave for $p\in[2,\infty]$ and convex for $p\in[1,2]$.
\end{lemma}

We can mold this lemma to apply to the function $\lm_p$ of \eqref{lampdefn0}.

\begin{lemma}\label{lem-strict}
Let $p\in [1,\infty)$ and $t_2 < \frac{1}{p}$. The map  $\R_+ \ni t_1 \mapsto \lm_p(\sqrt{t_1},\, t_2)$ is concave but not linear for $p>2$, linear for $p=2$, and convex but not linear for $p < 2$. 
\end{lemma}
\begin{proof}
It is easy to see that $\mathrm{M}_{\mu_{p,\beta}}(t) = \mathrm{M}_{\mu_{p,1/p}}(t\beta^{1/p})$. Together with   Lemma \ref{lem-mommon}, this implies that for all $\beta > 0$, the map $t\mapsto \log\mathrm{M}_{\mu_{p,\beta}}(\sqrt{t})$ is concave for $p\in[2,\infty]$ and convex for $p\in[1,2]$. For $t_1 \in \R$, $t_2< \frac{1}{p}$, we consider the case $\beta = (1-pt_2)^{-1}$ and apply Lemma \ref{lem-tailcont} to see that 
\begin{equation}\label{lmpexpand}
\lm_p(t_1,t_2) = -\tfrac{1}{p}\log(1-pt_2) + \log \mathrm{M}_{\mu_{p,(1-pt_2)^{-1}}}(t_1).
\end{equation}
This proves the concavity (resp., convexity) of $t_1 \mapsto \lm_p (\sqrt{t_1}, t_2)$ for $p \ge 2$ (resp., $p \le 2$).

It remains to show that linearity holds if and only if $p=2$. Note that for all $\beta > 0$,  $\mu_{2,\beta}$ is a Gaussian measure with mean 0 and variance $\beta$; thus, for $t \in \R_+$, we have 
\begin{equation*}
\log\mathrm{M}_{\mu_{2,\beta}}(\sqrt{t}) = \tfrac{\beta}{\sqrt{2}}  t.
\end{equation*}
Conversely, if $t_1\mapsto \lm_p(\sqrt{t_1},t_2)$ is linear, then \eqref{lmpexpand} implies that $t_1\mapsto \log\mathrm{M}_{\mu_{p,(1-pt_2)^{-1}}}(t_1)$ is quadratic, so $\mu_{p,(1-pt_2)^{-1}}$ must be Gaussian, hence $p=2$.
\end{proof}

We apply the concavity and convexity of the preceding lemma to prove inequalities for the function $\lmq_{p,\nu}$ defined in \eqref{lampdefn}.

\begin{lemma}\label{lem-mgfineq}
Let  $\nu\in\mathcal{P}(\R)$ be non-degenerate (i.e., not a Dirac mass at a single point).  If $p\in (2,\infty)$ and $m_2(\nu)\le 1$, then
\begin{equation*}
  \lmq_{p,\nu}(t_1,t_2) \le \lm_p(t_1,t_2), \quad t_1\in \R, \,t_2 < \tfrac{1}{p},
\end{equation*}
with equality if and only if $t_1 =0$. If $p\in (1,2)$ and $m_2(\nu) \ge 1$, then
\begin{equation*}
    \lmq_{p,\nu}(t_1,t_2) \ge \lm_p(t_1,t_2), \quad t_1\in \R, \,t_2 < \tfrac{1}{p},
\end{equation*}
with equality if and only if $t_1 =0$.
\end{lemma}

\begin{proof}
Fix  $p\in (2,\infty)$ and non-degenerate $\nu\in\mathcal{P}(\R)$ such that $m_2(\nu) \le 1$. Let $X\sim \nu$. Due to the concavity of $t_1\mapsto \lm_p(\sqrt{t_1},t_2)$ from Lemma \ref{lem-strict}, Jensen's inequality, and the fact that $t_1\mapsto \lm_p(t_1,t_2)$ is symmetric and increasing for $t_1>0$,
\begin{equation*}
\lmq_{p,\nu}(t_1,t_2) = \E_\nu\left[\lm_p((t_1^2X^2)^{1/2},\,t_2)\right] \le \lm_p\left(\left(\E_\nu[t_1^2X^2]\right)^{1/2}, \,t_2\right)  \le \lm_p(t_1,t_2). 
\end{equation*}
Since $t_1\mapsto \lm_p(\sqrt{t_1},t_2)$ is not linear for $p >2$, it follows that the first inequality above is an equality if and only if $t_1 = 0$ (i.e., when the random variable $t_1^2X^2$ is degenerate). The result for $p\in (1,2)$ and $m_2(\nu) \ge 1$ follows from similar calculations.
\end{proof}

\begin{remark}
The primary argument in the preceding proof of Lemma \ref{lem-mgfineq} is the concavity (or convexity) of $\Lambda \circ \sqrt{\cdot}$ and Jensen's inequality. A similar combination of tools was employed in \cite{barthe2003extremal}, but to a different end; in particular, on pages 2, 16, and 19 of \cite{barthe2003extremal},  Jensen's inequality is applied to a log-concave function $f$ and a vector $v\in \R^n$ to obtain the inequality
\begin{equation*}
\prod_{i=1}^n  f(v_i)^{1/n} \le f\left(\frac{1}{n}\sum_{i=1}^n v_i \right).
\end{equation*}
In that setting, this yields an upper bound on the volume of a slab orthogonal to any $\theta^{(n)} \in \sphn$ --- an upper bound that is attained by the slab orthogonal to $\iota^{(n)}$. On the other hand, we use Jensen's inequality in a slightly different way (with respect to  a general measure instead of a discrete measure) to show that the precise rate function for projections onto $\sigma$-a.e.\ $\theta\in\seq$ differs (i.e., $<$ rather than just $\le$) from the rate function for projections onto $\iota$.
\end{remark}

\begin{proof}[Proof of Theorem \ref{th-atyp}]
As observed in Remark \ref{rmk-shao}, as a consequence of the representation in Lemma \ref{lem-jointrep} and Lemma \ref{lem-unifequiv}, $(\sw_\iota^{(n,p)})_{n\in\N}$ satisfies the same LDP as the sequence $(\widetilde{\sw}_\iota^{(n,p)})_{n\in\N}$ of self-normalized sums defined in \eqref{shaosum}. The LDP for $(\widetilde{\sw}_\iota^{(n,p)})_{n\in\N}$ with rate function $\jc_p$ of \eqref{cramratedefn} follows from Cram\'er's theorem in $\R^2$ for    $\frac{1}{n}\sum_{i=1}^n (\ynp_i, |\ynp_i|^p)$, $n\in \N$, and the contraction principle applied to the map $\bar T_p$ of \eqref{bartp}.

As for comparing  the quenched and self-normalized rate functions, let $p\in(2,\infty)$, $(\tau_1,\tau_2)\in D_{\lm_p^*}$, and define 
\begin{equation*}
(  t_1^{\tau_1,\tau_2},t_2^{\tau_1,\tau_2}) \doteq \arg\max_{t_1\in \R, t_2<\tfrac{1}{p}} \left\{ t_1\tau_1 + t_2\tau_2 - \lm_{p}(t_1,t_2) \right\},
\end{equation*}
where the supremum is uniquely attained due to Lemma \ref{lem-attained}. Then, it follows from  the definition of the Legendre transform and Lemma \ref{lem-mgfineq} that 
\begin{align}
\lmq_{p,\mu_2}^*(\tau_1,\tau_2) &\ge t_1^{\tau_1,\tau_2}\tau_1 + t_2^{\tau_1,\tau_2}\tau_2 - \lmq_{p,\mu_2}( t_1^{\tau_1,\tau_2}, t_2^{\tau_1,\tau_2})\notag  \\
 &\stackrel{(\star)}{\ge}  t_1^{\tau_1,\tau_2}\tau_1 + t_2^{\tau_1,\tau_2}\tau_2 - \lm_{p}( t_1^{\tau_1,\tau_2}, t_2^{\tau_1,\tau_2}) \label{rateineqptwise} \\
  &= \lm_p^*(\tau_1,\tau_2).\notag
\end{align}
Note that Lemma \ref{lem-mgfineq} shows that the inequality $(\star)$ is an equality if and only if $t_1^{\tau_1,\tau_2}=0$. Due to the strict convexity of $\lm_p$, we have $( t_1^{\tau_1,\tau_2},t_2^{\tau_1,\tau_2}) = \nabla \lm_p^*(\tau_1,\tau_2)$. Since $\lm_p$ is essentially smooth (resp., symmetric in its first argument), $\lm_p^*$ is strictly convex (resp., symmetric in its first argument). Therefore,  $t_1^{\tau_1,\tau_2}=\partial_{\tau_1} \lm_p^*(\tau_1,\tau_2) = 0$  if and only if $\tau_1=0$.

Recall from Lemma \ref{lem-domains} that $D_{\iq_{p,\mu_2}}^\circ  \subset (-1,1) = D_{\jc_p}^\circ$. For $w\in (-1,1) \setminus D_{\iq_{p,\mu_2}}$, we have $\jc_p(w) < \iq_{p,\mu_2}(w) = \infty$. For $w\in D_{\iq_{p,\mu_2}}$, let 
\begin{equation*}
 (\tau_1^w,\tau_2^w) \in \arg \min_{\substack{\tau_1\in \R, \tau_2 > 0 : \\ \tau_1\tau_2^{-1/p}=w}}\lmq_{p,\mu_2}^*(\tau_1,\tau_2),
\end{equation*}
where a minimizer exists due to Lemma \ref{lem-attained}(2). Then, it follows from \eqref{rateineqptwise} and the definition of $\jc_p$ from \eqref{cramratedefn}, that
\begin{equation*}
\infty >   \iq_{p,\mu_2}(w) =  \lmq_{p,\mu_2}^*(\tau_1^w,\tau_2^w) \stackrel{(\ddagger)}{\ge} \lm_p^*(\tau_1^w,\tau_2^w)  \ge  \inf_{\substack{\tau_1\in \R, \tau_2 > 0 : \\ \tau_1\tau_2^{-1/p}=w}} \lm_p^*(\tau_1,\tau_2) =  \jc_p(w).
\end{equation*}
The assumption that  $w\in D_{\iq_{p,\mu_2}}$   implies that  $(\tau_1^w,\tau_2^w)\in D_{\lm_p^*}$. Thus, the inequality $(\ddagger)$ is strict if and only if the corresponding inequality \eqref{rateineqptwise} is strict, which is the case if and only if $\tau_1^w \ne 0$.  If $w\ne 0$, then the constraint $w=\tau_1\tau_2^{-1/p}$ implies that $\tau_1^w\ne 0$, so $(\ddagger)$ is  a strict inequality. On the other hand, if $w=0$, then $\iq_{p,\mu_2}(0) = 0 = \jc_p(0)$. This completes the proof for $p >2$.

The proof is essentially identical for $p<2$, with convexity replacing concavity. The identification in the case $p=2$ follows from Theorem \ref{th-p2}, which states that the rate function associated with $(\sw_\theta^{(n,2)})_{n\in \N}$ is the same for all $\theta \in \seq$, in particular for $\theta^{(n)} = \iota^{(n)} = \tfrac{1}{\sqrt{n}}(1,1,\dots,1)$.
\end{proof}

\subsection{Comparison of annealed and unweighted LDPs}\label{ssec-atypa}

Using similar methods as for Theorem \ref{th-atyp}, combined with the limit  log mgf $\lmbar_p$  of \eqref{pressfunc}, we obtain the following result which compares the sequence of fixed directions $\iota$ with the sequence of random directions $\Theta$.

\begin{proposition}
For $p\in (2,\infty)$,
\begin{equation*}
  \ia_p(w) \ge \jc_p(w), \quad w\in (-1,1),
\end{equation*}
with equality if and only if $w =0$.
\end{proposition}
\begin{proof}
Recall the definition of the limit log mgf $\lmbar_p$ given in \eqref{pressfunc}. Due to the variational representation stated in Lemma \ref{lem-lmgfvar} and Lemma \ref{lem-optprops}, there exists an optimal probability measure $\nu_p^\circ$ such that
\begin{equation}
  \lmbar_p(t_1,t_2) = \lmq_{p,\nu_p^\circ}(t_1,t_2) - \mathbb{H}(\nu_p^\circ). \label{nuorep}
\end{equation}
Note that $m_2(\nu_p^\circ)\le 1$, so by Lemma \ref{lem-mgfineq},
\begin{equation*}
\lmq_{p,\nu_p^\circ}(t_1,t_2) \le \lm_p(t_1,t_2), \quad t_1\in \R, t_2 < \tfrac{1}{p},
\end{equation*}
with equality if and only if $t_1 = 0$. Together with \eqref{nuorep}, this shows that 
\begin{equation*}
  \lmbar_p(t_1,t_2)  \le \lm_p(t_1,t_2) - \mathbb{H}(\nu_p^\circ) \le \lm_p(t_1,t_2),
\end{equation*}
with equality only if $t_1=0$ and $\nu_{p}^\circ = \mu_2$. From this inequality, the same considerations as in the proof of Theorem \ref{th-atyp}  --- except with $\lmq_{p,\mu_2}$ there  replaced by $\lmbar_p$  here --- complete the proof. 
\end{proof}

\section{Analogous results for product measures}\label{sec-prod} 

In this section, we consider projections of a random vector distributed according to a product measure, and state the analogous ``product measure" versions of the ``$\ell^p$ ball" results of Sect. \ref{sec-main}. For $n\in\N$ and $\gamma \in \mathcal{P}(\R)$, let 
\begin{equation}
\xng = (\xng_1,\dots,\xng_n) \sim \gamma^{\otimes n},
\end{equation}
independent of  $\bthn$.  Let $\mu_\infty$ be the uniform measure on $[-1,1]$, whose density is the limit $f_\infty = \lim_{p\ra\infty} f_p$ of the densities $f_p$ defined in \eqref{fpdef}. The $p=\infty$ analogs of our results stated in Sect.\ \ref{sec-main} follow as a consequence of the results in this section with $\gamma=\mu_\infty$.

The results in the product measure case are proved using very similar arguments as in the proofs of the $\ell^p$ ball case for $p<\infty$, given in Sect.\ \ref{sec-annealed}--\ref{sec-atyp}. In fact, the arguments in this section are typically slightly simpler, because the \emph{a priori} independence of the coordinates of $X^{(n,\gamma)}$ eliminates the need to appeal to the representation of the uniform measure on the $\ell^p$ ball given in Sect.\ \ref{sec-equiv}. For this reason, we will mostly only sketch the proofs in this section, highlighting only the main differences.

\subsection{Annealed LDP}

For $n\in \N$ and $\gamma\in \mathcal{P}(\R)$, let 
\begin{align}
  \sw^{(n,\gamma)} &\doteq \tfrac{1}{n^{1/2}} \langle \xng , \bthn\rangle_n,\label{wgth}\\
\lma_\gamma(t_0,t_1) &\doteq \log  \int_\R \int_\R e^{t_0 z^2 + t_1  zx} \mu_2(dz) \gamma(dx), \quad  t_0, t_1\in \R, \label{phigamdef}\\
  \ia_{\gamma}(w) &\doteq \inf_{\substack{\tau_0 > 0, \tau_1\in \R \, :\\ \tau_0^{-1/2}\tau_1 = w}} \lma_{\gamma}^*(\tau_0,\tau_1). 
\end{align}

\begin{theorem}\label{th-aldpprod}
Let $\gamma$ lie in the space  $\mathcal{T}_2$ defined in  \eqref{tpdef}. Then, the sequence $(W^{(n,\gamma)})_{n\in \N}$ satisfies an LDP with the quasiconvex good rate function $\ia_\gamma$.
\end{theorem}

\begin{proof}
 Let $Z^{(n)}=(\zn_1,\dots,\zn_n)$ be a standard Gaussian random vector (i.e., distributed according to $\mu_2^{\otimes n}$), independent of $\xng$. Define
\begin{equation}\label{wgz}
  \widetilde{W}^{(n,\gamma)} \doteq \tfrac{1}{n^{1/2}} \left\langle \xng , \tfrac{Z^{(n)}}{\ltwonrm{Z^{(n)}}}\right\rangle_n,
\end{equation}
and consider the associated  sum of i.i.d.\ $\R^2$-valued random variables,
\begin{equation*}
S^{(n,\gamma)} \doteq \frac{1}{n}\sum_{i=1}^n \left( |\zn_i|^2, \xng_i\, \zn_i \right).
\end{equation*}
Note that $\lma_{\gamma}$ of \eqref{phigamdef} is the log mgf of the summands of $S^{(n,\gamma)}$. 

Since $\Theta^{(n)} \eqdist Z^{(n)}/\|Z^{(n)}\|_{n,2}$ as shown in Lemma \ref{lem-jointrep}, we have $W^{(n,\gamma)}\eqdist \widetilde W^{(n,\gamma)}$, so it suffices to prove an LDP for $(\widetilde{W}^{(n,\gamma)})_{n\in\N}$.  Note that $\widetilde W^{(n,\gamma)} = T(S^{(n,\gamma)})$, where $T:\R^2\ra\R$ is defined by
\begin{equation}\label{tdef}
  T(\tau_0,\tau_1) = \tau_0^{-1/2}\tau_1.
\end{equation}
It is straightforward to check that if $\gamma\in\mathcal{T}_2$, then $0\in D_{\lma_\gamma}^\circ$ (see the proof of Theorem \ref{th-aldp} in Sect. \ref{ssec-anng2} for a related calculation), so by Cram\'er's theorem, the sequence $(S^{(n,\gamma)})_{n\in \N}$ satisfies an LDP in $\R^2$ with the good rate function  $\lma_{\gamma}^*$. Due to the continuity of $T$ on $D_{\lma_\gamma^*}$,  the contraction principle yields the LDP for $(\widetilde{W}^{(n,\gamma)})_{n\in\N}$ with the desired rate function $\ia_\gamma$.
\end{proof}

\subsection{Quenched LDP and atypical projection directions}\label{ssec-queprod}

Recall the mgf $\mathrm{M}_\gamma$ of \eqref{logmgfdef}. For $n\in \N$ and $\gamma,\nu \in \mathcal{P}(\R)$, define
\begin{align}
   \sw_\theta^{(n,\gamma)} &\doteq \tfrac{1}{n^{1/2}} \langle \xng , \thn\rangle_n, \label{wgth2}\\
\lmq_{\gamma,\nu}( t_1) &\doteq \int_\mathbb{R}  \log\mathrm{M}_\gamma(t_1u) \nu(du), \quad t_1 \in \mathbb{R}, \label{laminfdefn}\\
  \iq_{\gamma,\nu}(w) &\doteq \lmq_{\gamma,\nu}^*(w). \label{iqgamdefn}
\end{align}

\begin{theorem}\label{th-qldpprod}
Let $\gamma \in \mathcal{T}_q$  for some $q > 1$. Then, for $\sigma$-a.e.\ $\theta \in \seq$, the sequence $(W_\theta^{(n,\gamma)})_{n\in \N}$ satisfies  an LDP with the convex good rate function $\iq_{\gamma,\mu_2}$.
\end{theorem}

A version of Theorem \ref{th-qldpprod} with weaker conditions can be found in \cite[Theorem 2.4]{gkr3}. The reader can also find in \cite[Theorem 2.5]{gkr3} a comparison of $\iq_{\gamma,\mu_2}$ and $(\log \mathrm{M}_\gamma)^*$; the latter is the large deviation rate function for the sequence of empirical means of $X^{(n,\gamma)}$, as given by Cram\'er's theorem.

\subsection{Variational formula}

\begin{theorem}\label{th-comparprod}
Let $\gamma \in \mathcal{T}_q$ for some $q>2$. Then, for all $w\in \R$,
\begin{align}
  \ia_{\gamma}(w) &= \inf_{\substack{\nu \in \mathcal{P}(\R):\\ m_2(\nu)\le 1}} \left\{\iq_{\gamma,\nu}(w) + H(\nu |\mu_2) + \tfrac{1}{2}\left(1- m_2(\nu)\right)  \right\}. \label{varform2}
\end{align}
In particular, this implies that for all $w\in \R$, $\ia_{\gamma}(w) \le \iq_{\gamma,\mu_2}(w)$.
\end{theorem}

To prove Theorem \ref{th-comparprod}, we first establish appropriate versions of the lemmas established in Sect. \ref{sec-rel}. Let $\gamma \in \mathcal{T}_q$ for some $q> 2$, and define the functional $\lmbar_\gamma:\R \ra \R$ as follows: \begin{equation} \label{pressfunc2}
  \lmbar_\gamma(t) \doteq   \lim_{n\ra\infty} \frac{1}{n} \log \mathbb{E} \left[\exp\left( \sum_{i=1}^n  t \sqrt{n}\bthn_i \xng_i  \right)\right] , \quad t\in \R.
\end{equation}

\begin{lemma}\label{lem-lmbar}
Let $\gamma \in \mathcal{T}_p$. Then, 
\begin{equation} \label{mgfvar2}
   \lmbar_{\gamma}(t) = \sup_{\nu \in \mathcal{P}(\mathbb{R})} \left\{ \lmq_{\gamma,\nu}(t) - \mathbb{H}(\nu)\right\}, \quad t\in \R.
\end{equation}
In addition, $\lmbar_\gamma(t) < \infty$ for all $t\in \R$.
\end{lemma}
\begin{proof}[Sketch of Proof]
The proof of Lemma \ref{mgfvar2} centers around Varadhan's lemma, and follows from  similar calculations as in the proof of Lemma \ref{lem-lmgfvar}, except with $\log \mathrm{M}_\gamma$ in place of $\lm_p$.
\end{proof}

\begin{lemma}\label{lem-optprops2}
Let $\gamma \in \mathcal{T}_p$, and for fixed $t\in \R$, let $\phi:\mathcal{P}(\R)\ra \R$ denote the functional being maximized in \eqref{mgfvar2},
\begin{equation*}
  \phi(\nu) \doteq \lmq_{\gamma,\nu}(t)- \mathbb{H}(\nu).
\end{equation*}
Then, $\phi$ is strictly concave and upper semi-continuous (with respect to the Wasserstein-$\tfrac{p}{p-1}$ topology on $\mathcal{P}_{p/(p-1)}(\R)$). As a consequence, the supremum in \eqref{mgfvar2} is uniquely attained at some optimal $\nu^\circ$ such that $m_2(\nu^\circ) \le 1$.
\end{lemma}
\begin{proof}[Sketch of Proof]
The proof is essentially identical to the proof of Lemma \ref{lem-optprops}, except the continuity of $\nu\mapsto \lmq_{\gamma,\nu}(t)$ is given by Lemma \ref{lem-subgsn} instead of Lemma \ref{lem-tailcont}.
\end{proof}

\begin{lemma} \label{lem-varforbarstar2}
Let $\gamma \in \mathcal{T}_p$. Then, for $\tau\in \R$,
\begin{equation}\label{mimag}
   \lmbar_{\gamma}^*(\tau) = \inf_{\nu \in \mathcal{P}(\mathbb{R})}  \left\{ \lmq_{\gamma,\nu}^*(\tau) + \mathbb{H}(\nu)\right\}.
\end{equation}
\end{lemma}
\begin{proof}[Sketch of Proof]
The proof of Lemma \ref{lem-varforbarstar2} is similar to the proof of Lemma \ref{lem-varforbarstar}, where the main task is to verify the conditions of the Minimax Theorem (Theorem \ref{th-minmax}), in order to apply it to the variational formula \eqref{mgfvar2}. The main differences in this case are: we set $\mathcal{Y} = \mathbb{R}$, $D =  \mathbb{R}$, and for fixed $\tau$, the functional $F$ is set equal to   $F(\nu, t) \doteq t\tau -  \lmq_{\gamma,\nu}(t) + \mathbb{H}(\nu)$ for   $\nu \in C$ and $t \in D$. We omit the details.
\end{proof}

\begin{proof}[Proof of Theorem \ref{th-comparprod}]
A straightforward modification of the proof of Proposition \ref{prop-altldpconvex} shows that $(W^{(n,\gamma)})_{n\in\N}$ satisfies an LDP with a convex good rate function. Also note that the domain of the limit log mgf $\lmbar_\gamma$ is all of  $\R$. We utilize the following fact for large deviations in a topological vector space $\mathcal{X}$: if a given rate function is convex in $\mathcal{X}$, and the domain of the associated limit log mgf is the entire dual space $\mathcal{X}^*$, then the rate function can be identified with the Legendre transform of the limit log mgf (see, e.g., \cite[p.152, Theorem 4.5.10]{DemZeibook}). Therefore, the rate function for  $(W^{(n,\gamma)})_{n\in\N}$  is  $\lmbar_\gamma^*$, the Legendre transform of the limit  log mgf $\lmbar_\gamma$ defined in \eqref{pressfunc2}. This observation and the variational formula \eqref{mimag} complete the proof.  
\end{proof}

\section{Analysis of the variational problem}\label{sec-analysis}

In this section, we analyze the variational problems that relate the annealed and quenched rate functions. In Sect. \ref{ssec-infvar}, we analyze the variational problem of Theorem \ref{th-comparprod} for $\gamma=\mu_\infty$. In Sect. \ref{ssec-conj}, we formulate some conjectures for the variational problem of Theorem \ref{th-compar}.

\subsection{Comparison of quenched and annealed rate functions for $p=\infty$}\label{ssec-infvar}

Note that for $w=0$ and $p\in[2,\infty)$, the infimum  in the variational problem \eqref{varform1} is attained at $\mu_2$. Roughly speaking, this occurs because $w=0$ is the (LLN) limit of the random projection $\wnpbth$, and the Gaussian measure $\mu_2$ is the (LLN) limit of the empirical measure defined in \eqref{empir},
\begin{equation*}
 L_{n,\Theta} = \frac{1}{n}\sum_{i=1}^n \delta_{\sqrt{n}\bthn_i}  \Rightarrow \mu_2, \quad \text{ as } n\ra \infty.
\end{equation*}
For general $w\ne 0$, the minimizer (assuming it exists) may not necessarily be the Gaussian measure. 

For $p=2$, Lemma \ref{lem-var2} states that the infimum is attained at $\mu_2$ for \emph{all} $w\in \R$. This is because  the spherical symmetry of the uniform law on $\mathbb{B}_{n,2}$ is such that a projection onto a random direction has the same law as a projection onto a fixed direction (say, the canonical first coordinate $e_1^{(n)}$). In other words, large deviations of the random directions of projection play no role in the annealed large deviations, when $p=2$. 

In contrast, as clarified in Proposition \ref{prop-nongsn} below, the random directions of projection do play a role when the random vector to be projected is drawn according to the uniform measure on $[-1,1]^n$ instead of the uniform measure on $\mathbb{B}_{n,2}$.  That is, the unique minimizer of \eqref{varform2} is \emph{not} $\mu_2$, which  suggests the that deviations of the underlying ``environment" (the directions of projection) play a non-trivial role in the overall annealed large deviations.

\begin{lemma}\label{lem-unique}
For $\gamma \in \mathcal{T}_2$ and $w\in \R$ such that $\ia_\gamma(w) < \infty$, there exists a unique minimizer $\nu_{\gamma,w} \in \mathcal{P}(\R)$ that attains the infimum in   \eqref{varform2}.
\end{lemma}
\begin{proof}
The idea is similar to Lemma \ref{lem-optprops2}, which considers the related variational problem \eqref{mgfvar2}. Let $r\in (1,2)$, and equip $\mathcal{P}_r(\R)$ with the Wasserstein-$r$ topology. By Lemma \ref{lem-compact},  it follows that the infimum in \eqref{varform2} is over a convex, compact set.  In addition, $\nu\mapsto\iq_{\gamma,\nu}(w)$ is convex and lower semi-continuous, since it is  the supremum of the maps $\nu\mapsto tw - \lmq_{\gamma,\nu}(t)$, which are continuous due to Lemma \ref{lem-subgsn}, and also clearly linear by definition. Moreover, $\mathbb{H}$ is lower semi-continuous and strictly convex due to Proposition \ref{prop-sanovcone}. Thus, the infimum in \eqref{varform2} is the infimum of a lower semi-continuous strictly convex function over a compact convex set, so the infimum is uniquely attained.
\end{proof}

\textsc{Notation.} Fix the following notational convention for the remainder of this section: replace $\mu_\infty$ by $\infty$ in our notation for the mgfs and rate functions (i.e., write $\mathrm{M}_\infty$, $\lmq_{\infty,\nu}$, $\lma_\infty$, $\iq_{\infty, \nu}$, $\ia_\infty$), as well as in our notation for the optimizing measure of Lemma \ref{lem-unique}, replace $\nu_{\mu_\infty,w}$ with $\nu_{\infty,w}$.

\begin{proposition}\label{prop-nongsn}
Let $p=\infty$. There exists $w_*\in (0,1)$ such that if $w_* \le |w| < 1$, then $\nu_{\infty,w} \ne \mu_2$; that is, for some $w\in (-1,1)$, the minimizer in \eqref{varform2} is not standard Gaussian. This implies that for such $w$,  the following strict inequality holds: $\ia_{\infty}(w) < \iq_{\infty,\mu_2}(w)$.
\end{proposition}

To prove this, we begin by analyzing the asymptotics of the function $\lmq_{\infty,\nu}$ defined in \eqref{laminfdefn}.

\begin{lemma}\label{lem-asyms} 
Let $\mathrm{M}_\infty$ be the log mgf of $\mu_\infty$, as defined in \eqref{logmgfdef}, and let $m_1(\cdot)$ be the first moment map, as defined in \eqref{qmomdef}. Then,
\begin{equation}\label{minf}
\lim_{|t|\ra\infty}  \frac{\log \mathrm{M}_\infty(t)}{|t|}  = 1.
\end{equation}
For $\nu\in\mathcal{P}(\R)$,
\begin{equation}\label{psinf}
\lim_{|t|\ra\infty}  \frac{\lmq_{\infty,\nu}(t)}{|t|}  = m_1(\nu).
\end{equation}
In addition, $\lmq_{\infty,\nu}$ is strictly convex. As a consequence, we have
\begin{equation} \label{dpsinf}
(-m_1(\nu), +m_1(\nu)) \subset   D_{\lmq_{\infty,\nu}^*} \subset [-m_1(\nu), +m_1(\nu)].
\end{equation}
\end{lemma}
\begin{proof}
The limit \eqref{minf} follows from basic calculus. That is, applying the symmetry of $\log\mathrm{M}_\infty$, using the explicit expression $\log\mathrm{M}_\infty(t) = \log \left(\frac{\sinh t}{t}\right)$, and applying L'H\^opital's rule to compute the limit,
\begin{align*}
 \lim_{|t|\ra \infty} \frac{\log\mathrm{M}_\infty(t)}{|t|} &= \lim_{t\ra\infty} \frac{\log\mathrm{M}_\infty(t)}{t} = \lim_{t\ra\infty} (\log\mathrm{M}_\infty)'(t) = \lim_{t\ra\infty} \left(\coth t - \tfrac{1}{t}\right) = 1.
\end{align*}
As for the second limit \eqref{psinf}, by the monotone convergence theorem, for $\nu \in \mathcal{P}(\R)$,
\begin{equation*}
 \lim_{|t|\ra \infty} \frac{\lmq_{\infty,\nu}(t)}{|t|} =  \lim_{t\ra\infty} \int_{\R}|u|\left( \coth(tu) - \tfrac{1}{tu}\right)\nu(du) = m_1(\nu).
\end{equation*}

Note that  $\log\mathrm{M}_\infty$ is strictly convex due to basic properties of log mgfs, and therefore $\lmq_{\infty,\nu}$ is also strictly convex for all $\nu\in\mathcal{P}(\R)$, since integration with respect to $\nu$ is a linear functional.

We now prove the first inclusion of \eqref{dpsinf}.  The strict convexity of $\lmq_{\infty,\nu}$ and the asymptotic linearity given by \eqref{psinf} imply that for all $c < m_1(\nu)$, there exists some $t_c\in\R$ such that $\lmq_{\infty,\nu}(t) > c|t|$ for $|t| \ge t_c$. The upshot is that if $\epsilon > 0$ and $|w| < m_1(\nu) -\epsilon$, then
\begin{equation*}
\limsup_{|t|\ra \infty} \left[  tw - \lmq_{\infty,\nu}(t)\right] \le \limsup_{|t|\ra \infty}    |t|\,(|w| - m_1(\nu) + \epsilon)  = -\infty.
\end{equation*}
Hence, the map $F_{w,\nu}$ defined by $F_{w,\nu}(t)\doteq tw - \lmq_{\infty,\nu}(t)$ has compact upper level sets. Since $F_{w,\nu}$ is upper semi-continuous (due to the lower semi-continuity of $\lmq_{\infty,\nu}$), it follows that $F_{w,\nu}$ is bounded above in $\R$, implying that $\lmq_{\infty,\nu}^*(w) < \infty$ when $|w| < m_1(\nu) - \epsilon$. As this holds for all $\epsilon > 0$, we have that $(-m_1(\nu),+m_1(\nu)) \subset D_{\lmq_{\infty,\nu}^*}$.

To prove the the second inclusion of \eqref{dpsinf}, a similar argument as above shows that for $\epsilon > 0$, if $w > m_1(\nu) + \epsilon$, then
\begin{equation*}
\liminf_{t\ra+\infty} \left[  tw - \lmq_{\infty,\nu}(t)\right] \ge \liminf_{t\ra+\infty} t(w - m_1(\nu) - \epsilon) =+\infty,
\end{equation*}
and if $w < -(m_1(\nu) + \epsilon)$, then 
\begin{equation*}
\liminf_{t\ra-\infty} \left[  tw - \lmq_{\infty,\nu}(t)\right] \ge \liminf_{t\ra-\infty} t(w + m_1(\nu) + \epsilon) = +\infty.
\end{equation*}
Therefore, $\lmq_{\infty,\nu}^*(w) = \infty$ for $|w| > m_1(\nu) + \epsilon$. Because this holds for all $\epsilon > 0$, it follows that $D_{\lmq_{\infty,\nu}^*} \subset [-m_1(\nu),+m_1(\nu)]$.  
\end{proof}

\begin{remark} 
Note that $m_1(\mu_2) = \sqrt{2/\pi} \approx 0.798$, which lies on the boundary of the domain of $\iq_{\infty,\mu_2}$, as depicted in Figure \ref{fig-comp}.
\end{remark}

\begin{proof}[Proof of Proposition \ref{prop-nongsn}]
To show that the minimizer of the variational problem \eqref{varform2} is not $\mu_2$, it suffices to show that there exists \emph{some} measure $\nu_\circ  \in \mathcal{P}(\R)$ such that:
\begin{enumerate}[label=(\alph*)]
\item $\nu_\circ$ is absolutely continuous with respect to Lebesgue measure;
\item $m_2(\nu_\circ)\le 1$;
\item $m_1(\nu_\circ) > m_1(\mu_2)$.
\end{enumerate}
There exist several such measures, but for a concrete example, consider the uniform measure on $[-\sqrt{3},\sqrt{3}]$. Given any measure $\nu_\circ$ satisfying (a), (b), and (c), it follows that $H(\nu_\circ | \mu_2) < \infty$, and the definition of $\iq_{\gamma,\nu}$ in \eqref{iqgamdefn} and Lemma \ref{lem-asyms} imply that for $w\in(-1,1)$ such that $m_1(\mu_2) < |w| < m_1(\nu_\circ)$, we have
\begin{equation*}
\iq_{\infty,\nu_\circ} = \lmq_{\infty,\nu_\circ}^*(w) < \infty = \lmq_{\infty,\mu_2}^*(w) = \iq_{\infty,\mu_2}.
\end{equation*}
Therefore, the functional $\nu\mapsto \iq_{\infty,\nu}(w) + H(\nu|\mu_2) + \frac{1}{2}(1-m_2(\nu))$ is finite when $\nu=\nu_0$ but infinite when $\nu=\mu_2$, which proves the proposition.
\end{proof}

\subsection{Conjectures regarding the variational problem}\label{ssec-conj}

We believe that Proposition \ref{prop-nongsn} can be extended to all $w\ne 0$ in the domain of $\ia_\infty$, and that an analogous result should hold for all $p\in(2,\infty)$ as well as for products of measures other than $\gamma = \mu_\infty$. To be precise, we mean:

\begin{conjecture}\label{conj-var}
Let $p\in(2,\infty)$. For $w\in(-1,1)\setminus \{0\}$, the minimizer in \eqref{varform1} is not $\mu_2$. Similarly, for $\gamma \in \mathcal{T}_p$  and $w\in D_{\ia_{\gamma}} \setminus \{0\}$, the minimizer in \eqref{varform2} is not $\mu_2$. This implies that except at $w=0$, the annealed rate function lies strictly below the quenched rate function.
\end{conjecture}

This would require a new approach since: (i) our current proof relies on the exact asymptotics of Lemma \ref{lem-asyms} for the case $p=\infty$, which makes generalization to other product measures difficult; and (ii) for general $\ell^p$ balls, the variational problem is more complicated, due to the additional contraction step.

One possible approach to Conjecture \ref{conj-var} would be to analyze the intermediate variational problems \eqref{mgfvar} and \eqref{mgfvar2}. In the case $p=\infty$, it is possible to establish the following lemma:

\begin{lemma}\label{lem-intoptnon}
Let $F_t(\nu) \doteq \lmq_{\infty,\nu}(t) - \mathbb{H}(\nu)$. There exists $t_* >0$ such that if $|t| \ge t_*$, then the maximizer of \eqref{mgfvar2} is not the standard Gaussian. That is, for some probability measure $\nu_t\ne \mu_2$, we have $F_t(\nu_t) > F_t(\mu_2)$.
\end{lemma}

\begin{proof}[Sketch of Proof]
First, we can rewrite $F_t$ and \eqref{mgfvar2} in terms of the \emph{entropy} of $\nu$. Then, Lemma \ref{lem-asyms} can be applied to transform \eqref{mgfvar2} into a penalized maximum entropy problem, amenable to exact calculations.
\end{proof}

The main issues with this approach are that the claim is for $t$ sufficiently large, and the ``optimal" measure is not identified. Nonetheless, this approach offers an alternative variational problem which may be simpler to analyze than \eqref{varform2}.

\clearpage

\printnomenclature[1.5cm]
\label{notation}


\bibliography{ldpbib}
\bibliographystyle{imsart-number}

\addtocontents{toc}{\protect\setcounter{tocdepth}{-1}}

\end{document}